\documentclass[10pt]{amsart}
\usepackage[margin=0.8in]{geometry}

\usepackage{amssymb}
\usepackage{amsthm}
\usepackage{amsmath}
\usepackage{mathrsfs}
\usepackage{amsbsy}
\usepackage{mathdots}
\usepackage[all]{xy}
\usepackage{bm}
\usepackage{hyperref}
\usepackage{tikz}
\usepackage{array}
\usepackage{float}
\usepackage{enumerate}
\usepackage{xcolor}
\usepackage{hhline}
\allowdisplaybreaks
\usepackage[noadjust]{cite}
\usepackage{hyperref}
\usepackage{cleveref}
\usepackage{enumerate}
\usepackage[shortlabels]{enumitem}
\hypersetup{colorlinks=true,citecolor=blue,linkcolor=blue}

\usepackage{tikz}
\usetikzlibrary{calc}
\tikzset{vtx/.style={circle, fill, inner sep=1.5pt}}
\tikzset{vtxbig/.style={circle, draw, inner sep=7pt, green}}
\tikzset{openvtx/.style={circle, draw, inner sep=1.5pt}}

\usetikzlibrary{decorations.pathreplacing,calligraphy}

\usepackage{multicol}

\crefname{theorem}{}{Theorem}
\crefname{section}{}{}
\crefname{corollary}{}{}
\crefname{definition}{}{Definiton}
\crefname{lemma}{}{Lemmas}
\crefname{claim}{}{}
\crefname{equation}{}{}
\crefname{proposition}{}{}
\crefname{figure}{}{}
\crefname{example}{}{}
\crefname{remark}{}{}

\newenvironment{enumerate*}
  {\begin{enumerate}[(I)]
    \setlength{\itemsep}{10pt}
    \setlength{\parskip}{0pt}}
  {\end{enumerate}}

\newtheorem{theorem}{Theorem}[section]
\newtheorem{proposition}[theorem]{Proposition}
\newtheorem{corollary}[theorem]{Corollary}
\newtheorem{conjecture}[theorem]{Conjecture}
\newtheorem{question}[theorem]{Question}

\newtheorem{lemma}[theorem]{Lemma}

\theoremstyle{definition}
\newtheorem{definition}[theorem]{Definition}
\newtheorem{example}[theorem]{Example}

\theoremstyle{remark}
\newtheorem{remark}[theorem]{Remark}

\newcommand{\dfn}[1]{\textcolor{blue}{\emph{#1}}}
\newcommand{\eps}{\varepsilon}
\newcommand{\Pb}{\mathbb{P}}

\newcommand{\Z}{\mathbb{Z}}

\renewcommand{\L}{\mathcal{L}}
\newcommand{\G}{\mathcal{G}}
\newcommand{\x}{\mathcal{X}}
\newcommand{\y}{\mathcal{Y}}
\newcommand{\z}{\mathcal{Z}}
\renewcommand{\r}{\mathcal{R}}
\newcommand{\s}{\mathcal{S}}

\newcommand{\alphasim}{{[\alpha]_\sim}}
\newcommand{\alphasimeq}{{[\alpha]_\simeq}}

\newcommand{\taui}{\tau^{-1}}
\newcommand{\sigmai}{\sigma^{-1}}
\newcommand{\rhoi}{\rho^{-1}}

\newcommand{\betat}{\beta^{(t)}}
\newcommand{\sigmat}{\sigma^{(t)}}
\newcommand{\varphit}{\varphi_q}

\renewcommand{\S}{\mathsf{Star}}
\renewcommand{\P}{\mathsf{Path}}
\newcommand{\C}{\mathsf{Cycle}}

\newcommand{\GR}{{G^{(\rho)}}}
\newcommand{\HR}{{H^{(\rho)}}}

\newcommand{\es}{\varnothing}
\newcommand{\ind}{\mathbf{1}}
\newcommand{\id}{{\rm Id}}
\renewcommand{\l}{\ell}

\renewcommand{\t}[1]{\widetilde{#1}}
\renewcommand{\o}[1]{\overline{#1}}
\newcommand{\tr}[2]{(#1\text{ }#2)}

\newcommand{\Sym}{\mathfrak{S}}

\newcommand{\substak}[2]{$\begin{array}{@{}c@{}}#1\\#2\end{array}$}

\newcommand{\Acyc}{\operatorname{Acyc}}
\newcommand{\sgn}{\operatorname{sgn}}

\DeclareMathOperator{\FS}{\mathsf{FS}}
\DeclareMathOperator{\FSm}{\mathsf{FS_m}}
\DeclareMathOperator{\FSmm}{\mathsf{FS_{m,m}}}

\begin{document}

\title[]{Connectivity of old and new models of friends-and-strangers graphs}

\author[]{Aleksa Milojevi\'c}
\address[]{Department of Mathematics, Princeton University, Princeton, NJ 08540, USA}
\email{aleksam@princeton.edu}

\maketitle

\begin{abstract}
In this paper, we investigate the connectivity of friends-and-strangers graphs, which were introduced by Defant and Kravitz in 2020. We begin by considering friends-and-strangers graphs arising from two random graphs and consider the threshold probability at which such graphs attain maximal connectivity. We slightly improve the lower bounds on the threshold probabilities, thus disproving two conjectures of Alon, Defant and Kravitz. We also improve the upper bound on the threshold probability in the case of random bipartite graphs, and obtain a tight bound up to a factor of $n^{o(1)}$. Further, we introduce a generalization of the notion of friends-and-strangers graphs in which vertices of the starting graphs are allowed to have multiplicities and obtain generalizations of previous results of Wilson and of Defant and Kravitz in this new setting.
\end{abstract}

\section{Introduction}\label{sec:intro}

\noindent
The main objects of study in this paper are the friends-and-strangers graphs, which were introduced by Defant and Kravitz in 2020 \cite{DK}.

\begin{definition}
Given two simple graphs $X$ and $Y$ on $n$ vertices, we define the \dfn{friends-and-strangers graph} $\FS(X, Y)$ associated to $X$ and $Y$ as follows. Vertices of $\FS(X, Y)$ are all bijections $\sigma:V(X)\to V(Y)$, and the edges are given by pairs of bijections $\sigma$ and $\tau$ which satisfy $\sigma=\tau\circ \tr{a}{b}$, for some $a, b\in V(X)$ with $ab\in E(X)$ and $\sigma(a)\sigma(b)\in E(Y)$. In this case, we say that $\sigma$ and $\tau$ differ by a \dfn{$(X, Y)$-friendly swap across the edge $\sigma(a)\sigma(b)$} (or, equivalently, across the edge $ab$).
\end{definition}
\noindent
The graph $\FS(X, Y)$ has a natural interpretation, from which the name \textit{friends-and-strangers} originated. Namely, if we think of vertices of $X$ as $n$ different people and vertices of $Y$ as $n$ chairs, the bijections $\sigma:V(X)\to V(Y)$ represent arrangements of people onto chairs, with exactly one person sitting on each chair. We say that a pair of people are friends if they are adjacent in $X$, and that they are strangers otherwise. Then, a $(X, Y)$-friendly swap in $\FS(X, Y)$ corresponds to choosing two friends sitting on adjacent chairs and swapping their positions. It is natural to ask whether any arrangement can be obtained from any other by a sequence of friendly swaps, which corresponds to the graph $\FS(X, Y)$ being connected.

The concept of friends-and-strangers graphs generalizes several previously considered problems, as noted in \cite{ADK} and \cite{DK}. For example, in the famous 15-puzzle numbers $1, \dots, 15$ are placed on a $4\times 4$ grid, leaving one cell empty. Any number adjacent to the empty cell is allowed to move to the empty cell, and the goal is to reach a predetermined configuration of numbers. In the context of friends-and-strangers graph, this puzzle corresponds to $\FS(\S_{16}, \mathsf{Grid}_{4\times 4})$, the numbers $1, \dots, 15$ corresponding to leaves of $\S_{16}$ and the empty cell corresponding to its center. Generalizing the 15-puzzle, in 1974 Wilson \cite{W} characterized all graphs $X$ for which $\FS(\S_n, X)$ is connected. Further, Stanley investigated the connected components of $\FS(\P_n, \P_n)$ in \cite{S}, while Reidys \cite{R} used the graph $\FS(\P_n, \o{X})$ in order to investigate the acyclic orientations of $X$.

Defant and Kravitz \cite{DK} recognized that all of these results can be phrased in a unified framework, thus defining the friends-and-strangers graphs. After proving some of the basic properties of these graphs, such as symmetry $\FS(X, Y)\cong \FS(Y, X)$ and the fact that $\FS(X, Y)$ is bipartite, they went on to describe the connected components of $\FS(X, \P_n)$ and $\FS(X, \C_n)$ in terms of acyclic orientations of the complement $\o{X}$. As a consequence of the structural description of $\FS(X, \C_n)$, they showed that $\FS(X, \C_n)$ is connected if and only if the complement of $X$ is a forest of trees of coprime sizes (Corollary 4.14 of \cite{DK}). Building on their work, Jeong \cite{J1} showed that, for such graphs $X$, the friends-and-strangers graph $\FS(X, Y)$ is still connected even if $\C_n$ is replaced by any other biconnected graph $Y$. In a follow-up paper, Jeong \cite{J2} investigates the diameter of the connected components of the graph $\FS(X, Y)$. Finally, Defant, Dong, Lee, and Wei \cite{DDLW} extended the connectivity results to new classes of graphs, such as spiders and dandelions. Although most of these results address structural questions, other aspects of friends-and-strangers graphs are also worth investigating.

In a follow-up paper Alon, Defant, and Kravitz \cite{ADK} initiated the study of extremal questions about friends-and-strangers graphs. One of the main questions addressed in \cite{ADK} asks what is the smallest integer $d_n$ such that for all graphs $X, Y$ whose minimum degrees satisfy $\delta(X), \delta(Y)\geq d_n$, the friends-and-strangers graph $\FS(X, Y)$ is connected. After Alon, Defant, and Kravitz \cite{ADK} initially showed $d_n\geq \frac{3}{5}n-1$, and conjectured $d_n=\frac{3}{5}n+O(1)$, Bangachev \cite{B} settled this conjecture by showing that it is sufficient to have $\delta(X), \delta(Y)> n/2$ and $2\min\{\delta(X), \delta(Y)\}+3\max\{\delta(X), \delta(Y)\}\geq 3n$ for $\FS(X, Y)$ to be connected.

Further, the authors of \cite{ADK} also introduced several probabilistic questions about the connectivity of $\FS(X, Y)$ and the first half of our paper is dedicated to answering them. More precisely, if $X, Y$ are random graphs, chosen according to the Erd\H{o}s-R\'enyi model $\G(n, p)$, Alon, Defant and Kravitz investigated the threshold probability $p_{\rm gen}$ at which $\FS(X, Y)$ becomes connected with high probability. In their paper \cite{ADK}, they showed the inequalities $\Omega(n^{-1/2})\leq p_{\rm gen}\leq n^{-1/2+o(1)}$ and conjectured $p_{\rm gen}=\Theta(n^{-1/2})$. In this paper, we disprove this conjecture by showing that $p_{\rm gen}\geq \Omega({\log^{1/2} n/n^{1/2}})$, using a slightly tweaked version of the original argument from \cite{ADK}. Asymmetric variant of this question has also been investigated by Wang and Chen \cite{WC}, who considered the connectivity of $\FS(X, Y)$ when $X\sim \G(n, p_1)$ and $Y\sim \G(n, p_2)$, for different probabilities $p_1, p_2$. 

Finally, Alon, Defant, and Kravitz studied the case when $X, Y\sim \G(K_{n, n}, p)$ are random bipartite graphs, generated as edge-subgraphs of the balanced bipartite graph $K_{n, n}$ in which every edge is included with probability $p$, randomly and independently. When $X, Y$ are bipartite graphs, a parity obstruction prevents $\FS(X, Y)$ from being connected and hence it is natural to ask for the threshold probability $p_{\rm bip}$ at which the graph $\FS(X, Y)$ has exactly two components with high probability. In \cite{ADK}, the authors showed that $\Omega(n^{-1/2})\leq p_{\rm bip}\leq n^{-0.3+o(1)}$ and conjectured $p_{\rm bip}=\Theta(n^{-1/2})$, as in the non-bipartite case. We tighten the upper bound by showing that $p_{\rm bip}\leq n^{-1/2+o(1)}$ using a modified version of the approach from the non-bipartite case, and we tighten the lower bound to $p_{\rm bip}\geq \Omega(\log^{1/2}n/n^{1/2})$. These results close the polynomial gap between the upper and the lower bound on the threshold probability in the bipartite case, bringing it down to a factor of $n^{o(1)}$.

In Section \cref{sec:prelim} we present general properties of the $\FS(X, Y)$, some of which reproduced from \cite{ADK}, \cite{DK}. Alongside these, in Subsection~\cref{subsec:obstructionstoconnectivity} we prove the following strengthening of Propositions 3.1 and 4.1 from \cite{ADK}.

\begin{theorem}\label{thm:lowerboundstrengthening}
There exists a constant $\eps>0$ with the following property. For a large positive integer $n$ and \[p<\eps \left(\frac{\log n}{n}\right)^{1/2},\] if we choose random graphs $X, Y\sim \G(n, p)$ independently, the resulting friends-and-strangers graph $\FS(X, Y)$ has an isolated vertex with high probability. The same statement holds if we choose $X, Y$ to be random bipartite graphs $X, Y\sim \G(K_{n, n}, p)$.
\end{theorem}

\noindent
As previously mentioned, Theorem~\cref{thm:lowerboundstrengthening} disproves Conjectures 7.1 and 7.2 from \cite{ADK}. Our argument is very similar to the argument given in \cite{ADK}, and the improvement comes from replacing the application of the general Sauer-Spencer theorem by a stronger result of Bollob\'as, Janson and Scott \cite{BJS} which applies only in the case of random graphs (which is all we need). Then, in Section \cref{sec:random}, we show the following strengthening of Theorem 1.2 from \cite{ADK}.

\begin{theorem}\label{thm:randombipartiteconnectivity}
Let $n$ be a positive integer, and let $X, Y$ be random bipartite graphs independently chosen from $\G(K_{n, n}, p)$. If 
\begin{equation}\label{eqn:pbound}
    p\geq \frac{\exp\left(10(\log n)^{4/5}\right)}{n^{1/2}},
\end{equation}
then $\FS(X, Y)$ has two connected components with high probability.
\end{theorem}

\begin{remark}
If $X$ and $Y$ are bipartite graphs, Proposition 2.5 of \cite{ADK} shows that $\FS(X, Y)$ has at least two connected components, due to a parity obstruction. Thus, Theorem~\cref{thm:randombipartiteconnectivity} shows that if $p$ satisfies \cref{eqn:pbound}, then the number of connected components of $\FS(X, Y)$ is the least possible. Furthermore, the parity obstruction implies that, if $\FS(X, Y)$ has exactly two connected components, they must have the same size. Thus, if $\FS(X, Y)$ contains an isolated vertex, in the regime of Theorem~\cref{thm:lowerboundstrengthening}, it must have more than two connected components.
\end{remark}

\noindent
Together with Proposition 4.1 of \cite{ADK}, our result determines that $p\sim n^{-1/2}$ is the correct exponent for the threshold probability in the case of $X, Y\sim \G(K_{n, n}, p)$, which agrees with the threshold probability for $\FS(X, Y)$ to be connected when $X, Y\sim \G(n, p)$. Although the proof of Theorem~\cref{thm:randombipartiteconnectivity} follows the general strategy used by Alon, Defant, and Kravitz in the non-bipartite setting, several new ideas are needed, including an additional result about matchings in random bipartite graphs.

\medskip
In Sections~\cref{sec:wilson} and \cref{sec:pathscycles}, we investigate a generalization of the friends-and-strangers graphs obtained by adding multiplicities onto vertices. To understand and motivate this generalization, it is useful to adopt an alternative view of friends-and-strangers graphs. We will think of bijections $\sigma:V(X)\to V(Y)$ as assignments of labels $\sigma(v)\in V(Y)$ to vertices $v\in V(X)$. We will first present a motivating example based on the previous work of Kornhauser, Miller, and Spirakis \cite{KMS} and then give a formal definition. 

Following the early work of Wilson on understanding the graph $\FS(X, \S_n)$, Kornhauser, Miller, and Spirakis \cite{KMS} considered the setting in which $k$ different pebbles are placed on the graph with $n$ vertices, and a pebble is allowed to move to any empty vertex adjacent to its current position. In this setting, the result of Wilson \cite{W} can be interpreted as the case $n=k+1$, when only one vertex is left empty. This scenario can be interpreted in the language of friends-and-strangers graphs by setting $X$ to be the underlying graph on which the pebbles move, and setting the label graph $Y$ to be the graph consisting of $k$ pebbles connected to one blank vertex, in which the blank label is allowed to appear $n-k$ times on $X$. When $n>k+1$, the problems of this type do not fit into the framework of friends-and-strangers graphs anymore, since the empty vertices are allowed to repeat. Motivated by this example, let us now formally introduce multiplicity friends-and-strangers graphs.

\begin{definition}
Let $X, Y$ be simple graphs, on $n$ and $m$ vertices, respectively, and let $m\leq n$. Fix a list of multiplicities $c=(c_1, \dots, c_m)\in \Z_{>0}^m$. We call simple graph $Y$ together with the list of multiplicities $c$ a \dfn{multiplicity graph}, with \dfn{total multiplicity} $\sum_{y\in V(Y)}c_y$. If the total multiplicity of $Y$ is $n$, we define a multiplicity friends-and-strangers graph $\FSm(X, Y)$ whose vertices are functions $\sigma:V(X)\to V(Y)$, called \dfn{arrangements}, satisfying $|\sigmai(\{y\})|=c_y$ for all $y\in V(Y)$. Further, arrangements $\sigma$ and $\tau$ are adjacent in $\FSm(X, Y)$ when $\sigma=\tau\circ \tr{a}{b}$, where $ab\in E(X)$ and $\sigma(a)\sigma(b)\in E(Y)$.
\end{definition}

\begin{remark}
If the multiplicities of all vertices of $Y$ are equal to $1$, we recover the standard friends-and-strangers graph $\FS(X, Y)$. Further, note that the definition of $\FSm(X, Y)$ is not symmetric in $X, Y$ anymore. 
\end{remark}

\noindent
Intuitively, the multiplicity $c_y$ of a label $y\in V(Y)$ can be understood as the number of times this label appears on $X$. Alternatively, in the interpretation of $\FS(X, Y)$ which involves friends sitting on chairs, the multiplicity $c_y$ can be understood as the number of people from $X$ sitting on a chair $y\in V(Y)$.

Multiplicity friends-and-strangers graphs were also considered in a more applied setting, by Brailovskaya, Gowri, Yu, and Winfree \cite{BGYW}. They proposed a computation model in which various types of molecules are placed on a surface modeled by a grid, and certain pairs of molecules are allowed to swap through a chemical reaction. It turns out that their model can be understood in the language of friends-and-strangers graphs, the molecules types corresponding to labels placed on the grid graph, with the edges between molecule types encoding which swaps are chemically allowed.

\medskip

In Section \cref{sec:wilson}, we focus on investigating connectivity of the graphs $\FS_m(X, \S_n)$ and $\FS_m(\S_n, X)$, generalizing previous work of Wilson. Before formally stating our results, let us present the precise statement of Wilson's theorem. 
\begin{theorem}[\cite{W}]
Let $X$ be a simple graph on $n\geq 3$ vertices. The graph $\FS(X, \S_n)$ is connected if and only if $X$ is a biconnected graph, which is not bipartite, not a cycle of length at least $4$ and not isomorphic to $\theta_0$, which a specific graph on $7$ vertices shown on Figure~\cref{fig:theta0}. Furthermore, if $X$ is bipartite and biconnected, but it is not a cycle of length at least $4$ or $\theta_0$, then the graph $\FS(X, \S_n)$ has exactly two connected components.
\end{theorem}

\noindent
If a graph $X$ satisfies all conditions of Wilson's theorem for $\FS(X, \S_n)$ to be connected, we say $X$ is \dfn{Wilsonian}.

\begin{figure}[h]
\begin{tabular}{c}
\begin{tikzpicture}
\coordinate[vtx] (v0) at (0:1);
\coordinate[vtx] (v1) at (60:1);
\coordinate[vtx] (v2) at (120:1);
\coordinate[vtx] (v3) at (180:1);
\coordinate[vtx] (v4) at (240:1);
\coordinate[vtx] (v5) at (300:1);
\coordinate[vtx] (v6) at (0:0);

\draw (v0)--(v1);
\draw (v1)--(v2);
\draw (v2)--(v3);
\draw (v3)--(v4);
\draw (v4)--(v5);
\draw (v5)--(v0);

\draw (v0)--(v6);
\draw (v6)--(v3);
\end{tikzpicture}
\end{tabular}

\caption{Diagram of the graph $\theta_0$. \phantom{aaaaaaaaaa}}
\label{fig:theta0}
\end{figure}
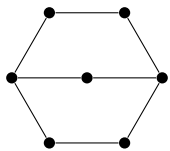

\noindent
Wilson's theorem can be generalized in two possible ways, depending if the multiplicities are assigned to the graph $X$ or to the star $\S_n$. Theorem~\cref{thm:WilsonmultiplicityX} characterizes the sets of multiplicities that can be assigned to vertices of $X$ such that $\FSm(\S_n, X)$ is still connected, while Theorem~\cref{thm:WilsonmultiplicityS} does the same in the case the multiplicities are assigned to the vertices of $\S_n$. Theorem \cref{thm:WilsonmultiplicityS} was originally derived by Kornhauser, Miller, and Spirakis in 1984 \cite{KMS}, using algebraic methods. Since some of the details of their proof are omitted, we present a rigorous, purely combinatorial proof of their results in the language of friends-and-strangers graphs. 

\begin{theorem}\label{thm:WilsonmultiplicityX}
Let $X$ be a connected multiplicity graph on the vertex set $[m]$ with multiplicities $c=(c_1, \dots, c_m)\in \Z_{>0}^m$. If $X$ is Wilsonian, the graph $\FSm(\S_n, X)$ is connected. If $X$ is biconnected but not Wilsonian, the graph $\FSm(\S_n, X)$ is connected if and only if there exists a vertex $v\in V(X)$ with multiplicity $c_v\geq 2$. Finally, if $X$ is not biconnected, the graph $\FSm(\S_n, X)$ is connected if and only if $c_v\geq 2$ for all cut vertices $v\in V(X)$.
\end{theorem}

\noindent
To state Theorem~\cref{thm:WilsonmultiplicityS}, we need to introduce the following notion. Given a simple graph $X$, a $k$-bridge in $X$ is a $k$-tuple of vertices $a_1, \dots, a_k\in V(X)$ for which the following two conditions hold: $a_{i+1}$ and $a_{i-1}$ are the only neighbors of $a_i$, for $i\in \{2, \dots, n-1\}$, and in $X|_{V(X)-\{a_2, \dots, a_{k-1}\}}$ the vertices $a_1$ and $a_k$ lie in different connected components, which both have size at least $2$.

\begin{figure}[h]
\begin{tabular}{c}
\begin{tikzpicture}
\coordinate[vtx] (a1) at (-2, 0);
\coordinate[vtx] (a2) at (-1, 0);

\coordinate[vtx] (ak-1) at (1, 0);
\coordinate[vtx] (ak) at (2, 0);

\draw (a1) node[below] {$a_1$};
\draw (a2) node[below] {$a_2$};
\draw (ak-1) node[below] {$a_{k-1}$};
\draw (ak) node[below] {$a_k$};
\draw (0,0) node {$\dots$};

\draw (a1) -- (a2);
\draw (a2) -- (-0.5, 0);
\draw (ak-1) -- (ak);
\draw (ak-1) -- (0.5, 0);
\draw (a1) -- (-3, 0.5);
\draw (a1) -- (-3, -0.5);
\draw (ak) -- (3, 0.5);
\draw (ak) -- (3, -0.5);

\draw (3, 0) ellipse (0.5cm and 0.7cm);
\draw (-3, 0) ellipse (0.5cm and 0.7cm);

\end{tikzpicture}
\end{tabular}

\caption{Diagram of a $k$-bridge consisting of vertices $a_1, \dots, a_k$. \phantom{aaaaaaaaaa}}
\label{fig:kbridge}
\end{figure}
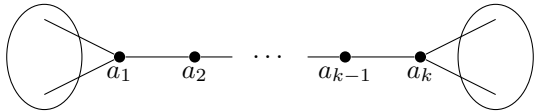

\begin{theorem}[\cite{KMS}] \label{thm:WilsonmultiplicityS}
Let $\S_m$ be a star on $m>2$ vertices and let $c\in \Z_{>0}^m$ be the list of multiplicities assigned to vertices of $\S_m$, with the center of $\S_m$ having multiplicity $k\geq 2$. Let $X$ be a connected simple graph on $n$ vertices, where $n$ is the total multiplicity of $\S_m$. If $X$ does not contain a $k$-bridge and if $X$ is not a cycle, the graph $\FSm(X, \S_m)$ is connected. If $X$ contains a $k$-bridge, $\FSm(X, \S_m)$ is disconnected. If $X$ is a cycle, the graph $\FSm(X, \S_m)$ is connected if and only if $m=3$ and one of the leaves of $\S_m$ has multiplicity $1$.
\end{theorem}

\noindent
In Section \cref{sec:pathscycles}, we describe the structure of  $\FSm(\P_n, X)$ and $\FSm(\C_n, X)$ when $X$ is a multiplicity graph. In \cite{DK}, Defant and Kravitz characterized the connected components of the graphs $\FS(\P_n, Y), \FS(\C_n, Y)$ in terms of the acyclic orientations of the complement $\o{Y}$, using the notion of \textit{toric equivalence}, which was initially introduced by Develin, Macauley, and Reiner in \cite{DMR}. Following this approach, we describe the connected components of $\FSm(\P_n, X), \FSm(\C_n, X)$ in terms of the acyclic orientations of $\o{X'}$, where $X'$ denotes the blow-up of $X$ with respect to the multiplicity list $c$. We then use this description to characterize all graphs $X$ and multiplicity lists $c$ for which $\FSm(\C_n, X)$ is connected. Since formal definitions and exact statements of these results are somewhat technical, we postpone them until Section \cref{sec:pathscycles}.

In Section \cref{sec:furtherresearch}, we end by presenting another generalization of friends-and-strangers graphs, which allows both graphs $X$ and $Y$ to have multiplicities, along with indicating possible further research directions.

\section{Preliminaries}\label{sec:prelim}

\noindent
In this section, we present some of the basic tools we will use throughout the paper. We begin by presenting two results used to show friends-and-strangers graphs are connected, and then show two typical obstructions to connectivity of these graphs. Together with these obstructions, we discuss the proof of Theorem~\cref{thm:lowerboundstrengthening} which shows that certain random friends-and-strangers graphs have isolated vertices. Finally, we will conclude by introducing the \textit{lift} of a multiplicity graph, a notion that proves to be key in understanding this generalization of friends-and-strangers graphs. 

\subsection{Connectivity of $\FS(X, Y)$.} 

One of the fundamental notions Alon, Defant, and Kravitz use in \cite{ADK} to show their results about connectivity of friends-and-strangers graphs are exchangeable pairs. More precisely, for a bijection $\sigma:V(X)\to V(Y)$, we say that a pair of vertices $u, v\in V(Y)$ is \dfn{$(X, Y)$-exchangeable from $\sigma$} if there exists a sequence of $(X, Y)$-friendly swaps which transforms the arrangement $\sigma$ into $\tr{u}{v}\circ \sigma$. Note that this sequence might consist of more than one swap. Since $\FS(X, Y)$ is symmetric in $X$ and $Y$, one might define the notion of exchangeability for pairs of vertices $u', v'\in V(X)$ in the symmetric fashion. The following result, stated as Proposition 2.8 in \cite{ADK}, shows how exchangeability can be used to establish connectivity of $\FS(X, Y)$.

\begin{proposition}[\cite{ADK}]\label{prop:exchangeabilityimpliesconnectedness}
Let $X, Y, \t Y$ be $n$-vertex graphs and suppose $Y\subseteq \t Y$. Suppose that for every edge $uv$ of $\t Y$ and every bijection $\sigma:V(X)\to V(Y)$ satisfying $\sigmai(u)\sigmai(v)\in E(X)$, the vertices $u$ and $v$ are $(X, Y)$-exchangeable from $\sigma$. Then, the connected components of $\FS(X, Y)$ and the connected components of $\FS(X, \t Y)$ have the same vertex sets. In particular, the number of connected components of $\FS(X, \t Y)$ is equal to the number of connected components of $\FS(X, Y)$.
\end{proposition}

\noindent
Note that Proposition~\cref{prop:exchangeabilityimpliesconnectedness} allows us to translate a global question about connectedness of $\FS(X, Y)$ to a local question about exchangeability of certain pairs of vertices, and therefore it will be very useful throughout the paper. Since the graph $\FS(X, K_n)$ is connected whenever $X$ is connected, the following special case of Proposition~\cref{prop:exchangeabilityimpliesconnectedness}, stated as Lemma 2.9 in \cite{ADK}, is particularly interesting. Setting $\t Y=K_n$ shows that, for a connected graph $X$, the graph $\FS(X, Y)$ is connected if every pair of adjacent vertices in $X$ is $(X, Y)$-exchangeable from any arrangement $\sigma$. 

In the setting of multiplicity graphs, the notion $(X, Y)$-exchangeability generalizes with almost the same definition. The only difference is that  exchangeability can only be defined for pairs of vertices of $X$. Formally, we say that a pair of vertices $u', v'\in V(X)$ is \dfn{$(X, Y)$-exchangeable from a bijection $\sigma$} if there exists a sequence of $(X, Y)$-friendly swaps transforming $\sigma$ into $\sigma\circ \tr{u'}{v'}$. It is not hard to see that the analogue of Proposition~\cref{prop:exchangeabilityimpliesconnectedness} holds even in the setting of multiplicity graphs.

The following useful result, first observed by Bangachev in \cite{B}, relates the notion of exchangeability from two different arrangements, $\sigma$ and $\sigma'$. 

\begin{proposition}[\cite{B}]\label{prop:exchangeabilityfromdifferentarrangements}
Suppose that $\sigma, \sigma'\in V(\FSm(X, Y))$ are in the same connected component of $\FSm(X, Y)$ and assume that the sequence of swaps transforming $\sigma$ into $\sigma'$ never involves vertices $u, v\in V(X)$. Then, if the pair of vertices $u, v$ is $(X, Y)$-exchangeable from $\sigma$, it is also $(X, Y)$-exchangeable from $\sigma'$.
\end{proposition}
\begin{proof}
Suppose that $\Sigma$ is the sequence of $(X, Y)$-friendly swaps which transform $\sigma$ into $\sigma'$, and let $\Sigma^{-1}$ be the reverse sequence which transforms $\sigma'$ into $\sigma$. Since none of the swaps in $\Sigma$ involve $u, v$, it is simple to see that $\Sigma$ also transforms $\sigma\circ \tr{u}{v}$ into $\sigma'\circ \tr{u}{v}$. Given that the pairs $(\sigma', \sigma)$, $(\sigma, \sigma\circ\tr u v)$ and $(\sigma\circ \tr{u}{v}, \sigma'\circ \tr{u}{v})$ are in the same connected component of $\FSm(X, Y)$, we conclude that the pair $u, v$ is $(X, Y)$-exchangeable from $\sigma'$.
\end{proof}

\subsection{Obstructions to connectivity of friends-and-strangers graphs.} \label{subsec:obstructionstoconnectivity} We will now present two typical obstructions to connectivity of $\FS(X, Y)$. The first one occurs when both $X$ and $Y$ have cut vertices, where a cut vertex is defined as vertex whose removal disconnects the graph. The obstruction presented in the following proposition is a direct generalization of Proposition 2.6 from \cite{DK} to the setting with multiplicities. Since the proof is completely analogous to the original, we omit it.

\begin{proposition}\label{prop:cutvertices}
Let $X$ be a connected graph on $m\geq 3$ vertices with a list of multiplicities $c\in \Z_{>0}^m$, which has total multiplicity $n$, and let $Y$ be a graph on $n$ vertices. Suppose $x_0\in V(X)$, $y_0\in 
V(Y)$ are cut vertices and that $c_{x_0}=1$. Also, let $X_1, \dots, X_r$ be the connected components of $X$ produced by removing $x_0$, and let $Y_1, \dots, Y_s$ be the connected components of $Y$ produced by removing $y_0$. Let $M$ be the set of $r\times s$ matrices with nonnegative integer entries in which the $i$-th row sums to $|X_i|$ and the $j$-th column sums to $|Y_j|$. Then, $\FSm(Y, X)$ has at least $|M|$ connected components.
\end{proposition}

\noindent
In case $Y$ has at least $3$ vertices, it is not hard to see that $|M|\geq 2$, and therefore Proposition~\cref{prop:cutvertices} implies that the graph $\FSm(Y, X)$ is disconnected.

Let us now present the second obstruction to connectivity which will be important in this paper. It was mentioned in the introduction that if $X, Y$ are bipartite graphs, the graph $\FS(X, Y)$ has at least two connected components, due to a parity obstruction. Moreover, if the graph $\FS(X, Y)$ has exactly two connected components, there is a simple way to test whether two arrangements are in the same connected component, which is described by the following proposition. This description will be particularly useful in Section \cref{sec:random}, where we consider random bipartite graphs $X$ and $Y$. This proposition was presented in \cite{ADK} as Proposition 2.5 and it is equivalent to Proposition 2.7 of \cite{DK}.

\begin{proposition}[\cite{ADK}, \cite{DK}] \label{prop:parityobstruction}
Let $X$ and $Y$ be bipartite graphs on $n$ vertices with vertex bipartitions $V(X)=A_X\bigsqcup B_X$ and $V(Y)=A_Y\bigsqcup B_Y$. If the bijections $\sigma$ and $\tau$ are in the same connected component of $\FS(X, Y)$, then $\sgn(\sigmai\circ\tau)$ has the same parity as $|\tau(A_X)\cap A_Y|-|\sigma(A_X)\cap A_Y|$.
\end{proposition}

\noindent
Here, the sign of the permutation $\pi$, denoted by $\sgn(\pi)$ denotes the parity of the number of inversions in the permutation $\pi$. Finally, we present the proof of Theorem~\cref{thm:lowerboundstrengthening} closely the approach from \cite{ADK}.

\begin{proof}[Proof of Theorem~\cref{thm:lowerboundstrengthening}.]

We begin by focusing on the non-bipartite case. Given two $n$-vertex graphs $X, Y$, an isolated vertex in the graph $\FS(X, Y)$ corresponds to a bijection $\sigma:V(X)\to V(Y)$ with the property that $\sigma(u)\sigma(v)\not\in E(Y)$ for all edges $uv\in E(X)$. In other words, an isolated vertex of $\FS(X, Y)$ corresponds to a packing of graphs $X, Y$. Theorem 1 from the paper of Bollob\'as, Janson and Scott \cite{BJS} gives a criterion for the existence of such packing, when $X$ and $Y$ are random graphs generated from $G(n, p)$. More precisely, plugging in $k=2$ and $p=q$ into this theorem gives the following statement: there exists a constant $\eps>0$ such that for $p^2\leq \eps \log n/n$ and independent random graphs $X, Y\sim G(n, p)$, there is a packing of $X, Y$ in $K_n$ with high probability. In other words, for $p<(\eps \log n/n)^{1/2}$ the friends-and-strangers graph $\FS(X, Y)$ contains an isolated vertex. 

The case of bipartite graphs is almost identical. Namely, random bipartite graphs on $2n$ vertices can be generated by taking $G(2n, p)$ and deleting all edges occurring among the first $n$ vertices and among the last $n$ vertices. As long as $p<(\eps \log 2n/2n)^{1/2}$, two independent graphs from $G(2n, p)$ can be packed with high probability, and hence the same holds after deleting the edges. This completes the proof of Theorem~\cref{thm:lowerboundstrengthening}.
\end{proof}

\begin{remark} \label{rmk:thresholdisolatedvertices}
Since the bounds of Theorem 1 in \cite{BJS} are tight, this proof actually determines exactly the threshold probability for the existence of isolated vertices in randomly generated friends-and-strangers graphs.
\end{remark}

\noindent
\subsection{Lift of a multiplicity graph.} Finally, the key tool for understanding the connectivity of multiplicity friends-and-strangers graphs is the lift of a multiplicity graph. 

\begin{definition}
Let $Y$ be a graph on $m$ vertices and let $c=(c_1, \dots, c_m)\in \Z_{>0}^m$ be its multiplicity list. We define the \dfn{lift} of $Y$ to be a simple graph $Y'$ on $\sum_{v\in V(Y)} c_v$ vertices which has a vertex partition of $Y'$, $V(Y')=\bigsqcup_{v\in V(Y)} S_v$, with the following properties:
\begin{itemize}
    \item $Y'|_{S_v}$ is a clique of size $c_v$ for all vertices $v\in V(Y)$, and
    \item vertices $u'\in S_u, v'\in S_v$ are adjacent in $Y'$ if and only if $u, v$ are adjacent in $Y$. 
\end{itemize}
\end{definition}

\noindent
In other words, the lift $Y'$ is a blow-up of the graph $Y$, in which the vertex $v$ is replaced with a clique of size $c_v$. Note that defining the lift of the graph also defines a natural projection map $\pi_Y$ from $Y'$ to $Y$ which maps all elements of $S_v$ to $v$.

To see why considering the lift of a graph is useful, let us describe an equivalence relation on the vertices of $\FS(X, Y')$, along with a way to describe the structure of $\FSm(X, Y)$ using $\FS(X, Y')$. Throughout this discussion, we assume a multiplicity graph $Y$ is given, with total capacity $n$, together with a simple graph $X$ on $n$ vertices.

For a finite set $A$, let $\Sym_A$ denote the group of all permutations $\rho:A\to A$, with the composition operation. Furthermore, for a multiplicity graph $Y$, let us define $\Sym_Y$ to be the group of permutations $\rho:V(Y')\to V(Y')$ which satisfy $\rho(S_v)=S_v$ for all cliques $S_v$. In other words, $\Sym_Y$ is the set of those permutations that permute the vertices of $Y'$ only within their cliques. Then, we define an equivalence relation $\equiv$ on the bijections from $V(X)$ to $V(Y')$ given by $\sigma_1\equiv \sigma_2$ if $\sigma_1=\rho\circ \sigma_2$ for some $\rho\in \Sym_Y$. Equivalently, one might say that $\sigma_1\equiv\sigma_2$ if $\pi_Y\circ \sigma_1=\pi_Y\circ \sigma_2$, as functions from $V(X)$ to $V(Y)$.

This equivalence relation induces a natural quotient operation on the graph $\FS(X, Y')$, which identifies all vertices equivalent under $\equiv$ and assigns an edge between distinct equivalence classes $[\sigma]_\equiv, [\tau]_\equiv$ precisely when there exist representatives $\sigma_0\in [\sigma]_\equiv, \tau_0\in [\tau]_\equiv$ which are adjacent in $\FS(X, Y')$. We denote the arising graph by $\FS(X, Y')/\!\equiv$

\begin{proposition}\label{prop:quotientisomorphism}
Let $X$ be a simple graph on $n$ vertices $Y$ be a multiplicity graph with total multiplicity $n$. Then, the quotient $\FS(X, Y')/\!\equiv$ is isomorphic to the graph $\FSm(X, Y)$.
\end{proposition}

\noindent
Before passing to the formal proof, let us illustrate this statement with an example.

\begin{example}
Suppose that $X=\P_3$, with vertices numbered 1, 2, 3, and $Y$ is a multiplicity graph on two adjacent vertices, with multiplicities $1$ and $2$. The lift $Y'$ is a triangle and the graphs $\FSm(X, Y)$ and $\FS(X, Y')$ are represented schematically in Figure~\cref{fig:ex1}.

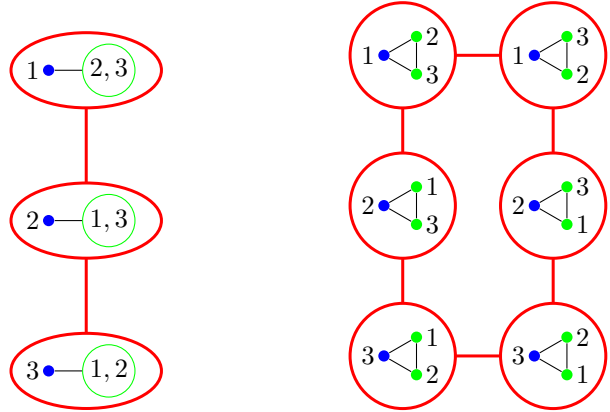
\begin{figure}[h]
\scalebox{1}{
\begin{tabular}{  c c  } 
\begin{tikzpicture}
\coordinate[vtx, blue] (u1) at (0, 2);
\coordinate[vtxbig] (v1) at (0.8, 2);
\draw (u1) node[left] {$1$};
\draw  (v1) node {$2,3$};
\draw (u1) -- (v1);
\draw[line width=0.4mm,red] (0.5,2) ellipse (1cm and 0.5cm);

\coordinate[vtx, blue] (u2) at (0, 0);
\coordinate[vtxbig] (v2) at (0.8, 0);
\draw (u2) node[left] {$2$};
\draw  (v2) node {$1,3$};
\draw (u2) -- (v2);
\draw[line width=0.4mm,red] (0.5, 0) ellipse (1cm and 0.5cm);

\coordinate[vtx, blue] (u3) at (0, -2);
\coordinate[vtxbig] (v3) at (0.8, -2);
\draw (u3) node[left] {$3$};
\draw (v3) node {$1,2$};
\draw (u3) -- (v3);
\draw[line width=0.4mm,red] (0.5, -2) ellipse (1cm and 0.5cm);

\draw[line width=0.4mm, red] (0.5, 1.5) -- (0.5, 0.5);
\draw[line width=0.4mm, red] (0.5, -1.5) -- (0.5, -0.5);

\end{tikzpicture}
&\hspace{2cm}
\begin{tikzpicture}

\coordinate[vtx, blue] (u1) at (0, 2);
\coordinate[vtx, green] (v1) at ($(u1) + (30:0.5)$);
\coordinate[vtx, green] (w1) at ($(u1) + (-30:0.5)$);
\draw (u1) node[left] {$1$};
\draw (v1) node[right] {$2$};
\draw (w1) node[right] {$3$};
\draw (u1) -- (v1) -- (w1) -- (u1);
\draw[line width=0.4mm,red] ($(u1)+(0:0.25)$) ellipse (0.7cm and 0.7cm);

\coordinate[vtx, blue] (u1) at (2, 2);
\coordinate[vtx, green] (v1) at ($(u1) + (30:0.5)$);
\coordinate[vtx, green] (w1) at ($(u1) + (-30:0.5)$);
\draw (u1) node[left] {$1$};
\draw  (v1) node[right] {$3$};
\draw  (w1) node[right] {$2$};
\draw (u1) -- (v1) -- (w1) -- (u1);
\draw[line width=0.4mm,red] ($(u1)+(0:0.25)$) ellipse (0.7cm and 0.7cm);

\coordinate[vtx, blue] (u1) at (0, 0);
\coordinate[vtx, green] (v1) at ($(u1) + (30:0.5)$);
\coordinate[vtx, green] (w1) at ($(u1) + (-30:0.5)$);
\draw (u1) node[left] {$2$};
\draw  (v1) node[right] {$1$};
\draw  (w1) node[right] {$3$};
\draw (u1) -- (v1) -- (w1) -- (u1);
\draw[line width=0.4mm,red] ($(u1)+(0:0.25)$) ellipse (0.7cm and 0.7cm);

\coordinate[vtx, blue] (u1) at (2, 0);
\coordinate[vtx, green] (v1) at ($(u1) + (30:0.5)$);
\coordinate[vtx, green] (w1) at ($(u1) + (-30:0.5)$);
\draw (u1) node[left] {$2$};
\draw  (v1) node[right] {$3$};
\draw  (w1) node[right] {$1$};
\draw (u1) -- (v1) -- (w1) -- (u1);
\draw[line width=0.4mm,red] ($(u1)+(0:0.25)$) ellipse (0.7cm and 0.7cm);

\coordinate[vtx, blue] (u1) at (0, -2);
\coordinate[vtx, green] (v1) at ($(u1) + (30:0.5)$);
\coordinate[vtx, green] (w1) at ($(u1) + (-30:0.5)$);
\draw (u1) node[left] {$3$};
\draw  (v1) node[right] {$1$};
\draw  (w1) node[right] {$2$};
\draw (u1) -- (v1) -- (w1) -- (u1);
\draw[line width=0.4mm,red] ($(u1)+(0:0.25)$) ellipse (0.7cm and 0.7cm);

\coordinate[vtx, blue] (u1) at (2, -2);
\coordinate[vtx, green] (v1) at ($(u1) + (30:0.5)$);
\coordinate[vtx, green] (w1) at ($(u1) + (-30:0.5)$);
\draw (u1) node[left] {$3$};
\draw (v1) node[right] {$2$};
\draw (w1) node[right] {$1$};
\draw (u1) -- (v1) -- (w1) -- (u1);
\draw[line width=0.4mm,red] ($(u1)+(0:0.25)$) ellipse (0.7cm and 0.7cm);

\draw[line width=0.4mm, red] (2.25, 1.3) -- (2.25, 0.7);
\draw[line width=0.4mm, red] (2.25, -1.3) -- (2.25, -0.7);
\draw[line width=0.4mm, red] (0.25, 1.3) -- (0.25, 0.7);
\draw[line width=0.4mm, red] (0.25, -1.3) -- (0.25, -0.7);
\draw[line width=0.4mm, red] (0.95, 2) -- (1.55, 2);
\draw[line width=0.4mm, red] (0.95, -2) -- (1.55, -2);

\end{tikzpicture}
\end{tabular}}
\caption{Schematic representation of graphs $\FSm(X, Y)$ on the left and $\FS(X, Y')$ on the right. }
\label{fig:ex1}
\end{figure}

The red contours represent vertices of the graphs $\FSm(X, Y)$ and $\FS(X, Y')$. Within each red vertex, we see the arrangement of labels $1, 2, 3$ from $V(X)$ onto the vertices of $Y$ or $Y'$ corresponding to it. Note that the vertices of $Y'$ are colored in green and blue in order to keep track which vertex of $Y$ they are coming from.

\noindent
In this example, $\Sym_Y$ consists of only two permutations, and the induced equivalence relation $\equiv$ has equivalence classes of size $2$. These equivalence classes are precisely the pairs of vertices of $\FS(X, Y')$ which are horizontally next to each other --- they differ precisely by exchanging the labels of the two green vertices.
\end{example}

\begin{proof}[Proof of Proposition~\cref{prop:quotientisomorphism}.]
Let us note that every equivalence class of bijections $[\sigma]_\equiv$ can be associated with a function $f_{[\sigma]_\equiv}=\pi_Y\circ \sigma$ from $V(X)$ to $V(Y)$, which is independent of the representative $\sigma\in [\sigma]_\equiv$. The isomorphism of $\FS(X, Y')/\! \equiv$ and $\FSm(X, Y)$ is given by the correspondence $[\sigma]_\equiv\mapsto f_{[\sigma]_\equiv}$. In the rest of the proof, we check that this correspondence preserves the edges of $\FS(X, Y')/\!\equiv$.

We begin by showing that if $[\sigma]_\equiv$ and $[\tau]_\equiv$ are distinct adjacent vertices of $\FS(X, Y')/\!\equiv$, then $f_{[\sigma]_\equiv}$ and  $f_{[\tau]_\equiv}$ are also adjacent in $\FSm(X, Y)$. By definition of the quotient graph $\FS(X, Y')$, we have representatives $\sigma_0\in [\sigma]_\equiv, \tau_0\in [\tau]_\equiv$ differing by a $(X, Y')$-friendly swap along $u_Xv_X\in E(X)$, i.e. $\sigma_0=\tau_0\circ \tr{u_X}{v_X}$. Since $\sigma_0$ and $\tau_0$ are in distinct equivalence classes, the vertices $\sigma_0(u_X), \sigma_0(v_X)\in V(Y')$ are adjacent, but not in the same clique in the clique decomposition of $Y'$. Therefore, the vertices $\pi_Y\circ \sigma_0(u_X), \pi_Y\circ \sigma_0(v_X)$ are distinct and adjacent in $Y$, meaning that swapping vertices $u_X, v_X$ is a $(X, Y)$-friendly swap for the function $f_{[\sigma]_\equiv}$. Since $f_{[\sigma]_\equiv}=f_{[\tau]_\equiv}\circ\tr{u_X}{v_X}$, the vertices $f_{[\sigma]_\equiv}$ and $f_{[\tau]_\equiv}$ are adjacent in $\FS(X, Y)$.

To show the converse statement, suppose $f_{[\sigma]_\equiv}$ and $f_{[\tau]_\equiv}$ are vertices of $\FS(X, Y)$ which differ by a $(X, Y)$-friendly swap of adjacent labels $u_X, v_X\in V(X)$ along the edge $u_Yv_Y\in E(Y)$, where $u_Y, v_Y$ are vertices of $Y$. For any representative $\sigma_0\in [\sigma]_\equiv$, we have $\sigma_0(u_X)\in S_{u_Y}$ and $\sigma_0(v_X)\in S_{v_Y}$, which means that $\sigma_0(u_X)\sigma_0(v_X)$ is an edge of $Y'$. Therefore, the swap $\sigma_0\mapsto \sigma_0\circ \tr{u_X}{v_X}$ is a $(X, Y')$-friendly swap. Noting that $\sigma_0\circ \tr{u_X}{v_X}\equiv \tau$, we conclude that there exists a $(X, Y')$-friendly swap mapping between representatives of $[\sigma]_\equiv$ and $[\tau]_\equiv$. This completes the proof of Proposition~\cref{prop:quotientisomorphism}.
\end{proof}

\noindent
Since the main focus of this work is connectivity of $\FSm(X, Y)$, the following simple corollary gives us an efficient criterion for determining the connectivity of $\FSm(X, Y)$ in terms of connectivity of $\FS(X, Y')$.

\begin{corollary}\label{cor:lifting}
Let $X$ be a simple graph on $n$ vertices and let $Y$ be a multiplicity graph of total multiplicity $n$. If the graph $\FS(X, Y')$ is connected, then so is $\FSm(X, Y)$.
\end{corollary}
\begin{proof}
It is immediate to see that the quotient $\FS(X, Y')/\!\equiv$ is connected whenever $\FS(X, Y')$ is connected. Together with Proposition~\cref{prop:quotientisomorphism}, this implies connectedness of $\FSm(X, Y)$.
\end{proof}

\noindent
In fact, we can establish an even more precise description of the connected components of $\FSm(X, Y)$ just based on $\FS(X, Y')$, which will be very important in Section \cref{sec:pathscycles}.

\begin{corollary}\label{cor:quotientcomponents}
Let $X$ be a simple graph on $n$ vertices and let $Y$ be a multiplicity graph of total multiplicity $n$. Let $E_0=\{\sigma\tau|\sigma, \tau\in V(\FS(X, Y'))$ and $\sigma\equiv\tau\}$ and let $G$ be the graph obtained by adding the edges of $E_0$ to $\FS(X, Y')$. Then, the connected components of $\FSm(X, Y)$ have the same vertex sets as projections of connected components of $G$ under the map $\sigma\mapsto \pi_Y\circ \sigma$.
\end{corollary}
\begin{proof}
By Proposition~\cref{prop:quotientisomorphism}, we may consider the graph $\FS(X, Y')/\!\equiv$ instead of $\FSm(X, Y)$, with the map $\sigma\mapsto [\sigma]_\equiv$ replacing $\sigma\mapsto \pi_Y\circ \sigma$. Let us denote the projection $\sigma\mapsto [\sigma]_\equiv$ by $\pi_\equiv$. 

Note that the partition of $V(G)$ into connected components is coarser than the partition into equivalence classes $[\sigma]_\equiv$, and therefore different connected components have disjoint images under $\pi_\equiv$. Since the map $\pi_\equiv:V(\FS(X, Y'))\to V(\FS(X, Y')/\!\equiv)$ is surjective, to complete the proof it suffices to show that, for an arbitrary connected component $C$ of the graph $G$, $\left(\FS(X, Y')/\!\equiv\right)|_{\pi_\equiv(C)}$ is a connected component of $\FS(X, Y')/\!\equiv$.

We begin by showing there are no outgoing edges from $\pi_\equiv(C)$. Suppose the contrary, and let $[\sigma]_\equiv[\tau]_\equiv$ be an edge of $\FS(X, Y')/\!\equiv$ with $[\sigma]_\equiv\in \pi_\equiv(C), [\tau]_\equiv\notin \pi_\equiv(C)$. By definition of $\FS(X, Y')/\!\equiv$, this means there exist representatives $\sigma_0\in [\sigma]_\equiv, \tau_0\in [\tau]_\equiv$ with $\sigma_0\tau_0\in E(\FS(X, Y'))$. Since the partition of $V(G)$ into connected components is coarser than the partition into equivalence classes, we must have $\sigma_0\in C$ and $\tau_0\notin C$, which is a contradiction to the fact that $C$ is a connected component of $G$.

On the other hand, by definition of $\FS(X, Y')/\!\equiv$, the map $\pi_\equiv$ preserves adjacencies between nonequivalent bijections, which means that $\left(\FS(X, Y')/\!\equiv\right)|_{\pi_\equiv(C)}$ is a connected subgraph of $\FS(X, Y')/\!\equiv$. This completes the proof of Corollary~\cref{cor:quotientcomponents}.
\end{proof}

\section{Random bipartite graphs}\label{sec:random}
\subsection{General proof strategy}\label{subsec:setup}

The main goal of this section is to prove Theorem~\cref{thm:randombipartiteconnectivity}, following the proof strategy presented in \cite{ADK}. We begin by introducing the terminology used in the proof and presenting a high-level overview of the proof strategy.

The first step of the proof is to reduce the question about connectivity of the graph $\FS(X, Y)$ to a local problem about exchangeability of certain pairs of vertices. More precisely, we begin by reducing Theorem~\cref{thm:randombipartiteconnectivity} to Proposition~\cref{prop:bipartiteexchangeability}, which mirrors Proposition 4.5 of \cite{ADK}.

\begin{proposition}\label{prop:bipartiteexchangeability}
Let $X, Y$ be independently chosen random graphs in $\G(K_{n, n}, p)$, where $p$ satisfies \cref{eqn:pbound}, as in Theorem \cref{thm:randombipartiteconnectivity}. Then, the following statement holds with high probability: for any arrangement $\sigma:V(X)\to V(Y)$ and any two vertices $u_0, v_0$ in different partite sets of $Y$ whose preimages $\sigmai(u_0), \sigmai(v_0)$ are in different partite sets of $X$, one has that $u_0$ and $v_0$ are $(\t X, Y)$-exchangeable from $\sigma$, where $\t X$ denotes the graph $X$ with edge $\sigmai(u_0)\sigmai(v_0)$ added.
\end{proposition}

\noindent
The following argument, reproduced from \cite{ADK}, shows how Proposition~\cref{prop:bipartiteexchangeability} implies Theorem~\cref{thm:randombipartiteconnectivity}. Intuitively, this argument should be thought of as the analogue of Proposition~\cref{prop:exchangeabilityimpliesconnectedness} in the bipartite setting.

\begin{proof}[Proof of Theorem~\cref{thm:randombipartiteconnectivity} from Proposition~\cref{prop:bipartiteexchangeability}.]

We use Proposition 2.8 from \cite{ADK}, which states the following. Let $X, Y$ and $\t Y$ be graphs on $[n]$, having $Y\subseteq \t Y$. Suppose that for every edge $uv\in E(\t Y)$ and every $\sigma:V(X)\to V(Y)$ with $\sigmai(u)\sigmai(v)\in E(X)$ the vertices $u, v$ are $(X, Y)$-exchangeable from $\sigma$. Then the number of connected components of $\FS(X, Y)$ is the same as the number of connected components of $\FS(X, \t Y)$.

Let us begin by applying Proposition~\cref{prop:bipartiteexchangeability} to conclude that, with high probability, any two vertices $u_0, v_0$ in different partite sets of $Y$ satisfying $\sigmai(u_0)\sigmai(v_0)\in E(X)$ are $(X, Y)$-exchangeable, as $\t X=X$. Hence, by setting $\t Y=K_{n, n}$ in Proposition 2.8 of \cite{ADK}, we conclude that $\FS(X, Y)$ has the same number of connected components as $\FS(X, K_{n, n})$. Since $\FS(X, K_{n, n})\cong \FS(K_{n, n}, X)$, it suffices to show that $\FS(K_{n, n}, X)$ has two connected components with high probability.

Swapping the roles of $X$ and $Y$ in Proposition~\cref{prop:bipartiteexchangeability}, we can also say that, with high probability, for any two vertices $u_0', v_0'$ in different partite sets of $X$, $u_0'$ and $v_0'$ are $(\t Y, X)$-exchangeable from any $\tau:V(Y)\to V(X)$, where $\t Y$ is the graph $Y$ with $\taui(u_0)\taui(v_0)$ added. Since $\t Y\subset K_{n, n}$, any two vertices $u_0', v_0'$ in different partite sets of $X$ are $(K_{n, n}, X)$-exchangeable.

Therefore, by Proposition 2.8 of \cite{ADK}, we conclude that $\FS(K_{n, n}, X)$ has the same number of connected components as $\FS(K_{n, n}, K_{n, n})$, which is two, as shown by Proposition 2.6 of \cite{ADK}. Hence, we conclude that the graph we started with, $\FS(X, Y)$, also has two connected components with high probability, completing the proof.
\end{proof}

\noindent
To prove Proposition~\cref{prop:bipartiteexchangeability}, we combine the methods used in \cite{ADK} to tackle the bipartite and non-bipartite cases. The approach for proving Proposition~\cref{prop:bipartiteexchangeability} relies on constructing sparse bipartite graphs $G, H$ on the vertex set $[m]\cup \{u, v\}$, in which the vertices $u$ and $v$ are $(G, H)$-exchangeable starting from the identity arrangement. Since the construction is involved and contains many details, we postpone the details until Subsections \cref{subsec:construction} and \cref{subsec:exchangeable}. Still, we will point out the properties of $G, H$ relevant for the proof.

After constructing $G$ and $H$, our next objective will be to show the following statement: With high probability, for every arrangement $\sigma$ and vertices $u_0, v_0$ from Proposition~\cref{prop:bipartiteexchangeability}, there exist embeddings $\psi_G:V(G)\to V(\t X)$ and $\psi_H:V(H)\to V(Y)$ respecting the arrangement $\sigma$ in the sense that $\sigma\circ \psi_G=\psi_H\circ \id$ and having $\psi_H(u)=u_0, \psi_H(v)=v_0$. In this context, an embedding is merely an injective map which preserves adjacencies. This objective will be accomplished in Subsections \cref{subsec:bipartiteembeddable} and \cref{subsec:constructingembeddings}. 

Intuitively, finding embeddings $\psi_G, \psi_H$ corresponds to finding subgraphs of $X$ and $Y$ which are isomorphic to $G$ and $H$ and are mapped to each other under $\sigma$. Proposition~\cref{prop:embedabilityimpliesexchangeability} shows how embeddings $\psi_G, \psi_H$ can be used to ensure that $u_0, v_0$ are $(X, Y)$-exchangeable.

\begin{proposition}\label{prop:embedabilityimpliesexchangeability}
Let $X, Y$ be bipartite graphs with the vertex set $[n]$, and let $\sigma:V(X)\to V(Y)$ be an arbitrary arrangement. Further, let $G, H$ be bipartite graphs with the vertex set $[m]\cup \{u, v\}$, such that $u$ and $v$ are $(G, H)$-exchangeable starting from the identity arrangement $\id:V(G)\to V(H)$. Suppose that there exist embeddings $\psi_G:V(G)\to V(\t X)$ and $\psi_H:V(H)\to V(Y)$ with $\sigma\circ \psi_G=\psi_H\circ \id$ and $\psi_H(u)=u_0, \psi_H(v)=v_0$. Then, $u_0$ and $v_0$ are $(\t X, Y)$-exchangeable from $\sigma$.
\end{proposition}
\begin{proof}
Let the sequence of $(G, H)$-friendly swaps transforming $\id$ into $\tr{u}{v}$ be $a_1b_1, a_2b_2, \dots$, $a_kb_k$. Then, the sequence of swaps $\psi_H(a_1)\psi_H(b_1), \dots, \psi_H(a_k)\psi_H(b_k)$ takes $\sigma$ to $\tr{u_0}{v_0}\circ \sigma$. Properties of $\psi_G, \psi_H$ guarantee that all of these swaps are $(\t X, Y)$-friendly, and hence $(u_0, v_0)$ are $(\t X, Y)$-exchangeable.
\end{proof}

\noindent
Finally, let us describe the strategy we use to construct the embeddings $\psi_G:V(G)\to V(X)$ and $\psi_H:V(H)\to V(Y)$. The idea behind this step of the proof is to choose large disjoint sets $V_1, \dots, V_m\subseteq V(Y)$ and look for a $m$-tuple of vertices $(\psi_H(1), \dots, \psi_H(m))\in V_1\times \dots \times V_m$ which will, together with vertices $u_0$ and $v_0$, span a copy of $H$ in $Y$ and whose preimages under $\sigma$ span a copy of $G$ in $X$. To formalize this idea, we need to introduce the following terminology, which initially appeared in \cite{ADK}. 

Let $G'=G|_{[m]}$ and $H'=H|_{[m]}$ be graphs $G, H$ with vertices $u, v$ removed, and let $V(G')=A_{G'}\bigsqcup B_{G'}$, $V(H')=A_{H'}\bigsqcup B_{H'}$ be the respective bipartitions. Also, assume that $X$ and $Y$ have bipartitions $V(X)=A_X\bigsqcup B_X$ and $V(Y)=A_Y\bigsqcup B_Y$. For a fixed arrangement $\sigma:V(X)\to V(Y)$, we call disjoint sets $V_1, \dots, V_m\subseteq V(Y)$ \dfn{admissible} for $\sigma$ and $(G', H')$ if the following two conditions are satisfied: 
\begin{itemize}
    \item $V_i\subseteq A_Y$ for $i\in A_{H'}$ and $V_i\subseteq B_Y$ if $i\in B_{H'}$, and 
    \item $\sigma^{-1}(V_i)\subseteq A_X$ for $i\in A_{G'}$ and $\sigmai(V_i)\subseteq B_X$ for $i\in B_{G'}$.
\end{itemize}

\noindent
We say that the pair $(G', H')$ is \dfn{bipartite-embeddable in $(X, Y)$} with respect to admissible sets $V_1, \dots, V_m$ and an arrangement $\sigma$ if there exist embeddings $\psi_{G}:V(G')\to V(X), \psi_{H}:V(H')\to V(Y)$ such that $\psi_{H}(i)\in V_i$ and $\sigma\circ \psi_{G}=\psi_{H}\circ\id$. Moreover, if $q_1, \dots, q_m$ are fixed positive integers satisfying $q_1+\cdots+q_m\leq n$, we say that $(G', H')$ is \dfn{$(q_1, \dots, q_m)$-bipartite-embeddable in $(X, Y)$} if $(G', H')$ is bipartite-embeddable in $(X, Y)$ for any admissible choice of $V_1, \dots, V_m$ and $\sigma$ satisfying $|V_i|=q_i$ for all $i\in [m]$.

The concept of $(q_1, \dots, q_m)$-bipartite-embeddability is useful for the following reasons. On one hand, Lemma 4.2 of \cite{ADK} gives an simple condition on $(q_1, \dots , q_m)$ which ensures that the pair $(G', H')$ is $(q_1, \dots, q_m)$-bipartite-embeddable in $X, Y$ with high probability. Since each of $q_1, \dots, q_m$ must be somewhat large in order to satisfy this condition, we cannot directly infer the existence of embeddings $\psi_G, \psi_H$, since Proposition~\cref{prop:embedabilityimpliesexchangeability} requires that $\psi_H(u)=u_0$ and $\psi_H(v)=v_0$.

Hence, we need to use the fact $(G', H')$ is $(q_1, \dots, q_m)$-biparite-embeddable in the following way. The idea is to choose $V_i\subseteq N_Y(u_0)$ for all $i\in N_H(u)$ and $V_i\subseteq \sigma(N_X(u_0'))$ for all $i\in N_G(u)$, since this ensures $\psi_H(i)u_0\in E(Y), \psi_G(i)u_0'\in E(X)$ (we require similar conditions for $i\in N_G(v), N_H(v)$). However, since the sets $V_1, \dots, V_m$ must be admissible, we must simultaneously ensure $V_i\subseteq A_Y$ for all $i\in A_H$ and $V_i\subseteq \sigma(A_X)$ for $i\in A_G$. Hence, based on how the pairs of sets $\sigma(N_X(u_0')), A_Y$ and $N_Y(u_0), \sigma(A_X)$ intersect, we have have four cases. To describe them, we need the notion of \textit{majority-mapping}, introduced in \cite{ADK}.

For an arrangement $\sigma:V(X)\to V(Y)$ and vertices $u_0, v_0$ having preimages $u_0'=\sigmai(u_0)$ and $v_0'=\sigmai(v_0)$, we say that $\sigma$ \textit{majority-maps} $N_X(u_0')$ to $A_Y$ if $|\sigma(N_X(u_0'))\cap A_Y|\geq \frac{1}{2} |N_X(u_0')|$. This definition is naturally generalized if we replace $A_Y$ by $B_Y$ or if we replace $\sigma$ and $Y$ by $\sigmai$ and $X$. We adopt the convention of naming $A_X$ and $A_Y$ so that $u_0'\in A_X$ and $u_0\in A_Y$. Based on where $\sigma$ and $\sigmai$ majority-maps $N_X(u_0'), N_X(v_0'), N_Y(u_0), N_Y(v_0)$, we have four essentially different cases.\footnote{Although these four cases are reminiscent of the four cases in the proof of Proposition 4.5 of \cite{ADK}, there is a slight difference. Namely, \cite{ADK} seems to miss the subtle difference between cases (C3) and (C4) presented here, and accounts for both of them in the case (IV). Further, \cite{ADK} does not make use of the symmetry $\sigma\mapsto \sigmai$ which allows us to reduce their cases (II) and (III) to a single case (C2).}

\begin{enumerate}[start=1, label={(C\arabic*):}]
    \item $\sigma$ majority-maps $N(u_0')$ and $N(v_0')$ to $B_Y$, and $\sigmai$ majority-maps $N(u_0)$ and $N(v_0)$ to $A_X$.
    \item $\sigma$ majority-maps $N(u_0')$ and $N(v_0')$ to $B_Y$, and $\sigmai$ majority-maps $N(u_0)$ to $A_X$ and $N(v_0)$ to $B_X$.
    \item $\sigma$ majority-maps $N(u_0')$ to $B_Y$ and $N(v_0')$ to $A_Y$, and $\sigmai$ majority-maps $N(u_0)$ to $A_X$ and $N(v_0)$ to $B_X$.
    \item $\sigma$ majority-maps $N(u_0')$ to $B_Y$ and $N(v_0')$ to $A_Y$, and $\sigmai$ majority-maps $N(u_0)$ to $B_X$ and $N(v_0)$ to $A_X$.
\end{enumerate}
\noindent
As we will see, all other cases can be reduced to those four by applying symmetry. These cases seem to be merely a technical requirement of our argument and it might be possible to apply some additional symmetries, or slightly alter the argument, in order to further reduce the number of cases.

In anticipation of these four cases, we will construct four pairs of graphs $(\GR, \HR)$, for $\rho\in \{1, 2, 3, 4\}$, in which sets $N_G(u), N_H(u), N_G(v), N_H(v)$ are disjoint and satisfy the following properties, corresponding to the cases (C1), (C2), (C3), (C4).

\begin{enumerate}[label={(C\arabic*):}]
    \item $N_{G^{(1)}}(u), N_{G^{(1)}}(v)\subseteq B_{H^{(1)}}$, $N_{H^{(1)}}(u), N_{H^{(1)}}(v)\subseteq A_{G^{(1)}}$.
    \item $N_{G^{(2)}}(u), N_{G^{(2)}}(v)\subseteq B_{H^{(2)}}$, $N_{H^{(2)}}(u)\subseteq A_{G^{(2)}}, N_{H^{(2)}}(v)\subseteq B_{G^{(2)}}$.
    \item $N_{G^{(3)}}(u)\subseteq A_{H^{(3)}}, N_{G^{(3)}}(v)\subseteq B_{H^{(3)}}$, $N_{H^{(3)}}(u)\subseteq A_{G^{(3)}}, N_{H^{(3)}}(v)\subseteq B_{G^{(3)}}$.
    \item $N_{G^{(4)}}(u)\subseteq A_{H^{(4)}}, N_{G^{(4)}}(v)\subseteq B_{H^{(4)}}$, $N_{H^{(4)}}(u)\subseteq B_{G^{(4)}}, N_{H^{(4)}}(v)\subseteq A_{G^{(4)}}$.
\end{enumerate}

\noindent
Before concluding this subsection, let us present the plan for this section. In Subsection \cref{subsec:construction}, we provide all details for constructing $\GR, \HR$. Then, in Subsection \cref{subsec:exchangeable}, we show that $u, v$ are indeed $(\GR, \HR)$-exchangeable from the identity permutation. Further, Subsection \cref{subsec:bipartiteembeddable} uses the lemma of Alon, Defant, and Kravitz that, for an appropriate choice of $q_1, \dots q_m$, the graphs $(\GR, \HR)$ are, with high probability, $(q_1, \dots, q_m)$-bipartite-embeddable in $X, Y$. Then Subsection \cref{subsec:constructingembeddings} shows how to use $(q_1, \dots, q_m)$-bipartite-embeddability in order to construct embeddings $\psi_G, \psi_H$ satisfying the conditions of Proposition~\cref{prop:embedabilityimpliesexchangeability}. Finally, in Subsection \cref{subsec:finalproof} we complete the proof by putting together all ingredients.

\subsection{Construction of $\GR$ and $\HR$}\label{subsec:construction}

In this subsection, we describe the details of the construction of graphs $\GR, \HR$. We set the number of vertices of $\GR, \HR$ to be $m^{(\rho)}+2$, where $m^{(1)}=m^{(2)}$ is the smallest integer larger than $(\log n)^{4/5}$ divisible by $8$ and $m^{(3)}=m^{(4)}=m^{(1)}+3$.

Throughout this subsection, we will simplify the notation $\GR, \HR, m^{(\rho)}$ and only write $G, H, m$, where the index $\rho$ will be understood. As the constructions in the four cases will be essentially the same, we present them together and note where the differences arise.

Let $\l=\lfloor \frac{1}{4}m^{1/4}\rfloor+\epsilon$, where $\epsilon\in \{0, \dots, 7\}$ is chosen so that $8|\l$ . In what follows, we will only deal with asymptotics as $m, n\to \infty$, and hence we will omit the floor symbols and $\epsilon$.

As previously noted, the vertex set of $G, H$ will be $[m]\cup \{u, v\}$. In the cases (C1) and (C2), we denote the elements of $[m]$ by
\[x_1, \dots, x_\l, y_1, \dots, y_\l , z_{1, 1}, \dots, z_{\l, k}, t_1, \dots, t_s, s_1, s_2, s_3, r_1, r_2, r_3, w,\]
where $k=2\l $ and $s=m-7-\l(k+2)$. In the cases (C3) and (C4), we denote the elements of $[m]$ by
\[x_1, \dots, x_\l, y_1, \dots, y_\l  , z_{1, 1}, \dots, z_{\l, k}, \t z_{1, 1}, \dots, \t z_{3, k}, t_1, \dots, t_s, s_1, s_2, s_3, r_1, r_2, r_3, w,\]
where $s=m-7-\l (k+2)-3k$.

Introducing such notation for elements of $[m]$ primarily serves to distinguish the neighbors of $u$ and $v$ from other vertices of the graph. For example, we will have $N_G(u)=\s=\{s_1, s_2, s_3\}, N_G(v)=\r=\{r_1, r_2, r_3\}$, $N_H(u)=\x=\{x_1, \dots, x_\l\}, N_H(v)=\y=\{y_1, \dots, y_\l  \}$. For the sake of completeness, we also introduce notation $\z=\{z_{1, 1}, \dots, z_{\l, k}\}$ in the cases (C1) and (C2), and $\z=\{z_{1, 1}, \dots, z_{\l, k}, \t z_{1, 1},$ $\dots, \t z_{3, k}\}$ in the cases (C3) and (C4). We will also assume that $G, H$ have bipartitions $V(G)=A_G\bigsqcup B_G$ and $V(H)=A_H\bigsqcup B_H$, with $u\in A_G, A_H$ and $v\in B_G, B_H$. 

We begin by describing the construction of $H$, since is much simpler than that of $G$. In the cases (C1) and (C2), we let $H$ have a bipartiton $A_H, B_H$, with 
\[A_H=\{u, y_1, \dots, y_\l  , w\}, \hspace{1cm} B_H=\{v, z_{1, 1}, \dots, z_{\l, k}, x_1, \dots, x_\l, t_1, \dots, t_s, s_1, s_2, s_3, r_1, r_2, r_3\}\] 
and the edges $\{ux_i:i\in [\l]\}, \{vy_i: i\in [\l]\}$, $\{y_i, z_{i, j}:i\in [\l], j\in [k]\}$, $\{wa:a\in B_H\}$. In the cases (C3) and (C4), we let $H$ have a bipartiton $A_H, B_H$, with \[A_H=\{u, y_1, \dots, y_\l  , w, s_1, s_2, s_3\}, \hspace{1cm} B_H=\{v, z_{1, 1},  \dots, \t z_{3, k}, x_1, \dots, x_\l, t_1, \dots, t_s, r_1, r_2, r_3\}\] and the edges $\{ux_i:i\in [\l]\}, \{vy_i: i\in [\l]\}$, $\{y_i, z_{i, j}:i\in [\l], j\in [k]\}$, $\{wa:a\in B_H\}$, $\{s_i, \t z_{i, j}: i\in [3], j\in [k]\}$. Figure~\cref{fig:graphH} illustrates this.

\begin{figure}[h]
\begin{tabular}{  c c  } 
\begin{tikzpicture}
\coordinate[vtx] (w) at (2, 0);
\draw (w) node[right] {$w$};

\coordinate[vtx] (x1) at (0, 3);
\draw (x1) node[above left] {$x_1$};
\coordinate[vtx] (xl) at (0, 2.2);
\draw (xl) node[above left] {$x_\l$};
\draw (0, 2.7) node {$\vdots$};

\draw (w) -- (x1);
\draw (w) -- (xl);

\coordinate[vtx] (u) at (-2, 2.55);
\draw (u) node[above] {$u$};

\draw (u) -- (x1);
\draw (u) -- (xl);

\coordinate[vtx] (z11) at (0, 1.7);
\draw (z11) node[above left] {$z_{1, 1}$};
\coordinate[vtx] (z1k) at (0, 0.9);
\draw (z1k) node[above left] {$z_{1, k}$};
\draw (0, 1.4) node {$\vdots$};
\draw (w) -- (z11);
\draw (w) -- (z1k);

\coordinate[vtx] (y1) at (-2, 1.3);
\draw (y1) node[above] {$y_1$};
\draw (y1) -- (z11);
\draw (y1) -- (z1k);

\draw (0, 0.6) node {$\vdots$};

\coordinate[vtx] (zl1) at (0, 0.0);
\draw (zl1) node[above left] {$z_{\l, 1}$};
\coordinate[vtx] (zlk) at (0, -0.8);
\draw (zlk) node[above left] {$z_{\l, k}$};
\draw (0, -0.3) node {$\vdots$};

\draw (w) -- (zl1);
\draw (w) -- (zlk);

\coordinate[vtx] (yl) at (-2, -0.3);
\draw (yl) node[above] {$y_\l  $};
\draw (yl) -- (zl1);
\draw (yl) -- (zlk);

\coordinate[vtx] (v) at (-4, 0.5);
\draw (v) node[above] {$v$};
\draw (-2, 0.7) node {$\vdots$};
\draw (y1) -- (v);
\draw (yl) -- (v);

\coordinate[vtx] (t1) at (0, -1.3);
\draw (t1) node[left] {$t_1$};
\coordinate[vtx] (ts) at (0, -2.1);
\draw (ts) node[left] {$t_s$};
\draw (0, -1.6) node {$\vdots$};

\draw (w) -- (t1);
\draw (w) -- (ts);

\coordinate[vtx] (s1) at (0, -2.6);
\draw (s1) node[left] {$s_1$};
\coordinate[vtx] (r3) at (0, -3.4);
\draw (r3) node[left] {$r_3$};
\draw (0, -2.9) node {$\vdots$};
\draw (w) -- (s1);
\draw (w) -- (r3);

\end{tikzpicture}
& 
\begin{tikzpicture}
\coordinate[vtx] (w) at (2, -0.5);
\draw (w) node[right] {$w$};

\coordinate[vtx] (x1) at (0, 3);
\draw (x1) node[above left] {$x_1$};
\coordinate[vtx] (xl) at (0, 2.2);
\draw (xl) node[above left] {$x_\l$};
\draw (0, 2.7) node {$\vdots$};

\draw (w) -- (x1);
\draw (w) -- (xl);

\coordinate[vtx] (u) at (-2, 2.55);
\draw (u) node[above] {$u$};

\draw (u) -- (x1);
\draw (u) -- (xl);

\coordinate[vtx] (z11) at (0, 1.7);
\draw (z11) node[above left] {$z_{1, 1}$};
\coordinate[vtx] (z1k) at (0, 0.9);
\draw (z1k) node[above left] {$z_{1, k}$};
\draw (0, 1.4) node {$\vdots$};
\draw (w) -- (z11);
\draw (w) -- (z1k);

\coordinate[vtx] (y1) at (-2, 1.3);
\draw (y1) node[above] {$y_1$};
\draw (y1) -- (z11);
\draw (y1) -- (z1k);

\draw (0, 0.6) node {$\vdots$};

\coordinate[vtx] (zl1) at (0, 0.0);
\draw (zl1) node[above left] {$z_{\l, 1}$};
\coordinate[vtx] (zlk) at (0, -0.8);
\draw (zlk) node[above left] {$z_{\l, k}$};
\draw (0, -0.3) node {$\vdots$};

\draw (w) -- (zl1);
\draw (w) -- (zlk);

\coordinate[vtx] (yl) at (-2, -0.3);
\draw (yl) node[above] {$y_\l  $};
\draw (yl) -- (zl1);
\draw (yl) -- (zlk);

\coordinate[vtx] (v) at (-4, 0.5);
\draw (v) node[above] {$v$};
\draw (-2, 0.7) node {$\vdots$};
\draw (y1) -- (v);
\draw (yl) -- (v);

\coordinate[vtx] (t1) at (0, -1.3);
\draw (t1) node[left] {$t_1$};
\coordinate[vtx] (ts) at (0, -2.1);
\draw (ts) node[left] {$t_s$};
\draw (0, -1.6) node {$\vdots$};

\draw (w) -- (t1);
\draw (w) -- (ts);

\coordinate[vtx] (s1) at (0, -2.6);
\draw (s1) node[left] {$r_1$};
\coordinate[vtx] (r3) at (0, -3.4);
\draw (r3) node[left] {$r_3$};
\draw (0, -2.9) node {$\vdots$};
\draw (w) -- (s1);
\draw (w) -- (r3);

\coordinate[vtx] (tz11) at (0, -3.9);
\draw (tz11) node[below left] {$\t z_{1, 1}$};
\coordinate[vtx] (tz1k) at (0, -4.7);
\draw (tz1k) node[below left] {$\t z_{1, 2}$};
\draw (0, -4.2) node {$\vdots$};

\draw (w) -- (tz11);
\draw (w) -- (tz1k);

\coordinate[vtx] (s1) at (-2, -4.3);
\draw (s1) node[above] {$s_1$};
\draw (s1) -- (tz11);
\draw (s1) -- (tz1k);

\draw (-2, -4.7) node {$\vdots$};
\draw (0, -5) node {$\vdots$};
\end{tikzpicture}
\end{tabular}
\caption{Diagrams of graphs $H$, for $\rho=1, 2$ on the left and for $\rho=3, 4$ on the right.}
\label{fig:graphH}
\end{figure}
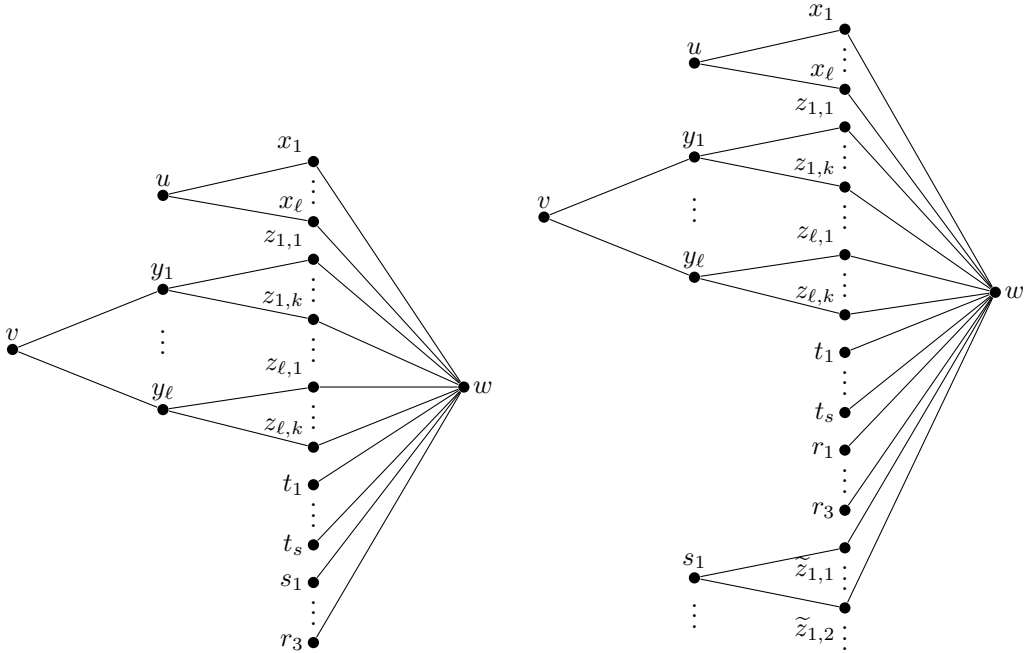

\noindent
Now, we describe the construction of $G$. Set $g$ to be the smallest integer greater than $m^{3/4}/4$ that is divisible by $4$. In the cases (C1) and (C2), let $G$ consist of a big cycle containing all vertices but $u, v$, $y_1, \dots, y_\l  $. In the cases (C3) and (C4), on top of these, the big cycle does not contain $\t z_{1, 1}, \t z_{2, 1}, \t z_{3, 1}$. As the big cycle is of even length, it is bipartite. We also introduce a notational convention that, if $a, b$ are vertices of the big cycle of $G$, $[a, b]$ denotes clockwise interval from $a$ to $b$, including all vertices between $a$ and $b$ as well as the endpoints. Further, we denote the clockwise distance between $a$ and $b$ by $d(a, b)$. 

As previously noted, the vertices $u$, $v$ are adjacent to $s_1, s_2, s_3$ and $r_1, r_2, r_3$, respectively. We arrange $s_1, s_2, s_3$ and $r_1, r_2, r_3$ so that $d(s_i, r_i)=\l'=\l-1$, $d(s_2, s_1)=d(s_3, s_2)=m/4$. Since $2|d(s_i, s_j)$ and $2\not | d(s_i, r_j)$, we conclude that $s_1, s_2, s_3$ lie in the same bipartite set, and that $r_1, r_2, r_3$ lie in the other partite set of $G$. Let the partite set of $G$ containing $s_1, s_2, s_3$ be $B_G$, and let the other partite set be $A_G$. Since $u\in A_G, v\in B_G$, even after adding edges $us_i$ and $vr_j$, the graph $G$ is still bipartite.

Let $S_0$ be the set of vertices of the big cycle that are in one of the intervals $[s_i, r_i]$, or whose neighbors on the big cycle are either $s_i$ or $r_i$. For each vertex of $a\in S_0$, we will add two edges to vertices of the interval $[r_3, s_2]$. We add these edges so that the graph remains bipartite and so that all of their endpoints and $s_2, r_3$ have pairwise distance at least $m/(10\l)\geq g$. The reason for adding these edges is to ensure that $G$ remain Wilsonian even upon removing a set of vertices from the intervals $[s_i, r_i]$.

We also place the vertices $x_1, \dots, x_\l$ to the interval $[r_3, s_2]$, making sure that their distance from each of the edge endpoints discussed in the previous paragraph is at least $m/(20\l )$. Additionally, in the cases (C1), (C2) and (C3) we ensure $x_i\in A_G$ for $i\in [\l]$ and in the case (C4) we ensure $x_i\in B_G$. Finally, we ensure that the distance between all $x_i$ and $s_2$ or $r_3$ is at least $g$. Since $d(r_3, s_2)=m/4-\l'$, the interval $[r_3, s_2]$ is long enough for this to be done.

Similarly, we place $z_{1, 1}, \dots, z_{\l, k}, w$ (and also $\t z_{i, j}$, $i\in[3], j\in [2, k]$ in the cases $3$ and $4$) onto $[r_2, s_1]$, ensuring that the distance between any two of them is at least $m/(10\l k)$ and distance between any of them to $s_1$ is at least $g$.

The vertices $y_i$ will be adjacent to exactly two vertices on the big cycle --- we denote them $t_{2i-1}$ and $t_{2i}$. Let us set $d(t_{2i-1}, t_{2i})=g$ and $d(t_{2i}, t_{2i+1})=k'=k-2$. Since both $g$ and $k$ are even, this means all $t_i$ for $i\in \{1, \dots, 2\l \}$ are in the same partite set of $G$, meaning that all vertices $y_i$ are in the other partite set of $G$. Further, we set $d(s_1, t_1)=k-1$ in the cases (C2), (C3), while $d(s_1, t_1)=k-2$ in the cases (C1), (C4). These choices ensure that $\y\subseteq A_G$ in the cases (C1), (C4) and that $\y\subseteq B_G$ in the cases (C2), (C3) for $i\in [\l]$.

Finally, we ensure that $s_1$ is adjacent to a vertex $q$ whose distance from is $t_{2\l}$ is $k-2$ in the cases (C1), (C4) and $k-1$ in the cases (C2), (C3). In either case, we have $q\in A_G$, and hence $G$ remains bipartite after adding the edge $s_1q$.

In the cases (C3) and (C4), the graph $G$ has additional vertices $\t z_{1, 1}, \t z_{2, 1}, \t z_{3, 1},\t z_{1, 2}$ outside the big cycle. Let $\t t_1, \dots, \t t_6$ be vertices of the big cycle satisfying $d(\t t_{2i-1}, s_i)=k'=k-2$, $d(\t t_{2i}, \t t_{2i-1})=g$, for $i=1, 2, 3$. Then, we add edges $\t t_{1}\t z_{1, 1}, \t z_{1, 1}\t t_{2}, \t t_3 \t z_{2, 1}, \t z_{2, 1} \t t_4, \t t_5 \t z_{3, 1}, \t z_{3, 1}\t t_6$ to the graph $G$. It is simple to check that all of these edges run between different partite sets of $G$, which completes the description of the graphs $G$ in all the cases. The constructions are illustrated in Figure~\cref{fig:graphG}.

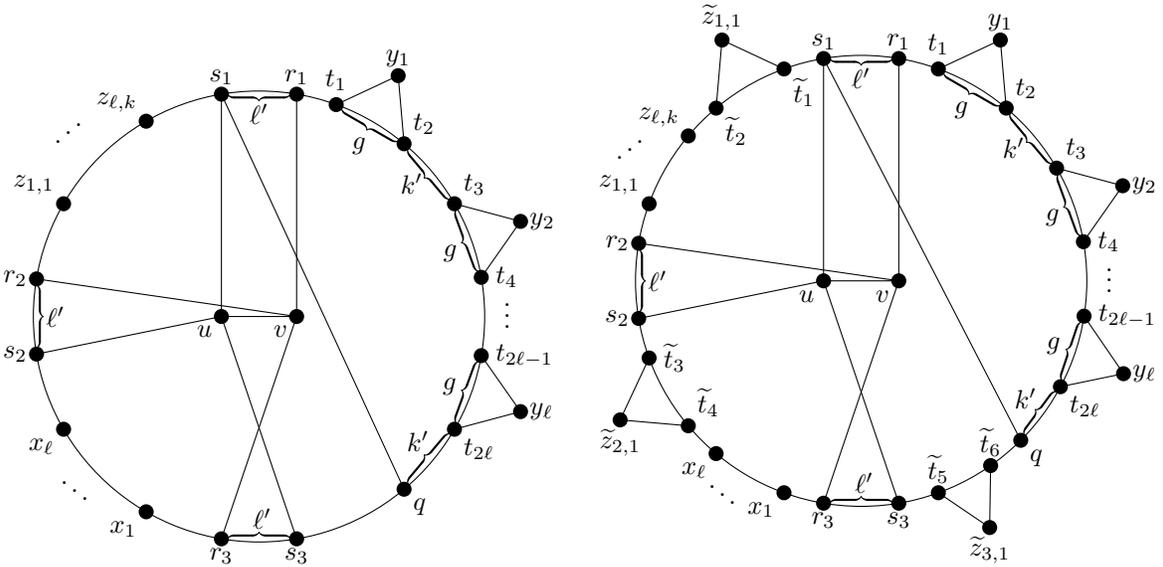
\begin{figure}[h!]
\begin{tabular}{c c}
\begin{tikzpicture}
\draw (0,0) circle (3);
\coordinate[vtx] (u) at (-0.5, 0);
\draw (u) node[below left] {$u$};
\fill (u) circle (1mm);

\coordinate[vtx] (v) at (0.5, 0);
\draw (v) node[below left] {$v$};
\fill (v) circle (1mm);

\coordinate[vtx] (s1) at (-0.5, 2.958);
\draw (s1) node[above] {$s_1$};
\fill (s1) circle (1mm);
\coordinate[vtx] (s2) at (-2.958, -0.5);
\draw (s2) node[left] {$s_2$};
\fill (s2) circle (1mm);
\coordinate[vtx] (s3) at (0.5, -2.958);
\draw (s3) node[below] {$s_3$};
\fill (s3) circle (1mm);

\coordinate[vtx] (r1) at (0.5, 2.958);
\draw (r1) node[above] {$r_1$};
\fill (r1) circle (1mm);
\coordinate[vtx] (r2) at (-2.958, 0.5);
\draw (r2) node[left] {$r_2$};
\fill (r2) circle (1mm);
\coordinate[vtx] (r3) at (-0.5, -2.958);
\draw (r3) node[below] {$r_3$};
\fill (r3) circle (1mm);

\draw (v)  -- (u);
\draw (u) -- (s1);
\draw (u) -- (s2);
\draw (u) -- (s3);
\draw (v) -- (r1);
\draw (v) -- (r2);
\draw (v) -- (r3);

\coordinate[vtx] (t1) at (70:3);
\draw (t1) node[above=2pt] {$t_1$};
\fill (t1) circle (1mm);
\coordinate[vtx] (t2) at (50:3);
\draw (t2) node[above right] {$t_2$};
\fill (t2) circle (1mm);
\coordinate[vtx] (y1) at (60:3.7);
\draw (y1) node[above] {$y_1$};
\fill (y1) circle (1mm);

\draw (y1) -- (t1);
\draw (y1) -- (t2);

\coordinate[vtx] (t3) at (30:3);
\draw (t3) node[above right] {$t_3$};
\fill (t3) circle (1mm);
\coordinate[vtx] (t4) at (10:3);
\draw (t4) node[right=2pt] {$t_4$};
\fill (t4) circle (1mm);
\coordinate[vtx] (y2) at (20:3.7);
\draw (y2) node[right] {$y_2$};
\fill (y2) circle (1mm);

\draw (y2) -- (t3);
\draw (y2) -- (t4);

\draw (2:3.3) node {$\vdots$};

\coordinate[vtx] (t5) at (-10:3);
\draw (t5) node[right=2pt] {$t_{2\l -1}$};
\fill (t5) circle (1mm);
\coordinate[vtx] (t6) at (-30:3);
\draw (t6) node[below right] {$t_{2\l}$};
\fill (t6) circle (1mm);
\coordinate[vtx] (y3) at (-20:3.7);
\draw (y3) node[right] {$y_\l  $};
\fill (y3) circle (1mm);

\draw (y3) -- (t5);
\draw (y3) -- (t6);

\coordinate[vtx] (q) at (-50:3);
\draw (q) node[below right] {$q$};
\fill (q) circle (1mm);
\draw (s1) -- (q);

\coordinate[vtx] (z1) at (150:3);
\draw (z1) node[above left] {$z_{1, 1}$};
\fill (z1) circle (1mm);
\coordinate[vtx] (z2) at (120:3);
\draw (z2) node[above left] {$z_{\l, k}$};
\fill (z2) circle (1mm);
\draw (135:3.6) node {$\iddots$};

\coordinate[vtx] (x1) at (240:3);
\draw (x1) node[below left] {$x_1$};
\fill (x1) circle (1mm);
\coordinate[vtx] (x2) at (210:3);
\draw (x2) node[below left] {$x_\l$};
\fill (x2) circle (1mm);
\draw (222:3.3) node {$\ddots$};

\draw [decorate,
    decoration = {calligraphic brace, mirror}, thick] (t1) --  (t2) node[pos=0.6, below left, black]{$g$};

\draw [decorate,
    decoration = {calligraphic brace, mirror}, thick] (t3) --  (t4) node[pos=0.7, left=2pt, black]{$g$};

\draw [decorate,
    decoration = {calligraphic brace, mirror}, thick] (t5) --  (t6) node[pos=0.4, left=1.5pt, black]{$g$};

\draw [decorate,
    decoration = {calligraphic brace, mirror}, thick] (t2) --  (t3) node[pos=0.7, left=2pt, black]{$k'$};

\draw [decorate,
    decoration = {calligraphic brace, mirror}, thick] (s1) --  (r1) node[pos=0.5, below, black]{$\l'$};

\draw [decorate,
    decoration = {calligraphic brace, mirror}, thick] (s2) --  (r2) node[pos=0.5, right, black]{$\l'$};

\draw [decorate,
    decoration = {calligraphic brace, mirror}, thick] (s3) --  (r3) node[pos=0.45, above, black]{$\l'$};

\draw [decorate,
    decoration = {calligraphic brace, mirror}, thick] (t6) --  (q) node[pos=0.15, left=2pt, black]{$k'$};

\end{tikzpicture}
&
\begin{tikzpicture}
\draw (0,0) circle (3);
\coordinate[vtx] (u) at (-0.5, 0);
\draw (u) node[below left] {$u$};
\fill (u) circle (1mm);

\coordinate[vtx] (v) at (0.5, 0);
\draw (v) node[below left] {$v$};
\fill (v) circle (1mm);

\coordinate[vtx] (s1) at (-0.5, 2.958);
\draw (s1) node[above] {$s_1$};
\fill (s1) circle (1mm);
\coordinate[vtx] (s2) at (-2.958, -0.5);
\draw (s2) node[left] {$s_2$};
\fill (s2) circle (1mm);
\coordinate[vtx] (s3) at (0.5, -2.958);
\draw (s3) node[below] {$s_3$};
\fill (s3) circle (1mm);

\coordinate[vtx] (r1) at (0.5, 2.958);
\draw (r1) node[above] {$r_1$};
\fill (r1) circle (1mm);
\coordinate[vtx] (r2) at (-2.958, 0.5);
\draw (r2) node[left] {$r_2$};
\fill (r2) circle (1mm);
\coordinate[vtx] (r3) at (-0.5, -2.958);
\draw (r3) node[below] {$r_3$};
\fill (r3) circle (1mm);

\draw (v)  -- (u);
\draw (u) -- (s1);
\draw (u) -- (s2);
\draw (u) -- (s3);
\draw (v) -- (r1);
\draw (v) -- (r2);
\draw (v) -- (r3);

\coordinate[vtx] (t1) at (70:3);
\draw (t1) node[above=2pt] {$t_1$};
\fill (t1) circle (1mm);
\coordinate[vtx] (t2) at (50:3);
\draw (t2) node[above right] {$t_2$};
\fill (t2) circle (1mm);
\coordinate[vtx] (y1) at (60:3.7);
\draw (y1) node[above] {$y_1$};
\fill (y1) circle (1mm);

\draw (y1) -- (t1);
\draw (y1) -- (t2);

\coordinate[vtx] (t3) at (30:3);
\draw (t3) node[above right] {$t_3$};
\fill (t3) circle (1mm);
\coordinate[vtx] (t4) at (10:3);
\draw (t4) node[right=2pt] {$t_4$};
\fill (t4) circle (1mm);
\coordinate[vtx] (y2) at (20:3.7);
\draw (y2) node[right] {$y_2$};
\fill (y2) circle (1mm);

\draw (y2) -- (t3);
\draw (y2) -- (t4);

\draw (2:3.3) node {$\vdots$};

\coordinate[vtx] (t5) at (-9:3);
\draw (t5) node[right=2pt] {$t_{2\l -1}$};
\fill (t5) circle (1mm);
\coordinate[vtx] (t6) at (-28:3);
\draw (t6) node[below right] {$t_{2\l}$};
\fill (t6) circle (1mm);
\coordinate[vtx] (y3) at (-19.5:3.7);
\draw (y3) node[right] {$y_\l  $};
\fill (y3) circle (1mm);

\draw (y3) -- (t5);
\draw (y3) -- (t6);

\coordinate[vtx] (q) at (-45:3);
\draw (q) node[below right] {$q$};
\fill (q) circle (1mm);
\draw (s1) -- (q);

\coordinate[vtx] (t1t) at (110:3);
\draw (t1t) node[below right] {$\t{t}_{1}$};
\fill (t1t) circle (1mm);
\coordinate[vtx] (t2t) at (130:3);
\draw (t2t) node[below right] {$\t{t}_{2}$};
\fill (t2t) circle (1mm);
\coordinate[vtx] (zt11) at (120:3.7);
\draw (zt11) node[above] {$\t{z}_{1, 1}$};
\fill (zt11) circle (1mm);
\draw (zt11) -- (t1t);
\draw (zt11) -- (t2t);

\coordinate[vtx] (z1) at (160:3);
\draw (z1) node[above left] {$z_{1, 1}$};
\fill (z1) circle (1mm);
\coordinate[vtx] (z2) at (140:3);
\draw (z2) node[above left] {$z_{\l, k}$};
\fill (z2) circle (1mm);
\draw (149:3.6) node {$\iddots$};

\coordinate[vtx] (x1) at (250:3);
\draw (x1) node[below left] {$x_1$};
\fill (x1) circle (1mm);
\coordinate[vtx] (x2) at (230:3);
\draw (x2) node[below left] {$x_\l$};
\fill (x2) circle (1mm);
\draw (236:3.3) node {$\ddots$};

\coordinate[vtx] (t3t) at (200:3);
\draw (t3t) node[right=2pt] {$\t{t}_{3}$};
\fill (t3t) circle (1mm);
\coordinate[vtx] (t4t) at (220:3);
\draw (t4t) node[above right] {$\t{t}_{4}$};
\fill (t4t) circle (1mm);
\coordinate[vtx] (zt21) at (210:3.7);
\draw (zt21) node[below] {$\t{z}_{2, 1}$};
\fill (zt21) circle (1mm);
\draw (zt21) -- (t3t);
\draw (zt21) -- (t4t);

\coordinate[vtx] (t5t) at (290:3);
\draw (t5t) node[above] {$\t{t}_{5}$};
\fill (t5t) circle (1mm);
\coordinate[vtx] (t6t) at (305:3);
\draw (t6t) node[above] {$\t{t}_{6}$};
\fill (t6t) circle (1mm);
\coordinate[vtx] (zt31) at (297.5:3.7);
\draw (zt31) node[below] {$\t{z}_{3, 1}$};
\fill (zt31) circle (1mm);
\draw (zt31) -- (t5t);
\draw (zt31) -- (t6t);

\draw [decorate,
    decoration = {calligraphic brace, mirror}, thick] (t1) --  (t2) node[pos=0.6, below left, black]{$g$};

\draw [decorate,
    decoration = {calligraphic brace, mirror}, thick] (t3) --  (t4) node[pos=0.7, left=2pt, black]{$g$};

\draw [decorate,
    decoration = {calligraphic brace, mirror}, thick] (t5) --  (t6) node[pos=0.4, left=1.5pt, black]{$g$};

\draw [decorate,
    decoration = {calligraphic brace, mirror}, thick] (t2) --  (t3) node[pos=0.7, left=2pt, black]{$k'$};

\draw [decorate,
    decoration = {calligraphic brace, mirror}, thick] (s1) --  (r1) node[pos=0.5, below, black]{$\l'$};

\draw [decorate,
    decoration = {calligraphic brace, mirror}, thick] (s2) --  (r2) node[pos=0.5, right, black]{$\l'$};

\draw [decorate,
    decoration = {calligraphic brace, mirror}, thick] (s3) --  (r3) node[pos=0.45, above, black]{$\l'$};

\draw [decorate,
    decoration = {calligraphic brace, mirror}, thick] (t6) --  (q) node[pos=0.15, left=2pt, black]{$k'$};

\end{tikzpicture}
\end{tabular}
\caption{Construction of graphs $G^{(1)}$ and $G^{(2)}$ on the left and $G^{(3)}$ and $G^{(4)}$ on the right}
\label{fig:graphG}
\end{figure}

\noindent
Since the above construction describes the graphs $G$ fully, we will now enumerate several important properties of these graphs which will be used in the subsequent proofs. These properties should also help in understanding why the graphs $G, H$ have been constructed in this way. 

\begin{enumerate}[label={(P\arabic*):}]
    \item The graph $G|_{[m]}$ has only one cycle shorter than $g$ $-$ this cycle contains the vertices $[s_1, t_1], y_1, [t_2, t_3]$, $y_2, \dots, y_\l  , [t_{2\l}, q], s_1$ and has length $k\cdot (l+1)-2\geq lk$.
    \item The smallest distance, in the graph $G$ between any two elements of the set $\x\cup\y\cup \z\cup \s\cup \r$ is at least $\l'=\l-1$.
    \item Let $P=\{y_1, \dots, y_\l  , u, v\}\cup \bigcup_{i=1}^3 [s_i, r_i]$ in the cases (C1) and (C2) and $P=\{y_1, \dots, y_\l,$ $u, v\}\cup \bigcup_{i=1}^3 [s_i, r_i]\cup $ $\{\t z_{1, 1}, \t z_{2, 1}, \t z_{3, 1}\}$ in the cases (C3) and (C4). If $S_1$ is a subset of $P$, we note that $G|_{V(G)\backslash S_1}$ is Wilsonian. In other words, removing a vertices of $P$ from $G$ does not change the fact $G$ is Wilsonian.
    \item The number of edges of $G$ is $|E(G)|\leq|V(G)|+ 4\l+15\leq |V(G)|+5\l $. Moreover, for any $J\subseteq V(G)$, the subgraph $G|_J$ has at most $|J|+5\l $ edges.
\end{enumerate}

\subsection{Showing that $u$ and $v$ are $(\GR, \HR)$-exchangeable from the identity permutation}\label{subsec:exchangeable}

Now, we will prove the most important property of graphs $(\GR, \HR)$ --- that the vertices $u$ and $v$ are $(\GR, \HR)$-exchangeable starting from the identity permutation. 
\begin{lemma}\label{lemma:exchangeability12}
For $\rho\in \{1, 2\}$, the vertices $u$ and $v$ are $(\GR, \HR)$-exchangeable from the identity permutation.
\end{lemma}
\begin{proof}
Again, for simplicity of notation, we omit the index $(\rho)$ in this proof. The general strategy of the proof of this lemma hinges on repeated applications of Wilson's theorem, exploiting the fact that $H|_{\{w\}\cup B_H-\{v\}}$ is a star centered at $w$.

Let us fix two vertices $a_1$ and $a_2$ in the same partite set of $G$, and let $\tau:V(G)\to V(H)$ be an arbitrary arrangement. If $\tau':V(G)\to V(H)$ is any arrangement differing from $\tau$ only on elements of $\tau^{-1}(\{w\}\cup B_H-\{v\})$, and if the graph $G|_{\tau^{-1}(\{w\}\cup B_H-\{v\})}$ is Wilsonian and contains $a_1, a_2$, then $\tau$ is in the same connected component as either $\tau'$ or $\tau'\circ \tr{a_1}{a_2}$. To see why this is true, note that the arrangements $\tau'$ and $\tau'\circ \tr{a_1}{a_2}$ are not in the same connected component, due to Proposition \cref{prop:parityobstruction}. Since Wilson's theorem implies the graph $\FS(G|_{\tau^{-1}(\{w\}\cup B_H-\{v\})}, H_{\{w\}\cup B_H-\{v\}})$ has exactly two connected components, one of $\tau'$ and $\tau'\circ \tr{a_1}{a_2}$ must be in the same connected component as $\tau$.

We will the above apply simple fact with $a_1, a_2$ being vertices satisfying $d(a_1, q)=4$ and $d(a_2, q)=2$. The fact that $G|_{\tau^{-1}(\{w\}\cup B_H-\{v\})}$ is Wilsonian will follow from $\taui(A_H-\{w\}\cup \{v\})\subseteq P$ and property (P3) of the graph $G$. Hence, whenever we can ensure that the vertices of $\taui(A_{H}-\{w\}\cup \{v\})$ all lie in $P$, we can obtain any arrangement of the remaining vertices, up to possibly transposing $a_1$ and $a_2$ (we will omit this clause in the future, but it will be understood whenever we apply Wilson's theorem).

We will think of arrangements $\tau:V(G)\to V(H)$ as placing the labels corresponding to vertices of $H$ onto the vertices of $G$ so that the label $p\in V(H)$ is placed onto the vertex $\taui(p)\in V(G)$. Moreover, when we refer to the $(G, H)$-friendly swap $ab$, it will be understood the swap involves labels $a, b\in V(H)$. Finally, the reader is encouraged to draw out the following steps which transform $\id$ into $\tr{u}{v}$.

Let $\tau_1$ be an arrangement in which $\taui_1(z_{1, 1}), \dots, \taui_1(z_{1, k-1})$ appear consecutively in clockwise order on $[s_1, t_1]$. By Wilson's theorem, there exists a sequence of $(G, H)$-friendly swaps transforming $\id$ to $\tau_1$.

Let $\tau_2$ be an arrangement obtained from $\tau_1$ by performing swaps $y_1z_{1, k-1}, y_1z_{1, k-2}, \dots, y_1z_{1, 1}$. Note that $\tau_2$ satisfies $\taui_2(y_1)=s_1$, and that it can be obtained from $\id$ by a sequence of $(G, H)$-friendly swaps. In essence, we can think of this sequence as moving the label $y_1\in V(H)$ from the vertex $y_1\in V(G)$ to the vertex $s_1\in V(G)$.

Let $\tau_3$ be the arrangement which differs from $\tau_2$ only on values $\taui_2(\{z_{2, 1}, \dots, z_{2, k-1}\})\cup \{y_1\} \cup [t_2, t_3]$, and in which the vertices $\taui_3(z_{2, 2}), \dots, \taui_3(z_{2, k-1})$ appear consecutively in clockwise order on $[t_2, t_3]$ and which satisfies $\taui_3(z_{2, 1})=y_1$. By Wilson's theorem, there exists a sequence of $(G, H)$-friendly swaps transforming $\tau_2$ to $\tau_3$.

Let $\tau_4$ be an arrangement obtained from $\tau_3$ by performing swaps $y_2z_{2, k-1}, y_2z_{2, k-2}, \dots, y_2z_{2, 1}$. Note that $\tau_4$ satisfies $\taui_4(y_2)=y_1$. Intuitively, $\tau_4$ is obtained from $\tau_2$ by performing a sequence of swaps which moves the label $y_2\in V(H)$ from the vertex $y_2\in V(G)$ to the vertex $y_1\in V(G)$.

Repeatedly applying sequences of swaps analogous to those taking $y_1$ to $s_1$ or $y_2$ to $y_1$, one can arrive at an arrangement $\tau_5$ in which vertices $\taui_5(y_1), \dots, \taui_5(y_\l  )$ appear consecutively in clockwise order on $[s_1, r_1]$.

Let $\tau_6$ be the arrangement obtained from $\tau_5$ by performing the swaps $vy_\l  , \dots, vy_1$ which satisfies $\taui_6(v)=s_1$ and $\taui_6(y_\l  )=v$.

Let $\tau_7$ be the arrangement which differs from $\tau_6$ only on values $\taui_6(\{x_{1}, \dots, x_{l}\}) \cup [s_3, r_3]$, and in which the vertices $\taui_6(x_1), \dots, \taui_6(x_{\l})$ appear consecutively in clockwise order on $[s_3, r_3]$. By Wilson's theorem, there exists a sequence of $(G, H)$-friendly swaps transforming $\tau_6$ to $\tau_7$.

Let $\tau_8$ be the arrangement obtained from $\tau_7$ by performing the swaps $ux_1, \dots, ux_{\l-1}$ which satisfies $\taui_8(x_1)=u$ and $\taui_8(u)$ is adjacent to the vertex $r_3\in V(G)$.

Let $\tau_9=\tr{z_{\l, 1}}{x_1}\circ \tau_8$. By Wilson's theorem, there exists a sequence of $(G, H)$-friendly swaps transforming $\tau_8$ to $\tau_9$. The arrangement $\tau_9$ satisfies $\taui_9(z_{\l, 1})=u$.

Let $\tau_{10}$ be the arrangement arising from $\tau_9$ after performing swaps $y_{\l}z_{\l, 1}, y_\l  v$. In $\tau_{10}$, the vertices $\taui_{10}(y_\l  ), \taui_{10}(y_1),$ $\taui_{10}(y_2), \dots, \taui_{10}(y_{\l-1})$ appear consecutively in clockwise order on $[s_1, r_1]$. Also, $\taui_{10}(v)=u$.

Let $\tau_{11}$ be the arrangement which differs from $\tau_{10}$ only on values $\taui_{10}(\{x_{1}, x_{2}\}) \cup \{r_3, v\}$, and which satisfies $\taui_{11}(x_1)=r_3, \taui_{11}(x_2)=v$. By Wilson's theorem, there exists a sequence of $(G, H)$-friendly swaps transforming $\tau_{10}$ to $\tau_{11}$.

Let $\tau_{12}$ be the arrangement obtained from $\tau_7$ by performing the swap $ux_1$, satisfying $\tau_{12}(u)=v$ and $\tau_{12}(v)=u$. It remains to show that all other labels can be returned to their original positions. By applying the reverse of the process which brought the labels $y_1, \dots, y_{\l-1}$ onto the vertices $[s_1, r_1]$, it is simple to see that there exists an arrangement $\tau_{13}$ which satisfies $\tau_{13}|_{\{u, v\}}=\tau_{12}|_{\{u, v\}}$ and $\tau_{13}(y_i)=y_i$ for $i\in [l-1]$.

Let $\tau_{14}$ be the arrangement which differs from $\tau_{13}$ only on values $\taui_{13}(\{z_{\l, 1}, \dots, z_{\l, k-1}\}) \cup [t_{2\l}, q]\cup \{y_\l \}$, and in which vertices $\taui_{14}(z_{\l, k-1}), \dots, \taui_{14}(z_{\l, 2})$ appear consecutively in clockwise order on $[t_{2\l}, q]$ and $\taui_{14}(z_{\l, 1})=y_\l $. By Wilson's theorem, there exists a sequence of $(G, H)$-friendly swaps transforming $\tau_{13}$ to $\tau_{14}$.

Let $\tau_{15}$ be the arrangement obtained from $\tau_{14}$ by performing the swaps $y_\l z_{\l, k-1}, \dots, y_\l z_{\l, 1}$ which satisfies $\taui_{15}(y_\l )=y_\l $. 

Finally, note that $\taui_{15}$ fixes all elements of $A_H-\{w\}\cup \{v\}$, except $u$ and $v$ which are swapped. Hence, using Wilson's theorem, we conclude that there exists a sequence of $(G, H)$-friendly swaps transforming $\tau_{14}$ to $\tr{u}{v}$. However, note that we actually obtain $\tau_{15}=\tr{u}{v}$ only up to possibly transposing $a_1$ and $a_2$, even though we did not explicitly state this after every application of Wilson's theorem in the proof.

Hence, we either have $\tau_{15}=\tr{u}{v}$ or $\tau_{15}=\tr{u}{v}\circ \tr{a_1}{a_2}$. However, note that it is not possible that we obtained the latter arrangement, starting from the identity permutation, due to parity constraints. Therefore, we conclude $\tau_{15}=\tr{u}{v}$ and the proof is complete. 
\end{proof}

\begin{lemma}\label{lemma:exchangeability34}
For $\rho\in \{3, 4\}$, the vertices $u$ and $v$ are $(\GR, \HR)$-exchangeable from the identity permutation.
\end{lemma}
\begin{proof}
In the cases (C3) and (C4), the identity permutation satisfies $A_H-\{w\}\cup \{v\}\subseteq P$, and hence there exists a sequence of $(\GR, \HR)$-friendly swaps taking $\id$ to $\tau_0$ where $\tau_0$ has the property that $\taui_0(\t z_{1, 1}), \dots, \taui_0(\t z_{1, k-2})$ are consecutively arranged in counterclockwise order on the big cycle, $\taui_0(\t z_{1, 1})$ being the neighbor of $s_1$ and $\taui_0(\t z_{1, k-2})=\t t_1$. Moreover, assume $\taui_0(\t z_{1, k-1})=\t z_{1, 1}$. Using Wilson's theorem as described above, we conclude that $\tau_0$ can be obtained from $\id$ by a sequence of $(\GR, \HR)$-friendly swaps. If we consider the sequence of swaps $s_1\t z_{1, 1}, s_1\t z_{1, 2}, \dots, s_1\t z_{1, k-1}$ applied on $\tau_0$, we obtain an arrangement $\tau_1$ having $\tau_1(s_1)=\t z_{1, 1}$ and $\tau_1(\t z_{1, 1})=s_1$. Finally, using one more application of Wilson's theorem, there exists a sequence of $(\GR, \HR)$-friendly swaps which transforms $\tau_1$ into an arrangement $\tau_2=\tr{s_1}{\t z_{1, 1}}$. Applying the same argument to $s_2, s_3$, we obtain an arrangement $\tau_3$ which transposes elements $s_i$ and $\t z_{i, 1}$, and leaves all other elements fixed. Since the elements $u$ and $v$ were not involved in any of the swaps transforming $\id$ to $\tau_3$, it is not hard to see that the same sequence of swaps transforms $\tr{u}{v}$ to $\tau_3\circ \tr{u}{v}$.

The important thing to note now is that eliminating $s_1, s_2, s_3$ from $\HR$ and $\t z_{1, 1}, \t z_{2, 1}, \t z_{3, 1}$ from $\GR$ essentially reduces to the case $\rho\in \{1, 2\}$. Let us now explain why more formally. 

Note that $\GR-\{\t z_{1, 1}, \t z_{2, 1}, \t z_{3, 1}\}$ for $\rho\in \{3, 4\}$ is isomorphic, as an unlabeled graph, to $G^{(5-\rho)}$, and that $\HR-\{s_1, s_2, s_3\}$ for $\rho\in \{3, 4\}$ is isomorphic, again in the unlabeled sense, to $H^{(5-\rho)}$. In other words, there exist embeddings $\varphi_{G}:G^{(5-\rho)}\to \GR$ and $\varphi_{H}:H^{(5-\rho)}\to \HR$ which map between the arrangement $\id$ on $(G^{(5-\rho)}, H^{(5-\rho)})$ and arrangement $\tau_3$ on $(\GR, \HR)$. Alternatively, we may say that $\tau_3\circ \varphi_{G}=\varphi_{H}\circ \id$ and $\varphi_H(u)=u, \varphi_H(v)=v$. Since Lemma~\cref{lemma:exchangeability12} shows that $u$ and $v$ are $(G^{(5-\rho)}, H^{(5-\rho)})$-exchangeable from $\id$, using Proposition~\cref{prop:embedabilityimpliesexchangeability} translates to $u$ and $v$ being $(\GR, \HR)$-exchangeable from $\tau_3$. Combining the sequences obtaining transformations $\id\to \tau_3$, $\tau_3\to \tau_3\circ \tr u v$ and $\tau_3\circ \tr u v \to \tr u v$, we conclude that $\id$ and $\tr u v$ are connected in the friends-and-strangers graph $\FS(\GR, \HR)$, which implies $u, v$ are $(\GR, \HR)$-exchangeable from $\id$.
\end{proof}

\subsection{Showing that $\GR$, $\HR$ are $(q_1, \dots, q_m)$-bipartite-embeddable in $X, Y$}\label{subsec:bipartiteembeddable}

In this section, we will prove that, with high probability, the graphs $\GR|_{[m]}, \HR|_{[m]}$ are $(q_1, \dots, q_m)$-bipartite-embeddable in $(X, Y)$ with high probability. The main technical tool which allows us to show this is Lemma 4.2 from \cite{ADK}, which states the following.

\begin{lemma}[\cite{ADK}]\label{lemma:technical}
Let $G'$ and $H'$ be bipartite graphs on the vertex set $[m]$ with bipartitions $V(G')=A_{G'}\bigsqcup B_{G'}$, $V(H')=A_{H'}\bigsqcup B_{H'}$. Let $n, q_1, \dots, q_m$ be positive integers such that $Q:=q_1+\cdots+q_m\leq 2n$. For every set $J\subseteq [m]$, let $\beta(J)=|E(G'|_J)|+|E(H'|_J)|$. Choose $p\in [0, 1]$, and let $X, Y$ be independently chosen random bipartite graphs in $\G(K_{n, n}, p)$. Suppose that for every $J\subseteq [m]$ with $\beta(J)\geq 1$ we have 
\begin{equation}\label{eqn:embcondition}
    p^{\beta(J)}\prod_{j\in J}q_j\geq 3\cdot 2^{m+1}Q\log (2n).
\end{equation}
Then, the probability that the pair $(G', H')$ is $(q_1, \dots, q_m)$-bipartite-embeddable in $(X, Y)$ is at least $1-(2n)^{-Q}$.
\end{lemma}

\noindent
In Lemma~\cref{lemma:embeddability}, we check that condition \cref{eqn:embcondition} indeed holds for $\GR|_{[m]}, \HR|_{[m]}$ under the appropriate choice of values $q_1, \dots, q_m$. Let $\Gamma=N_H(u)\cup N_H(v)\cup N_G(u)\cup N_G(v)=\x\cup \y\cup \s\cup \r$. 

\begin{lemma}\label{lemma:embeddability}
Let $X, Y$ be random bipartite graphs, independently chosen from $\G(K_{n, n}, p)$, where $p$ satisfies \cref{eqn:pbound}. Moreover, suppose that $q_i=\frac{n}{2m}$ for $i\in [m]-\Gamma$ and $q_i=\frac{pn}{3\l }$ for $i\in \Gamma$. Then, for any $\rho\in \{1, 2, 3, 4\}$, the pair of graphs $(\GR|_{[m]}, \HR|_{[m]})$ is $(q_1, \dots, q_m)$-bipartite-embeddable in $(X, Y)$ with high probability.
\end{lemma}
\begin{proof}
This proof is inspired by the proof of Lemma 3.4 in \cite{ADK}, with the main difference coming only from the complexity of our construction. To simplify notation, let $G'=\GR|_{[m]}, H'=\HR|_{[m]}$, and $Q=\sum_{i=1}^m q_i$. Using Lemma~\cref{lemma:technical}, the statement about embeddability of $(G', H')$ reduces to checking that 
\[p^{\beta(J)}\prod_{j\in J}q_j\geq 3\cdot 2^{m+1}Q\log (2n),\]
for all $J$ with $\beta(J)\geq 1$. Letting $\gamma(J)=|\Gamma\cap J|$ and using $Q\leq n$, the above inequality reduces to 
\begin{align}\label{eqn:mainineq}
    p^{\beta(J)}\left(\frac{pn}{3\l }\right)^{\gamma(J)}\left(\frac{n}{2m}\right)^{|J|-\gamma(J)}\geq 3\cdot 2^{m+1}n\log 2n.
\end{align}
Since $2m>3\l $, one can bound the left hand side of equation \cref{eqn:mainineq} in the following way:
\[p^{\beta(J)}\left(\frac{pn}{3\l }\right)^{\gamma(J)}\left(\frac{n}{2m}\right)^{|J|-\gamma(J)}\geq p^{\beta(J)+\gamma(J)-2|J|}\left(\frac{p^2n}{2m}\right)^{|J|}\left(\frac{2m}{3\l }\right)^{\gamma(J)}\geq p^{\beta(J)+\gamma(J)-2|J|}\left(\frac{p^2n}{2m}\right)^{|J|}\]

\noindent
We will now bound $\beta(J)$ and $\gamma(J)$ by analyzing the structure of the graphs $G'$ and $H'$. Recall that $H'$ has a bipartition $V(H')=A_{H'}\bigsqcup B_{H'}$, where $A_{H'}=A_H-\{u\}$ and $B_{H'}=B_H-\{v\}$. With this is mind, we can express the number of edges of $H'|_J$ as the sum of degrees of vertices in one of the partite classes, $|E(H'|_J)|=\sum_{a\in J\cap A_{H'}}|N_{H'}(a)\cap J|$. Expanding out this expression gives: 
\begin{align}\label{eqn:boundHJ}
    |E(H'|_J)|=\ind_{w\in J}|N_{H'}(w)\cap J| +\sum_{a\in A_{H'}-\{w\}} |N_{H'}(a)\cap J|\leq \ind_{w\in J}|J\cap B_{H'}|+|J\cap \z|,
\end{align}
where $\ind_{w\in J}$ denotes the indicator variable which equals $1$ when $w\in J$ and $0$ otherwise.

Let $\Gamma'=(\x\cup \z\cup \s\cup \r)\cap B_{H'}$, and define $\gamma'(J)=|J\cap \Gamma'|$ and $\alpha(J)=|E(G'|_J)|$. Then, the bound \cref{eqn:boundHJ} applied to give:
\begin{align}\label{eqn:boundbetagamma}
    \beta(J)+\gamma(J)-2|J|&= (|E(G'|_J)|+|E(H'|_{J})|)+(\gamma'(J)+|J\cap (A_{H'}-\{w\})|-|J\cap \z|)-2|J|\nonumber \\
    &\leq \alpha(J)+\ind_{w\in J} |J\cap B_{H'}|+\gamma'(J)+|J\cap (A_{H'}-\{w\})|-2|J|\nonumber\\
    &\leq \alpha(J)+\gamma'(J)-|J|-\ind_{w\in J}.
\end{align}
\noindent
We now consider five cases and apply the above bounds in resolving them.

\medskip
\noindent
\textbf{Case 1:} Suppose that $|J|\geq g$. Since $|E(G')|\leq |V(G')|+5\l $, by the property (P4) of the graphs $G'$ and since $G'$ is a connected graph, we also have $\alpha(J)=|E(G'|_J)|\leq |J|+5\l $. Therefore, one can apply the bound \cref{eqn:boundbetagamma} and show \cref{eqn:mainineq} in the following way:
\[p^{\beta(J)+\gamma(J)-2|J|}\left(\frac{p^2n}{2m}\right)^{|J|}
\geq p^{\alpha(J)+\gamma'(J)-|J|}\left(\frac{p^2n}{2m}\right)^{|J|}
\geq p^{5\l +|\Gamma'|}\left(\frac{\exp\left(20(\log n)^{4/5}\right)}{2m}\right)^{g}\]
\[\geq n^{-\left(6\l+k(\l+3)+6\right)/2}\exp\left(5(\log n)^{4/5}m^{3/4}-g\log (2m)\right)\]
\[\geq \exp\left(5 (\log n)^{7/5}-\frac{3}{2}\l^2 \log n\right)
\geq \exp\left(3(\log n)^{7/5}\right)\gg 3\cdot 2^{m+1} n\log (2n).\]

\medskip
\noindent
\textbf{Case 2:} Suppose that $|J|<g$ and $\gamma'(J)> c(J)$, where $c(J)$ is the number of connected components of $J$. Since there exists exactly one cycle of length less than $g$ in $G$, we must have $\alpha(J)\leq |J|-c(J)+1$. On the other hand, since the distance between any two elements of $\Gamma'$ is at least $\l'=\l-1$, we have $|J|\geq (\gamma'(J)-c(J))\l'$. Therefore
\[p^{\beta(J)+\gamma(J)-2|J|}\left(\frac{p^2n}{2m}\right)^{|J|}\geq
p^{\alpha(J)+\gamma'(J)-|J|}\left(\frac{p^2n}{2m}\right)^{|J|}\geq 
p^{\gamma'(J)-c(J)+1}\left(\frac{\exp\left(20\log n^{4/5}\right)}{2m}\right)^{(\gamma'(J)-c(J))\l'}\]
\[\geq n^{-1/2}\left(n^{-1/2}\exp(4 \log n)\right)^{\gamma'(J)-c(J)}
\geq n^{3}\gg 3\cdot 2^{m+1}n\log(2n) .\]

\medskip
\noindent
\textbf{Case 3:} Suppose that $\gamma'(J)\leq c(J)$ and $\l\leq |J|<g$. As discussed in the previous case, we have $\alpha(J)\leq |J|-c(J)+1$, which implies 
\[
p^{\beta(J)+\gamma(J)-2|J|}\left(\frac{p^2n}{2m}\right)^{|J|}\geq p^{\alpha(J)+\gamma'(J)-|J|}\left(\frac{p^2n}{2m}\right)^{|J|}\geq
p^{\gamma'(J)-c(J)+1}\left(\frac{\exp\left(20(\log n)^{4/5}\right)}{2m}\right)^{\l}\]
\[\geq n^{-1/2}\exp\left(5\log n-(\log n)^{1/5}\log (2m)\right)\geq n^4\gg
3\cdot 2^{m+1}n\log (2n).\]

\medskip
\noindent
\textbf{Case 4:} Suppose that $\gamma'(J)\leq c(J)$, $|J|<\l$ and $w\in J$. Since $|J|$ is smaller than the girth of $G'$, $G'|_J$ must be a forest and therefore $\alpha(J)= |J|-c(J)$. Also, as $\beta(J)\geq 1$, $J$ must induce at least one edge in $G'$ or $H'$, and hence $|J|\geq 2$. Finally, since elements of $ (\Gamma'\cup \{w\})$ are at mutual distance at least $\l'$, we conclude that every connected component of $G'|_J$ contains at most one element of $\Gamma'\cup \{w\}$, implying $\gamma'(J)+1=|J\cap (\Gamma'\cup \{w\})|\leq c(J)$. Combining these facts, one has:
\[
p^{\beta(J)+\gamma(J)-2|J|}\left(\frac{p^2n}{2m}\right)^{|J|}\geq
p^{\alpha(J)+\gamma'(J)-|J|-1}\left(\frac{p^2n}{2m}\right)^{|J|}\geq
p^{\gamma'(J)-c(J)-1}\left(\frac{p^2n}{2m}\right)^{|J|}
\]
\[\geq p^{-2}\left(\frac{p^2n}{2m}\right)^2=
\frac{p^2n^2}{(2m)^2}\geq \frac{\exp\left(20(\log n)^{4/5}\right) n}{4m^2}\gg 3\cdot 2^{m+1} n\log 2n.
\]

\medskip
\noindent
\textbf{Case 5:} Suppose that $\gamma'(J)\leq c(J)$, $|J|<\l$, and $w\not \in J$. As in the previous case, we infer that $J$ is a forest, and that $\alpha(J)=|J|-c(J)$. If $J$ induces no edges on $H'$ we have \[\beta(J)+\gamma(J)-2|J|=\alpha(J)+\gamma(J)-2|J|=\gamma(J)-c(J)-|J|.\]
Note that every two elements of $\Gamma$ are at distance at least $\l'=\l-1$, which gives $\gamma(J)\leq c(J)$, and hence
\[
p^{\beta(J)+\gamma(J)-2|J|}\left(\frac{p^2n}{2m}\right)^{|J|}\geq p^{-2}\left(\frac{p^2n}{2m}\right)^2=\frac{p^2n^2}{(2m)^2}\]
\[\geq \frac{\exp\left(20(\log n)^{4/5}\right) n}{4m^2}\gg 3\cdot 2^{m+1} n\log 2n.\]
\noindent
In the case that $J$ induces an edge in $H'$ one can use the fact that $w\not\in J$ to improve the bound \cref{eqn:boundHJ}
\begin{align*}
    \beta(J)+\gamma(J)-2|J|&= \alpha(J)+\gamma(J)+|J\cap \z|-2|J|\\
    &=\gamma'(J)-c(J)+|J\cap A_{H'}|-|J|=\gamma'(J)-c(J)-|J\cap B_{H'}|.
\end{align*}
Since $|J|< \l$ and the distance between any two vertices of $(A_{H'}-\{w\})\cup \z$ is at least $\l'=\l-1$, we conclude that $c(J)\geq |J\cap(A_{H'}\cup \z)|$. Noting that $J$ intersects both $A_{H'}$ and $B_{H'}$, we see that $c(J)\geq \gamma'(J)+1$ and $|J\cap B_{H'}|\geq 1$. Hence, $\gamma'(J)-c(J)-|J\cap B_{H'}|\leq -2$ and so
\[
p^{\beta(J)+\gamma(J)-2|J|}\left(\frac{p^2n}{2m}\right)^{|J|}\geq p^{-2}\left(\frac{p^2n}{2m}\right)^2=\frac{p^2n^2}{(2m)^2}\]
\[\geq \frac{\exp\left(20(\log n)^{4/5}\right) n}{4m^2}\gg 3\cdot 2^{m+1} n\log 2n,\]
which completes the proof of Lemma~\cref{lemma:embeddability}.
\end{proof}

\subsection{Constructing $\psi_G$ and $\psi_H$}\label{subsec:constructingembeddings}

Before showing how to construct $\psi_G, \psi_H$ and completing the proof, we have one additional step. Namely, in the construction of $\psi_G, \psi_H$, we will need that the arrangement $\sigma$ sends approximately the same number of vertices from $A_X$ to $A_Y$ and $B_Y$. More precisely, we say that an arrangement $\sigma:V(X)\to V(Y)$ is \dfn{balanced} if $2n/3\geq |\sigma(A_X)\cap A_Y|\geq n/3$. Otherwise, $\sigma$ is \dfn{unbalanced}. Although it may seem at first that this definition is not symmetric in $A_X$ and $B_X$, the symmetry comes from the fact that $|\sigma(A_X)\cap A_Y|=n-|\sigma(B_X)\cap A_Y|=|\sigma(B_X)\cap B_Y|=n-|\sigma(A_X)\cap B_Y|$. Together with Proposition~\cref{prop:exchangeabilityfromdifferentarrangements}, the following result essentially shows that we may, without loss of generality, assume that we are exchanging $u_0$ and $v_0$ starting from a balanced arrangement. Although Proposition~\cref{prop:balancing} is stated as a result about exchangeability in various arrangements, the underlying argument is really about large matchings in random bipartite graphs. 

\begin{proposition}\label{prop:balancing}
Let $X, Y$ be random bipartite graphs, independently chosen form $\G(K_{n, n}, p)$. With high probability, the following statement holds: for any unbalanced arrangement $\sigma:V(X)\to V(Y)$ and any vertices $u_0, v_0\in V(Y)$, there exists a sequence of $(X, Y)$-friendly swaps, not involving the vertices $u_0$ and $v_0$, which transforms the arrangement $\sigma$ into a balanced arrangement $\sigma':V(X)\to V(Y)$.
\end{proposition}
\begin{proof}
The first step of the proof will be to show that for a fixed permutation $\sigma:V(X)\to V(Y)$, the required sequence exists with probability at least $1-e^{-\Omega(p^2n^2)}$. After showing this, a simple application of the union bound will suffice to complete the proof.

Having fixed an unbalanced arrangement $\sigma:V(X)\to V(Y)$, we will consider the case when $|\sigma(A_X)\cap A_Y|> 2n/3$. Note that the other case, when $|\sigma(A_X)\cap A_Y|< n/3$, can be reduced to the first one by exchanging labels $A_Y$ and $B_Y$. The main idea of the proof will be to show that, with high probability, there exist many pairs of vertices $a\in A_Y, b\in B_Y$ with $\sigma(a)\in A_X, \sigma(b)\in B_X$ which can be swapped using a $(X, Y)$-friendly swap, since performing roughly $n/3$ such swaps will transform $\sigma$ into a balanced arrangement $\sigma'$. We will then show that choosing many such disjoint pairs $(a, b)$ corresponds to finding a large matching in a certain random bipartite graph, which will in turn be accomplished using K\"onig's theorem.

Before passing to the proof, let us introduce some additional notation. Define $C_Y=\sigma(A_X)\cap A_Y$ and $D_Y=\sigma(B_X)\cap B_Y$, and let $C_X, D_X$ be the preimages of these sets under $\sigma$, $C_X=\sigmai(C_Y), D_X=\sigmai(D_Y)$. Also, let us denote the size of each of the sets $C_X, D_X, C_Y, D_Y$ by $n'$, where $n'\geq 2n/3$. 

For a fixed permutation $\sigma:V(X)\to V(Y)$ and random $X, Y\sim \G(n, p)$, we construct a bipartite graph $Z$ with the bipartition $V(Z)=C_Y\bigsqcup D_Y$ such that $cd\in E(Z)$ if and only if $cd\in E(Y)$ and $\sigmai(c)\sigmai(d)\in E(X)$. Since $\sigma$ is fixed, the events $cd\in E(Z)$ are mutually independent for all $c\in C_Y, d\in D_Y$, and they occur with probability $p^2$. Hence, $Z$ has a distribution of a random bipartite graph $\G(K_{n', n'}, p^2)$.

Suppose that $Z$ has a matching $M=\{c_1d_1, \dots, c_kd_k\}$ of size $k>n/2$. Assuming $u_0, v_0\not\in\{c_1, d_1, \dots, c_{n/3}, d_{n/3}\}$, we see that performing the $(X, Y)$-friendly swaps $c_1d_1, \dots, c_{n/3}d_{n/3}$ gives a sequence of swaps which transforms $\sigma$ into a balanced arrangement $\sigma'$. Hence, it suffices to show that the size of the maximum matching of $Z$, denoted by $\nu(Z)$, is at least $n/2$ with probability at least $1-e^{-\Omega(p^2n^2)}$.

By K\"onig's theorem, the size of the largest matching in $Z$ is equal to the size of the smallest vertex cover of $Z$, which is denoted by $\tau(Z)$. Therefore, we will begin by bounding the probability that a fixed set $V$, of size $n/2$, is a vertex cover of $Z$. Note that a set $V$ covers at most $|V|n'$ edges of the graph $K_{n', n'}$. Hence, the probability that $V$ is a vertex cover of $Z$ is equal to the probability that none of the remaining $n'(n'-|V|)$ edges exist in $E(Z)$. In other words, we have;
\begin{equation*}
    \Pb[V\text{ is a vertex cover of } Z]\leq (1-p^2)^{n'(n'-|V|)}\leq e^{-\Omega( p^2 n^2)}.
\end{equation*}
Hence, the probability that there exists a vertex cover of $Z$ of size $\leq n/2$ is at most 
\begin{equation*}
    \Pb[\nu(Z)\leq n/2]\leq \binom{n}{n/2}e^{-\Omega(p^2 n^2)}\leq 2^ne^{-\Omega(p^2n^2)}\leq e^{-\Omega( p^2n^2)}.
\end{equation*}

\noindent
The above computation shows that the probability that the required sequence exists for a fixed choice of $\sigma$ is at least $1-e^{-\Omega(p^2n^2)}$. Hence, a union bound over all choices of $\sigma$ shows that the statement of the proposition fails with probability $\leq n! e^{-\Omega(p^2n^2)}=o(1)$, as $n\to \infty$. This completes the proof of the Proposition~\cref{prop:balancing}.
\end{proof}

\noindent
The last prerequisite we need for the proof is the way to construct $\psi_G$ and $\psi_H$ when $\GR|_{[m]}, \HR|_{[m]}$ are $(q_1, \dots, q_m)$-bipartite-embeddable in $X, Y$.

\begin{proposition} \label{prop:bipartiteembeddabilityimpliesembeddability}
Let $X, Y$ be random bipartite graphs, independently chosen from $\G(K_{n, n}, p)$. With high probability, the following statement holds: For any two vertices $u_0, v_0\in V(Y)$ satisfying $u_0v_0\in E(Y), \sigmai(u_0)\in A_X, \sigmai(v_0)\in A_Y$ and every balanced $\sigma:V(X)\to V(Y)$ satisfying the assumptions of the case (C$\rho$), for $\rho\in \{1, 2, 3, 4\}$, there exist embeddings $\psi_G:V(\GR)\to V(\t X)$ and $\psi_H:V(\HR)\to V(Y)$ with $\sigma\circ \psi_G=\psi_H\circ \id$ and $\psi_H(u)=u_0, \psi_H(v)=v_0$.
\end{proposition}
\begin{proof}
The main idea of the proof is to carefully choose the sets $V_1, \dots, V_m$ such that the adjacency of the corresponding vertices of $X, Y$ with $u_0, v_0, u_0', v_0'$ is guaranteed by the choice of $V_1, \dots, V_m$. For the sake of concreteness, we will describe in detail how these sets are chosen only in the case $\rho=1$, as the choices are made in a completely analogous way in all other cases. 

From the condition $\sigma\circ \psi_G=\psi_H\circ \id$ and $\psi_H(u)=u_0, \psi_H(v)=v_0$, we must have $\psi_G(u)=\sigmai(u_0)=u_0'$ and $\psi_G(v)=v_0'$.

Since $p\gg \frac{\log n}{n}$, we have that, with high probability, all vertices of $X, Y$ have degrees at least $2pn/3$, which we assume to be the case. For $i\in N_{G^{(1)}}(u)$, we choose $V_i$ of size $\frac{pn}{3\l }$ to lie in $\sigma(N_X(u_0'))\cap B_Y$, and for $i\in N_{G^{(1)}}(v)$, we choose $V_i$ of size $\frac{pn}{3\l }$ to lie in $\sigma(N_X(v_0'))\cap B_Y$. By assumption $\sigma$ majority-maps $N_X(u_0')$ and $N_X(v_0')$ to $B_Y$, so the sizes of the intersections $\sigma(N_X(u_0'))\cap B_Y$ and $\sigma(N_X(v_0'))\cap B_Y$ are at least $\frac{pn}{3}$. Hence, one can indeed choose $l$ disjoint sets $V_i$ of size $\frac{pn}{3\l }$ satisfying the required properties.

Similarly, for $i\in N_{H^{(1)}}(u)$ we choose a set $V_i\subset N_Y(u_0)\cap \sigma(A_X)$ of size $\frac{pn}{3\l }$ and for $i\in N_{H^{(1)}}(v)$ we choose $V_i\subset N_Y(v)\cap \sigma(A_X)$, again of size $\frac{pn}{3\l }$. For vertices $i\in [m]$ which are not neighbors of $u, v$ in either ${G^{(1)}}$ or ${H^{(1)}}$, say for $i\in A_G\cap A_H-\Gamma$, we choose a set $V_i\subset \sigma(A_X)\cap A_Y$, of size $\frac{n}{4m}$. Since $\sigma$ is balanced, we have $|\sigma(A_X)\cap A_Y|\geq n/3$ and hence the sets $V_i$ for $i\in A_G\cap A_H-\Gamma$ can be chosen to be disjoint from all previously chosen sets. The sets corresponding to vertices $i\in A_G\cap B_H, B_G\cap A_H, B_G\cap B_H$ are chosen analogously. 

From our choice of $V_1, \dots, V_m$ it is clear that this list is admissible, and that $|V_i|=q_i$, where $q_i$ is defined in the statement of Lemma~\cref{lemma:embeddability}. Hence, by Lemma~\cref{lemma:embeddability}, we conclude there are embeddings $\psi_G:[m]\to V(X), \psi_H:[m]\to V(Y)$. By the choice of $V_i$, it is clear that $\psi_H(i)$ is adjacent to $u_0$ whenever $iu\in E(H)$, etc. Furthermore, since $\sigmai(u_0)\sigmai(v_0)\in E(\t X)$, we conclude that setting $\psi_H(u)=u_0, \psi_H(v)=v_0, \psi_G(u)=\sigmai(u_0), \psi_G(v)=\sigmai(v_0)$ extends $\psi_G, \psi_H$ to the embeddings $\psi_G:V(\GR)\to V(\t X)$ and $\psi_H:V(\HR)\to V(Y)$, thus completing the proof. \end{proof}

\subsection{Completing the proof}\label{subsec:finalproof}

Having presented all prerequisites in the previous sections, we now put everything together and present the whole proof of Proposition~\cref{prop:bipartiteexchangeability}, which implies Theorem~\cref{thm:randombipartiteconnectivity}.

\begin{proof}[Proof of Proposition~\cref{prop:bipartiteexchangeability}.]

Since, with high probability, any arrangement $\sigma$ can be transformed into a balanced arrangement without ever swapping $u_0, v_0$, Proposition~\cref{prop:exchangeabilityfromdifferentarrangements} shows that it suffices to check the statement of Proposition~\cref{prop:bipartiteexchangeability} only for balanced arrangements $\sigma$. If $\sigma$ satisfies the assumptions of the case (C$\rho$), for some $\rho\in \{1, 2, 3, 4\}$, combining Proposition~\cref{prop:bipartiteembeddabilityimpliesembeddability} and Lemma~\cref{lemma:embeddability}, one concludes that, with high probability, there exist embeddings $\psi_G:V(\GR)\to V(X), \psi_H:V(\HR)\to V(Y)$ with the properties required by Proposition~\cref{prop:embedabilityimpliesexchangeability}. Since Lemmas~\cref{lemma:exchangeability12} and \cref{lemma:exchangeability34} show that the vertices $u, v\in V(\HR)$ are $(\GR, \HR)$-exchangeable, Proposition~\cref{prop:embedabilityimpliesexchangeability} applies to show that $u_0$ and $v_0$ are indeed $(\t X, Y)$-exchangeble, which completes the discussion when $\sigma$ satisfies the conditions of one of the cases (C$\rho$).

If $\sigma$ does not satisfy any of these cases, a symmetry argument reduces to the previously discussed case. There are three essential symmetries we use. Firstly, swapping the names of $u_0$ and $v_0$ does not affect the fact that these vertices are $(X, Y)$-exchangeable from $\sigma$. This symmetry has the effect of changing $A_X$ and $A_Y$ into $B_X$ and $B_Y$, and vice versa. Further, noting that $u_0, v_0$ are exchangeable from $\sigma$ if and only if they are exchangeable from $\tr{u_0}{v_0} \circ \sigma$ allows us to change $A_Y$ into $B_Y$ and vice versa without altering $A_X$ and $B_X$. Finally, noting that $u_0, v_0$ are $(X, Y)$-exchangeable from $\sigma$ if and only if $\sigmai(u_0), \sigmai(v_0)$ are $(Y, X)$-exchangeable from $\sigmai$ allows us to swap labels $A_X$ and $B_X$ into $A_Y$ and $B_Y$, and vice versa. We apply these symmetries as follows.

If $\sigma$ majority-maps $N_X(u_0'), N_X(v_0')$ into the same partite set of $Y$ and $\sigmai$ majority-maps $N_Y(u_0), N_Y(v_0)$ into the same partite set of $X$, we first apply the first symmetry, if needed, in order to have both $N_Y(u_0)$ and $N_Y(v_0)$ majority-mapped into $A_X$. Then, we apply the second symmetry, if needed, to have both $A_X$ and $B_X$ majority-mapped into $B_Y$, without affecting the labels of $A_X$ and $B_X$. This reduces $\sigma$ to the case (C1).

If $\sigma$ majority-maps $N_X(u_0'), N_X(v_0')$ to different partite sets of $Y$, and $\sigmai$ majority-maps $N_Y(u_0), N_Y(v_0)$ to different partite sets of $X$, we perform a second symmetry, if needed, in order to ensure that $N_X(u_0')$ is majority-mapped into $B_Y$ and $N_X(v_0')$ is majority-mapped into $A_Y$. Depending on where $N_Y(u_0)$ and $N_Y(v_0)$ are mapped, we either reduce to the case (C3) or case (C4).

Finally, if one pair of sets $N_X(u_0'), N_X(v_0')$ and $N_Y(u_0), N_Y(v_0)$ is majority-mapped into the same partite set while the other pair is majority-mapped into different sets, we apply the third symmetry to ensure that $N_Y(u_0)$ and $N_Y(v_0)$ are mapped to different sets. Then, by applying a combination of the first two symmetries, we can reduce to the case (C2).

The above discussion shows that all cases can be reduced to the four distinguished cases, and hence it suffices to consider only these cases. Since the discussion for these was already presented, the proof is complete.
\end{proof}

\section{Generalizations of Wilson's theorem}\label{sec:wilson}

\noindent
The goal of this section is to characterize multiplicity lists $c\in \Z_{>0}^m$ for which the graphs $\FSm(X, \S_m)$ and $\FSm(\S_n, X)$ are connected, by showing Theorems~\cref{thm:WilsonmultiplicityX} and \cref{thm:WilsonmultiplicityS}. Let us begin by addressing the case in which $X$ is a multiplicity graph.

\begin{proof}[Proof of Theorem~\cref{thm:WilsonmultiplicityX}.]

Suppose first $X$ is Wilsonian. It is simple to check that the lift $X'$ of a Wilsonian graph $X$ with any multiplicity list $c\in \Z_{>0}^m$ is still Wilsonian, and hence Wilson's theorem implies that the graph $\FS(\S_n, X')$ is connected. Using Corollary~\cref{cor:lifting}, this shows that $\FSm(\S_n, X)$ is also connected.
In the case $X$ is biconnected but not Wilsonian and at least one vertex $v\in V(X)$ has multiplicity $c_v\geq 2$, we will again show that the lift of $X$ is Wilsonian. If $v$ is not an isolated vertex, replacing it with a clique of size at least $2$ induces a triangle in the lift graph $X'$. Since $X'$ is biconnected, and neither cycles of length at least $4$ nor the graph $\theta_0$ contain triangles, we conclude that $X'$ must be Wilsonian. As in the previous case, this suffices to show $\FSm(\S_n, X)$ is connected. On the other hand, if all multiplicities $c_v$ are $1$, the graph $\FSm(\S_n, X)$ is isomorphic to the classical version $\FS(\S_n, X)$, which has no multiplicities. Wilson's theorem then shows that $\FS(\S_n, X)$ is not connected, which completes this case.

Finally, we address the case when $X$ is not biconnected. Suppose that all cut vertices of $X$ have multiplicities at least $2$. Again, the goal is to show $X'$ is Wilsonian. Since the lift $X'$ contains a triangle, it cannot be either a cycle of length at least $4$ or $\theta_0$. Similarly, since all cut vertices of $X$ have multiplicities at least $2$, the graph $X'$ has no cut vertices and hence it is biconnected. Therefore, $X'$ is Wilsonian and $\FSm(\S_n, X)$ is biconnected. On the other hand, when $X$ has a cut vertex of multiplicity one, Proposition~\cref{prop:cutvertices} applies because the center of $\S_n$ is a cut vertex. Therefore, we conclude that $\FSm(\S_n, X)$ is disconnected, which completes the proof of Theorem~\cref{thm:WilsonmultiplicityX}.
\end{proof}

\noindent
Now, we focus on the case when $\S_n$ is the multiplicity graph, and prove Theorem~\cref{thm:WilsonmultiplicityS}. The proof will be split into three propositions --- Proposition~\cref{prop:containskbridge} will consider the case when $X$ contains a $k$-bridge, Proposition~\cref{prop:nokbridge} will consider the case when $X$ does not contain a $k$-bridge, and Proposition~\cref{prop:cycle} will consider the case when $X$ is a cycle.

Throughout the whole discussion, we will consider vertices $\sigma\in V(\FSm(X, \S_m))$ as assignments of labels, represented by vertices of $\S_m$, to the vertices of $X$, with the condition that the label $i\in V(\S_m)$ appears $c_i$ times. Moreover, following Wilson's notation from \cite{W}, we denote the center of the star by $\es$ and we denote the multiplicity of $\es$ by $k=c_\es$.

\begin{proposition}\label{prop:containskbridge}
Let $X$ be a simple graph containing a $k$-bridge and let $\S_m$ be the star graph, along with the multiplicity list $c\in \Z_{>0}^m$, in which the center has multiplicity $c_\es=k$. Then, the friends-and-strangers graph $\FSm(X, \S_m)$ is not connected.
\end{proposition}
\begin{proof}
Since $X$ contains a $k$-bridge $a_1, \dots, a_k$, its subgraph $X|_{V(X)-\{a_2, \dots, a_{k-1}\}}$ splits into connected components $A_0$ and $B_0$, where $a_1\in A_0$ and $a_k\in B_0$. We also set $A=A_0-\{a_1\}, B=B_0-\{a_k\}$. As $k$ is fixed throughout the proof, we call this specific $k$-bridge just the bridge. Without loss of generality, we may assume $|A|\geq |B|$.

Since the star $\S_m$ has at least two leaves, there exists a leaf whose multiplicity is at most $|A|$; let us denote this leaf by $l$. Consider the assignment $\sigma:V(X)\to V(\S_m)$, $\sigma\in V(\FSm(X, \S_m))$, which satisfies $\sigma^{-1}(l)\subset A$ and $\sigma^{-1}(\es)=\{a_1, \dots, a_k\}$, and let $\mathcal{C}$ denote the connected component containing $\sigma$ in $\FSm(X, \S_m)$. We will show that any vertex $\tau\in \mathcal{C}$ satisfies $\taui(l)\subset A\cup \{a_1, \dots, a_k\}$. Showing this would immediately imply that $\FSm(X, \S_m)$ is disconnected, since there are assignments $\tau'\in V(\FSm(X, \S_m))$ violating this property.

In fact, we show a stronger statement. For an assignment of labels $\tau:V(X)\to V(\S_m)$, we let $x(\tau)=|\taui(\es)\cap A|$ and we let $1\leq i_1 <i_2 \dots <  i_{r}\leq k$ be the complete set of indices for which $\tau(a_{i_j})\neq \es$. Letting $I(\tau)=\{i_1, \dots, i_{x(\tau)}\}$, we will show that for any $\tau\in C$ we have 
\begin{equation}\label{eqn:invariant}
    \taui(l)\subseteq A\cup \{a_i:i\in I(\tau)\}.
\end{equation}
Intuitively, $x(\tau)$ is the number of blank symbols assigned to the vertices of $A$, and we claim that all vertices labeled $l$ under $\tau$ appear either in $A$ or among the first $x(\tau)$ non-blank labels on the bridge. Note that there are indeed at least $x(\tau)$ non-blank labels on the bridge, i.e., $r\geq x(\tau)$, because there are at most $k-x(\tau)$ blank labels on the bridge.

It suffices to check that, if $\tau_1$ and $\tau_2$ differ by a $(X, \S_m)$-friendly swap, then $\tau_1$ satisfies the property \cref{eqn:invariant} if and only if $\tau_2$ also does. Let $I(\tau_1)=\{i_1, \dots, i_{x(\tau_1)}\}$ and $I(\tau_2)=\{j_1, \dots, j_{x(\tau_2)}\}$. Let $uv\in E(X)$ be the edge on which the $(X, \S_m)$-friendly swap is performed. We consider several cases, based on the position of the edge $uv$. 
\medskip

\noindent
\textbf{Case 1:} Suppose that $uv=a_\nu a_{\nu+1}$ for some $\nu\in \{1, \dots, n-1\}$. Since $u, v\notin A$, we have $x(\tau_1)=x(\tau_2)$. As $\tau_1$ and $\tau_2$ are symmetric and one of the swapped labels is $\es$, we may assume that $\tau_1(a_\nu)=\es=\tau_2(a_{\nu+1})$. Since $\tau_1(a_{\nu+1})\neq \es$, we must have $\nu+1=i_{t}$ for some $t\in \{1, \dots r\}$. Furthermore, since $\tau_2(a_i)\neq \es$ and $\tau_1, \tau_2$ agree on $V(X)-\{a_\nu, a_{\nu+1}\}$, we must also have $\nu=j_t$, and $i_\lambda=j_\lambda$ for $\lambda\neq t$. Since $\tau_1=\tau_2\circ \tr{a_{i_t}}{a_{i_t+1}}=\tau_2\circ (a_{i_t}, a_{j_t})$, it is not hard to see that $\taui_1(l)\subseteq A\cup \{a_{i_1}, \dots, a_{i_{x(\tau_1)}}\}$ is equivalent to $\taui_2(l)\subseteq A\cup \{a_{j_1}, \dots a_{j_{x(\tau_2)}}\}$, which is what we aimed to show.

\noindent
\textbf{Case 2:} Suppose that $u=a_1$, $v\in A$. If we assume, by symmetry, that $\tau_1(a_1)=\es$, we see that $\taui_2(\es)\cap A=(\taui_1(\es)\cap A)\cup \{v\}$, implying $x(\tau_2)=x(\tau_1)+1$. Since $\tau_2(a_1)\neq \es$, we also have $j_1=1$ and $j_{t}=i_{t+1}$ for $t\in \{1, \dots, x(\tau_1)\}$. Since the transposition $\tr{a_1}{v}$ fixes the set $A\cup \{a_1, a_{i_1}, \dots, a_{i_{x(\tau_1)}}\}$ and we have $\taui_1(a_1)\neq l$, we have the equivalence $\taui_1(l)\subseteq A\cup \{a_{i_1}, \dots, a_{i_{x(\tau_1)}}\}\Longleftrightarrow \taui_2(l)\subseteq  A\cup \{a_1, a_{i_1}, \dots, a_{i_{x(\tau_2)}}\}$.

\noindent
\textbf{Case 3:} Suppose that $u, v\in A$ or $u, v\in B$. Then, $x(\tau_1)=x(\tau_2)$ and $\{i_1, \dots, i_{x(\tau_1)}\}=\{j_1, \dots, j_{x(\tau_2)}\}$. Since the transposition $\tr{u}{v}$ fixes the set $A\cup \{a_{i_1}, \dots, a_{i_{x(\tau_1)}}\}$, the equivalence is clear.

\noindent
\textbf{Case 4:} Suppose that $u=a_k, v\in B$. By symmetry, we assume that $\tau_1(a_k)=\es$, and consequently $i_{x(\tau_1)}\leq k-1$, i.e., the last non-blank label is at position at most $k-1$. Since neither of $u, v$ is in $A$, the number of blank labels in $A$ remains constant and so $x(\tau_2)=x(\tau_1)$. Furthermore, as none of the first $x(\tau_1)$ non-blank labels on the bridge is involved in the swap, we have $\{i_1, \dots, i_{x(\tau_1)}\}=\{j_1, \dots, j_{x(\tau_2)}\}$. Therefore, we also see that the last non-blank label in $\tau_2$ is at position at most $j_{x(\tau_2)}\leq k-1$. This means that the transposition $\tr{a_k}{v}$ involves no elements of $A\cup \{a_{i_1}, \dots, a_{i_{x(\tau_1)}}\}$, which concludes this case.

\medskip \noindent
Having completed all the cases, we conclude that $\taui(l)\subseteq A\cup \{a_{i_1}, \dots, a_{i_{x(\tau)}}\}$ for all $\tau\in \mathcal{C}$, which shows that the graph $\FSm(X, \S_m)$ is not connected when $X$ contains a $k$-bridge. 
\end{proof}
\begin{remark}
The proof of Proposition~\cref{prop:containskbridge} is reminiscent of Theorem 6.1 from \cite{DK}, although the technical details of the proof are slightly more general in our proof, due to the multiplicities. Proposition~\cref{prop:containskbridge} can be generalized even further, with almost the same proof, to give the following statement: suppose $X$ is a simple graph containing a $k$-bridge and $Y$ is a multiplicity graph containing vertices $v_1, \dots, v_t$ whose removal disconnects $Y$. If $c_{v_1}+\cdots+c_{v_t}\leq k$, then $\FSm(X, Y)$ is disconnected.
\end{remark}

\noindent
Now, we focus on showing that the graph $\FSm(X, \S_m)$ is connected whenever $X$ has no $k$-bridges, which is shown in Proposition~\cref{prop:nokbridge}. Before presenting the proof of this proposition, we present two simple results which will be used throughout the proof.

\begin{lemma}\label{lemma:Xk+2}
Suppose that $X$ is a connected simple graph on $k+2$ vertices, not isomorphic to a path or a cycle. Then, the graph $\FSm(X, \S_m)$ is connected.
\end{lemma}
\begin{proof}
By Proposition~\cref{prop:exchangeabilityimpliesconnectedness}, it suffices to show that any two adjacent vertices $u, v\in V(X)$ are $(X, \S_m)$-exchangeable from any assignment $\sigma\in V(\FSm(X, \S_m))$. This is clear if any of these vertices is labeled by $\es$, or if $\sigma(u)=\sigma(v)$. Hence, let us assume that $\es\neq \sigma(u)\neq \sigma(v)\neq \es$. Note that all vertices except $u$ and $v$ are labeled by $\es$ under $\sigma$. Since $X$ is not a path or a cycle, it contains a vertex of degree at least $3$ --- let $x_1$ be this vertex. By symmetry of $u$ and $v$, we may assume that there is a path from $v$ to $x_1$ with vertices $v, u, x_t, \dots, x_2, x_1$. Suppose also that $x_1$ has neighbors $y$ and $z$, which are different from $x_2$. Then, the sequence of swaps $ux_t, \dots, x_2x_1,$ $x_1y, vu, ux_t, \dots, x_2x_1, x_1z,$ $x_1y, x_1x_2, \dots, x_tu,$ $uv,$ $zx_1, x_1x_2, \dots, x_tu$ exchanges $u$ and $v$.
\end{proof}

\begin{proposition}\label{prop:nokbridge}
Suppose $X$ is a simple graph that does not contain a $k$-bridge and is not isomorphic to a cycle. Then, the graph $\FSm(X, \S_m)$ is connected. 
\end{proposition}
\begin{proof} 
Our main tool will be the notion of exchangeability, which was defined in Section \cref{sec:prelim}. Namely, Proposition~\cref{prop:exchangeabilityimpliesconnectedness} reduces showing that $\FSm(X, \S_m)$ is connected to showing that any two adjacent vertices $u, v\in V(X)$ are $(X, \S_m)$-exchangeable from any starting arrangement $\sigma\in V(\FSm(X, \S_m))$, which is done by explicitly constructing a sequence of $(X, \S_m)$-friendly swaps which transforms $\sigma$ into $\sigma\circ \tr{u}{v}$. If either $\sigma(u)=\es$ or $\sigma(v)=\es$, it is easy to see that the swap along the edge $uv$ itself is a $(X, \S_m)$-friendly swap, thus accomplishing the goal. Therefore, may we assume that $\sigma(u), \sigma(v)\neq \es$ in the rest of the proof. 

Following the notation from \cite{W}, for a path $p=a_1a_2\cdots a_l$ and an assignment $\tau$ with $\taui(a_1)=\es$, we denote the sequence of swaps along the edges $a_1a_2, a_2a_3,\dots, a_{l-1}a_l$ by $\sigma_p$. Note that $\tau$ and $\tau\circ \sigma_p$ are in the same connected component of $\FS(X, \S_m)$.

The first step of the proof will be \textit{move} every blank symbol from its original position to obtain an arrangement in which the set of vertices labeled by a blank symbol, together with $u$ and $v$, forms a connected subgraph of $X$. More formally, we will show there exists a sequence of $(X, \S_m)$-friendly swaps, not including vertices $u$ and $v$, which transforms $\sigma$ into an assignment $\sigma'$ for which $X|_{\sigma'^{-1}(\es)\cup \{u, v\}}$ is connected. After ensuring this subgraph is not a path or a cycle, we will be in the position to apply Lemma~\cref{lemma:Xk+2}, which will show that $u$ and $v$ are exchangeable.

We will achieve our first goal inductively. Suppose that after several swaps, the vertices $u$, $v$, and $i$ other vertices with blank labels form a connected subgraph $C\subset X$. If $i<k$, pick a vertex $w\in V(X)$ labeled by $\es$ and let $p$ be the shortest path from $w$ to a vertex in the neighborhood of $C$. Applying the sequence of swaps $\sigma_p$ transforms the existing arrangement into an arrangement in which the vertices $u$, $v$, and $i+1$ other vertices with blank labels form a connected subgraph of $X$. If we apply this procedure as long as $i<k$, we arrive at the arrangement $\sigma'$ with the property that $X|_{\sigma'^{-1}(\es)\cup \{u, v\}}$ is connected.

Hence, $X|_{\sigma'^{-1}(\es)\cup \{u, v\}}$ is a connected subgraph of $X$ on $k+2$ vertices, and we denote it by $H$. Since vertices $u, v$ were not included in any of the swaps transforming $\sigma$ into $\sigma_k$, by Proposition~\cref{prop:exchangeabilityfromdifferentarrangements}, it suffices to show that the pair $u, v$ is $(X, \S_m)$-exchangeable from $\sigma'$.

In the case $H$ is not isomorphic to a path, Lemma~\cref{lemma:Xk+2} applies to subgraphs of $X$ and $\S_m$ spanned by the vertices $\{u, v\}\cup \sigma'^{-1}(\es)$ and $\{\sigma(u), \sigma(v), \es\}$, and it show that the vertices $u$ and $v$ are exchangeable with respect to these subgraphs. However, this also means $u, v$ are $(X, \S_m)$-exchangeable from $\sigma'$. Since transforming $\sigma$ into $\sigma'$ did not involve swaps including $u$ or $v$, Proposition~\cref{prop:exchangeabilityfromdifferentarrangements} shows $u, v$ are $(X, \S_m)$-exchangeable from $\sigma$.

In the case $H$ is isomorphic to a path or a cycle, we denote its vertices by $a_1, \dots, a_{k+2}$, having $u=a_i$ and $v=a_{i+1}$. We now consider two cases, based on whether any vertex of $H$ has degree $3$ in $X$.

\medskip
\noindent
\textbf{Case 1:} Suppose that a vertex of $a_j\in H$ has degree at least $3$ in $X$. Then, we either have $a_j\not \in \{u, v\}$ or $a_j\in \{u, v\}$. In the case $a_j\not\in \{u, v\}$, we may assume that $j<i$, by reversing the indices of $a_1, \dots, a_k$. Then, the sequence of swaps $ya_j, a_ja_{j-1}, \dots, a_2a_1$ transforms $\sigma_k$ into $\sigma_k'$, in which $X|_{{\sigma_k'}^{-1}(\es)\cup \{u\}}$ is not isomorphic to a path, thus reducing this case to Lemma~\cref{lemma:Xk+2}.

Suppose now that $a_j\in \{u, v\}$. By potentially relabeling $u$ and $v$ or reversing indices of $a_1, \dots, a_k$, we may assume that $a_j=a_i=u$ and $a_{i+1}=v$. Since $i=j\leq k-2$, we have $i+2\leq k$ and we may consider the following sequence of swaps
\begin{equation}\label{eqn:sequence1}
    a_{i+2}v, vu, a_iy, a_{i-1}u, a_{i-1}a_{i-2}, \dots, a_2a_1.
\end{equation}
This sequence of swaps transforms $\sigma_k$ into $\sigma_{k}'$ in which $X|_{{\sigma_k'}^{-1}(\es)\cup \{u\}}$ is not isomorphic to a path, thus reducing this case to Lemma~\cref{lemma:Xk+2}. After exchanging $u, v$ using Lemma~\cref{lemma:Xk+2}, it suffices to reverse the sequence \cref{eqn:sequence1} to show $u, v$ are $(X, \S_m)$-exchangeable from $\sigma'$.

\medskip
\noindent
\textbf{Case 2:} Suppose that $H$ does not contain a vertex of degree $3$ in $X$. Since $X$ is connected, $H$ is not a cycle, and because $X$ does not contain a $k$-bridge, the graph $X-{\{a_3, \dots, a_k\}}$ is connected. Hence, $H$ must be an interval on the cycle $C$, whose vertices we denote by $a_1, \dots, a_{k+2}, a_{k+3}, \dots, a_l$. Since $X$ is not a cycle, there exists a vertex $a_t\in C$ of degree at least $3$ in $X$. If we pick any vertex $a_j$ of $H$ labeled by $\es$, performing swaps along edges $a_ja_{j+1}, \dots, a_{l-1}a_l, a_la_1, \dots, a_{j-2}a_{j-1}$ has the effect of shifting all labels of $C$ one position backwards. By performing this sequence of swaps a sufficient number of times, we may assume that the labels $\sigma(u)$ and $\sigma(v)$ are moved to vertices $a_{t}$ and $a_{t+1}$. Now, a vertex labeled by $\sigma(u)$ has degree at least $3$, and hence Case $1$ implies that $a_t$ and $a_{t+1}$ are $(X, \S_m)$-exchangeable from the resulting arrangement. Reversing the counterclockwise shifts previously performed shows that $\sigma'\circ \tr{u}{v}$ can be obtained from $\sigma'$ through a sequence of $(X, 
\S_m)$-friendly swaps, which completes the proof.
\end{proof}

\noindent
Finally, to complete the discussion of $\FSm(X, \S_m)$, we address the case when $X$ is a cycle.

\begin{proposition}\label{prop:cycle}
Let $\S_m$ be a multiplicity graph with the multiplicity list $c\in \Z_{>0}^m$ with $c_{\es}=k$ and total multiplicity $n$. Then, $\FSm(\C_n, \S_m)$ is connected if and only if $m=3$ and one of the leaves of $\S_m$ has multiplicity $1$.
\end{proposition}
\begin{proof}
The key observation in this proof is that the cyclic order of non-blank labels remains unchanged after a $(\C_n, \S_m)$-friendly swap. Hence, if $\FS(\C_n, \S_m)$ is connected, there must be only one possible cyclic order of non-blank labels. Suppose that the multiplicities of the leaves of $\S_m$ are $c_1, \dots, c_{m-1}$. The number of cyclic orders of elements $1, \dots, m-1$ where element $i$ appears with multiplicity $c_i$ is $\frac{(c_1+\dots +c_{m-1}-1)!}{c_1!\cdots c_m!}$, which equals $1$ if and only if $m=3$ and either $c_1=1$ or $c_2=1$. On the other hand, if the cyclic order of non-blank labels matches in two different assignments, it is not hard to see that one can be transformed into the other using a sequence of $(\C_n, \S_m)$-friendly swaps. 
\end{proof}

\noindent
The combination of Propositions~\cref{prop:containskbridge}, \cref{prop:nokbridge} and \cref{prop:cycle} immediately yields the proof of Theorem~\cref{thm:WilsonmultiplicityS}, completing this section.

\section{Paths and cycles}\label{sec:pathscycles}

\subsection{Background and terminology.} 

In this section, we investigate the structure of the friends-and-strangers graphs arising when one of the graphs is a path or a cycle, while the other graph is a multiplicity graph. This extends the results obtained by Defant and Kravitz in \cite{DK}, who characterized the connected components of $\FS(\P_n, Y)$, $\FS(\C_n, Y)$ for a simple graph $Y$ in terms of the acyclic orientations of its complement $\o{Y}$. Before presenting the generalizations, we introduce relevant terminology and make a brief overview of the results obtained by Defant and Kravitz.

For a simple graph $G$, an \dfn{orientation} of the graph corresponds to choosing a direction for each of its edges, and an orientation is \dfn{acyclic} if the arising directed graph has no cycles. We denote the set of acyclic orientations of $G$ by $\Acyc(G)$. Further, if we think of $\P_n$ and $\C_n$ as graphs on the vertex set $[n]$, we see that every bijection $\sigma:[n]\to V(G)$ induces an acyclic orientation $\alpha_G(\sigma)$ in which an edge $uv\in E(G)$ is directed from $u\to v$ if and only if $\sigmai(u)<\sigmai(v)$. Finally, for a given acyclic orientation $\alpha\in \Acyc(G)$, we define the set of its \dfn{linear extensions} $\L(\alpha)$ to be the set of all bijections $\sigma:[n]\to V(G)$ which induce the orientation $\alpha$ on $G$, i.e., $\L(\alpha)=\{\sigma:[n]\to V(G)|\alpha_G(\sigma)=\alpha\}$. Even more generally, for a set of acyclic orientations $A\subseteq \Acyc(G)$, we may define the set of linear extensions of $A$ to be $\L(A)=\bigcup_{\alpha\in A}\L(\alpha)$. In \cite{DK}, Defant and Kravitz showed that arrangements $\sigma, \tau:V(\P_n)\to V(Y)$ are in the same connected component of $\FS(\P_n, Y)$ if and only if they induce the same acyclic orientation of the complement of $Y$, i.e., $\alpha_{\o{Y}}(\sigma)= \alpha_{\o{Y}}(\tau)$. 

More work is needed to describe the connected components of $\C_n$, since we need to introduce the notions of toric and double flip equivalence. For a given acyclic $\alpha\in \Acyc(G)$, a vertex $u\in V(G)$ is a \dfn{source} of $\alpha$ if all edges $uv\in E(G)$ are oriented $u\to v$, and a vertex $u\in V(G)$ is a \dfn{sink} of $\alpha$ if all edges $uv\in E(G)$ are oriented $v\to u$. Given an orientation $\alpha\in \Acyc(G)$ and a vertex $u$ which is either a source or a sink of $\alpha$, reversing all edges incident to $u$ gives a new acyclic orientation of $G$. We refer to this operation as a \dfn{flip}, and which is called \dfn{positive} if $u$ is a source of $\alpha$, and \dfn{negative} otherwise. This allows to define an equivalence relation called \dfn{toric equivalence} by declaring that $\alpha, \alpha'$ are equivalent if one can be obtained from the other by a sequence of flips. If $\alpha, 
\alpha'$ are torically equivalent, we write $\alpha\sim\alpha'$. Also, we write $\Acyc(G)/\sim$ for the set of toric equivalence classes of acyclic orientations, and we call elements of $\Acyc(G)/\sim$ \dfn{toric acyclic orientations}. 

If $u, v$ are non-adjacent vertices of $G$ such that $u$ is a source and $v$ is a sink of $\alpha$, respectively, performing two simultaneous flips on $\alpha$, at $u$ and $v$, produces a new acyclic orientation. We call this operation a \dfn{double flip}, and, as before, we define the equivalence relation $\approx$, called \dfn{double flip equivalence}, by declaring that two acyclic orientations are double flip equivalent if one can be obtained from the other by a sequence of double flips. It is worth noting that double flip equivalence implies toric equivalence of two orientations, whereas the reverse is not always true. We denote the set of double flip equivalence classes by $\Acyc(G)/\approx$ and we denote the equivalence class of an orientation $\alpha\in \Acyc(G)$ by $[\alpha]_\approx$.

In this language, one can describe the connected components of $\FS(\C_n, Y)$ in several ways. Theorem 4.1 of \cite{DK} states that the arrangements $\sigma, \tau:V(\C_n)\to V(Y)$ are in the same connected component of $\FS(\C_n, Y)$ if and only if their induced acyclic orientations of $\o{Y}$ are double flip equivalent, i.e., $\alpha_{\o{Y}}(\sigma)\approx\alpha_{\o{Y}}(\tau)$. Investigating the relation between double flip equivalence and toric equivalence further, Defant and Kravitz showed that every toric equivalence class $\alphasim$ can be viewed as a union of double flip equivalence classes forming an orbit under an action of a certain abelian group. This allows them to characterize all graphs $Y$ for which $\FS(\C_n, Y)$ is connected, and the answer turns out to be the set of all graphs $Y$ whose complement is a forest of trees of coprime sizes. 

In the process of showing these results, Defant and Kravitz also showed the following useful fact, which we will also make use of. Namely, Corollary 4.2 of \cite{DK} shows that if $\alpha\approx \beta$ and $\vec{\alpha}, \vec{\beta}$ arise from $\alpha, \beta$, respectively, by flipping one source into a sink, then $\vec{\alpha}\approx \vec{\beta}$. The following statement is a direct corollary of their result.

\begin{corollary}\label{cor:doubleflips}
Let $\alpha, \beta\in \Acyc(G)$ be acyclic orientations, and assume that $\beta$ is obtained from $\alpha$ through a sequence of $a_{+}$ positive flips and $a_{-}$ negative flips. Then, any sequence of $a_{+}-a_{-}$ positive flips transforms $\alpha$ into $\beta'$, which is double flip equivalent to $\beta$.
\end{corollary}

In the rest of this section, we present a generalization of results from \cite{DK} to the multiplicity setting. First, we discuss the graphs $\FSm(\P_n, X)$, and show that their connected components correspond to equivalence classes of $\Acyc(\o{X'})$ under a certain equivalence relation. Then, we characterize the connected components of $\FSm(\C_n, X)$ and give a full characterization of graphs $X$ for which $\FSm(\C_n, X)$ is connected.

\subsection{Connectivity of $\FSm(\P_n, X)$}
We begin by describing the connected components of $\FSm(\P_n, X)$ for a multiplicity graph $X$, which is analogous to Theorem 3.1 of \cite{DK}. The main technical tool we use is the notion of permutation equivalence and the bulk of work lies in investigating how this equivalence relation interacts with toric and double flip equivalence on the lift of $X$.

More precisely, the notion of permutation equivalence $\equiv$, introduced in Section~\cref{sec:prelim}, can be extended to acyclic orientations in the following way. If $X'$ is the lift of $X$, having the clique decomposition $V(X')=\bigsqcup_{v\in V(X)} S_v$, we can define the action of $\Sym_X$ on $\Acyc(X')$ by setting $\rho(\alpha_{\o{X'}}(\sigma))=\alpha_{\o{X'}}(\rho\circ \sigma)$ for every $\sigma:[n]\to V(X')$. Recall that $\Sym_X=\prod_{v\in V(X)} \Sym_{S_v}$ denotes the set of permutation which fix the cliques $S_v$ of $X'$.

We say that acyclic orientations $\alpha, \beta\in \Acyc(\o{X'})$ are \dfn{permutation equivalent}, written as $\alpha\equiv\beta$, if they are in the same orbit of the action of $\Sym_X$ on $\Acyc(X')$. In combinatorial terms, we have $\alpha\equiv\beta$ precisely when there is a permutation $\rho\in \Sym_X$ which transforms $\alpha$ into $\beta$, in the sense that $u\to_{\alpha} v$ is equivalent to $\rho(u)\to_{\beta} \rho(v)$.  A simple consequence of this definition is that permutation equivalent arrangements $\sigma, \tau:V(\C_n)\to V(X')$ induce permutation equivalent orientations $\alpha_{\o{X'}}(\sigma)\equiv \alpha_{\o{X'}}(\tau)$.

\begin{theorem}\label{thm:multiplicitypaths}
Let $X$ be a multiplicity graph of total multiplicity $n$ and let $X'$ be its lift. The arrangements $[\sigma]_\equiv, [\tau]_\equiv\in V(\FSm(\P_n, X))$ are in the same connected component if and only if $\sigma, \tau$ induce permutation equivalent orientations on $\o{X'}$, i.e. $\alpha_{\o{X'}}(\sigma)\equiv\alpha_{\o{X'}}(\tau)$. 
\end{theorem}
\begin{proof}

In this proof, we use Corollary~\cref{cor:quotientcomponents} to reduce Theorem~\cref{thm:multiplicitypaths} to the Theorem 3.1 of \cite{DK}. 

From Corollary~\cref{cor:quotientcomponents}, we know that $[\sigma]_\equiv, [\tau]_\equiv$ are in the same connected component of $\FSm(\P_n, X)$ if and only if $\sigma, \tau$ are in the same connected component of $G$, where $G$ is the graph $\FS(\P_n, X')$ with the added set of edges between pairs of permutation equivalent arrangements. 

Suppose now that $\sigma, \tau$ are indeed in the same connected component of $G$. To show $\alpha_{\o{X'}}(\sigma)\equiv\alpha_{\o{X'}}(\tau)$, it suffices to check this in case $\sigma$ and $\tau$ are adjacent. In case $\sigma, \tau$ differ by a $(\P_n, X')$-friendly swap, it follows from Theorem 3.1 of \cite{DK} that $\alpha_{\o{X'}}(\sigma)=\alpha_{\o{X'}}(\tau)$. In case $\sigma$ and $\tau$ are permutation equivalent, it follows from our previous observation that so their induced orientations must be permutation equivalent as well. 

To show the other direction, assume $\alpha_{\o{X'}}(\tau)\equiv\alpha_{\o{X'}}(\sigma)$. Then there exists a permutation $\rho\in \Sym_X$ such that $\alpha_{\o{X'}}(\tau)=\alpha_{\o{X'}}({\rho\circ \sigma})$. Theorem 3.1 of \cite{DK} ensures that $\tau$ and $\rho\circ \sigma$ are connected in $\FS(\P_n, X')$, while $\sigma\equiv \rho\circ\sigma$, implying $\{\sigma, \rho\circ \sigma\}\in E_0$. Combining these, we see $\sigma$ is connected to $\tau$ in $G$, completing the proof.
\end{proof}


\subsection{Connectivity of $\FSm(\C_n, X)$.} We now focus on describing the connected components of $\FSm(\C_n, X)$. The proofs in this section add two new ingredients to the work of Defant and Kravitz. The first new idea is to combine pairs of equivalence relations  $\sim, \equiv$ and $\approx, \equiv$ into their finest common coarsening, while the second one is to define the period of an acyclic orientation $\alpha\in \Acyc(\o{X'})$.

More precisely, the first step is to define an equivalence relation $\simeq$ on $\Acyc(\o{X'})$ in which $\alpha\simeq\beta$ if and only if there exists a sequence of orientations $\alpha=\alpha_0, \alpha_1, \dots, \alpha_n=\beta$ in which $\alpha_i$ and $\alpha_{i+1}$ are either permutation equivalent or torically equivalent. Note that permuting moves and flips commute, in the sense that flipping a vertex $u$ in $\alpha$ and obtaining an orientation $\beta\in \Acyc(\o{X'})$ and then taking $\beta$ to $\rho(\beta)$ corresponds to flipping a vertex $\rho(u)$ in $\rho(\alpha)$, which also results in the orientation $\rho(\beta)$. Hence, another way to define $\simeq$ is to declare $\alpha\simeq \beta$ whenever there exists an orientation $\gamma\in \Acyc(\o{X'})$ satisfying $\alpha\equiv\gamma\sim \beta$.

Following the same strategy, one can also define an equivalence relation $\approxeq$ as the finest common coarsening of $\approx$ and $\equiv$. Since $\simeq$ and $\approxeq$ are both coarser than $\equiv$, they extend to equivalence relations on $\Acyc(\o{X'})/\!\equiv$. Using these notions, we can generalize Theorem 4.1 of \cite{DK} in the following way.

\begin{theorem}\label{thm:multiplicitycyclesdoubleflip}
Let $X$ be a multiplicity graph and let $X'$ be its lift.  The arrangements $[\sigma]_\equiv, [\tau]_\equiv\in V(\FSm(\C_n, X))$ are in the same connected component if and only if $\alpha_{\o{X'}}(\sigma)\approxeq\alpha_{\o{X'}}(\tau)$.
\end{theorem}
\begin{proof}
The structure of the proof closely follows the proof of the Theorem~\cref{thm:multiplicitypaths}. Recall that Corollary~\cref{cor:quotientcomponents} guarantees that the connected components of $\FS(X, \C_n)$ correspond to projections under $\sigma\mapsto [\sigma]_\equiv$ of connected components in $G=\big(V(\FS(X', \C_n)), E(\FS(X', \C_n))\cup E_0\big)$, where $E_0=\{\sigma\tau:\sigma, \tau\in V(\FS(X', \C_n)), \sigma\equiv\tau\}$. Hence, to show Theorem~\cref{thm:multiplicitycyclesdoubleflip}, it suffices to check that $\sigma, \tau\in V(G)$ are connected in $G$ if and only if their induced orientations $\alpha_{\o{X'}}(\sigma)$ and $\alpha_{\o{X'}}(\tau)$ satisfy $\alpha_{\o{X'}}(\sigma)\approxeq\alpha_{\o{X'}}(\tau)$.

To show the forward direction, assume that $\sigma$ and $\tau$ are in the same connected component. Furthermore, if we induct on the distance between $\sigma$ and $\tau$, without loss of generality we may assume $\sigma \tau \in E(G)=E(\FS(X', \C_n))\cup E_0$. If $\sigma\tau\in E_0$, we have $\sigma\equiv\tau$, and hence $\alpha_{\o{X'}}(\sigma)\equiv\alpha_{\o{X'}}(\tau)$. If $\sigma\tau\in E(\FS(X', Y'))$, we have $\alpha_{\o{X'}}(\sigma)\approx\alpha_{\o{X'}}(\tau)$, from Theorem 4.1 of \cite{DK}. In either case, we clearly have $\alpha_{\o{X'}}(\sigma) \approxeq \alpha_{\o{X'}}(\tau)$.

On the other hand, if $\alpha_{\o{X'}}(\sigma) \approxeq \alpha_{\o{X'}}(\tau)$, then there exists a permutation $\rho\in \Sym_X$ and a sequence of double flip moves transforming $\alpha_{\o{X'}}(\sigma)$ into $\alpha_{\o{X'}}(\tau)$. As double flips correspond to edges of $\FS(X', \C_n)$, and permuting with $\rho$ corresponds to an edge in $E_0$, we conclude that $\sigma$ and $\tau$ are connected in $G$, completing the proof.
\end{proof}

\noindent
One of the key insights Defant and Kravitz used to understand the structure of $\FS(\C_n, Y)$ was the automorphism $\varphi_n:[n]\to [n]$ of the graph $\C_n$ given by $\varphi_n(k)=k+1$ for $k<n$ and $\varphi_n(n)=1$. Intuitively, one can think of $\varphi_n$ as a cyclically shifting $\C_n$ one place forward. Since $\varphi_n$ is an automorphism of the cycle $\C_n$, it also induces an automorphism $\varphi_n^*$ of the friends-and-strangers graph $\FS(\C_n, Y)$ (Proposition 2.3 of \cite{DK}).  It is not hard to see that these ideas extend to the multiplicity setting and that $\varphi_n^*$ can be viewed both as an automorphism of $\FS(\C_n, X')$ and of $\FSm(\C_n, X)$, which will be used extensively henceforth. For simplicity of notation, we will suppress the index $n$ when it is clear what the underlying cycle is. Using the automorphism $\varphi^*$, we define the second main notion of this section, the period of an orientation.

\begin{definition}
Let $\sigma:[n]\to V(X')$ be a arrangement corresponding to a vertex of $\FS(\C_n, X')$. The \textit{period} of $\sigma$ is the smallest positive integer $\pi_\sigma$ for which $\sigma\circ \varphi^{\pi_\sigma}\equiv\sigma$. The \textit{period} of an acyclic orientation $\alpha\in \Acyc(\o{X'})$ is equal to the smallest period $\pi_\sigma$ of any linear extension $\sigma\in \L(\alpha)$.
\end{definition}

\begin{example}
Consider a multiplicity graph $X$ with $V(X)=\{u_1, u_2, u_3\}$ and $E(X)=\{u_1u_2, u_2u_3\}$. Moreover, assume that $u_1$ and $u_2$ have capacity $2$ and $u_3$ has capacity $4$. Then, the arrangement $\sigma:[8]\to V(X)$ sending $1, 5\mapsto u_1, 3, 7\mapsto u_2$ and $2, 4, 6, 8\mapsto u_3$ has period $\pi_\sigma=4$. On the other hand, the arrangement $\tau:[8]\to V(X)$ given by $2, 5\mapsto u_1, 3, 7\mapsto u_2, 1, 4, 6, 8\mapsto u_3$ has period $8$, even though it induces the same orientation $\alpha$ on $\o{X'}$ as $\sigma$. Moreover, the period of the induced orientation $\alpha$ is $4$ in this case. 
\end{example}

\begin{proposition}\label{prop:structureperiodic}
Let $\sigma:[n]\to V(X')$  be an arrangement of period $\pi_\sigma$, and let $V(X')=\bigsqcup_{v\in V(X)} S_v$ be the clique decomposition of $V(X')$. Consider the partition of $[n]$ into congruence classes modulo $\pi_\sigma$, given by $[n]=\bigsqcup_{j\in [\pi_\sigma]}\{k\in [n]:k\equiv_{\pi_\sigma} j\}=\bigsqcup_{j\in [\pi_\sigma]} P_j$. For every $j\in [\pi_\sigma]$, $\sigma(P_j)$ is fully contained in one of the cliques and each clique $S_v$ contains the images of exactly $c_v n/\pi_{\sigma}$ congruence classes $P_j$.
\end{proposition}
\begin{proof}
We begin by showing that, if $\sigma(j)\in S_v$, then $\sigma(j+\pi_\sigma)\in S_v$ as well. From the definition of $\pi_\sigma$, we have $\sigma\circ \varphi^{\pi_\sigma}=\rho\circ \sigma$, for some $\rho\in \Sym_X$. Substituting $j$ in this equality gives $\sigma(j+\pi_\sigma)=\rho(\sigma(j))$. However, as $\rho$ only permutes the elements within the cliques, and $\sigma(j)\in S_v$, we infer $\sigma(j+\pi_\sigma)\in S_v$. 
Repeated application of the above argument now gives that $j\in S_v$ implies $\sigma(P_j)\subseteq S_v$, which completes the proof of the first statement. The second statement follows by comparing cardinalities of sets $\sigma(P_j)$ and $S_v$.
\end{proof}

\noindent
The following proposition shows that the period of torically or permutation equivalent orientations is the same. Hence, we may talk about the period of a equivalence class $[\alpha]_\simeq\in \Acyc(\o{X'})/\!\simeq$.

\begin{proposition}
Let $X$ be a connected graph and let $\alpha, \beta\in \Acyc(\o{X'})$ be acyclic orientation satisfying $\alpha\simeq \beta$. Then $\pi_{\alpha}=\pi_\beta$.
\end{proposition}
\begin{proof}
It suffices to check that $\pi_\alpha=\pi_\beta$ whenever $\alpha\equiv\beta$ and $\alpha\sim \beta$. Let us begin by addressing the case when $\alpha\equiv\beta$. We will show that $\pi_{\alpha}\geq \pi_{\beta}$. By symmetry of $\alpha$ and $\beta$, this will also imply $\pi_\beta\geq \pi_\alpha$, which suffices to conclude $\pi_\alpha=\pi_\beta$.

As $\alpha\equiv\beta$, we have a permutation $\rho\in \Sym_X$ transforming $\alpha$ into $\beta$. Let $\sigma$ be a linear extension of $\alpha$ with period $\pi_\alpha$. Then, $\rho\circ\sigma$ is a linear extension of $\beta$, with period at most $\pi_\alpha$, because $(\rho\circ \sigma)\circ \varphi^{\pi_\alpha}=\rho\circ (\sigma\circ \varphi^{\pi_\alpha})=\rho\circ \sigma$. Hence, $\pi_\beta\leq \pi_\alpha$, which completes the discussion in the case $\alpha\equiv\beta$.

If $\alpha\sim\beta$, the proof is more involved and needs the following auxiliary lemma.

\begin{lemma}\label{lemma:periodicsources}
Let $\alpha\in \Acyc(\o{X'})$, and let $u$ be a source of $\alpha$. Then, there exists a linear extension $\sigma\in \L(\alpha)$ with $\pi_\sigma=\pi_\alpha$ and $\sigma(1)=u$. Similarly, if $v$ is a sink of $\alpha$, there exists a linear extension $\sigma\in \L(\alpha)$ with $\pi_\sigma=\pi_\alpha$ and $\sigma(n)=v$.
\end{lemma}
\begin{proof}

As the two statements given in the lemma are completely analogous, we discuss only the case when $u$ is a source. Let $\tau:[n]\to V(X')$ be a linear extension of $\alpha$ with $\pi_\tau=\pi_\alpha$, which does not necessarily satisfy $\tau(1)=u$.

Consider the set of vertices of $X'$ that are labeled by a number not larger than $\pi_\alpha$, which is given by $S=\{\tau(i):i\leq \pi_\alpha\}$. If we recall the decomposition $[n]=\bigsqcup_{j\in [\pi_\sigma]}P_j$ from Proposition~\cref{prop:structureperiodic}, we set $S$ to be the set of images of minimal elements of progressions $P_j$. 

We note that $u\in S$. To see why, let $S_v$ be the clique of $X'$ which contains $u$. As the vertex $u$ is a source, it is the minimum of $\sigma^{-1}$ on the clique $S_v$. By Proposition~\cref{prop:structureperiodic}, we know that the minimum of $\sigma^{-1}$ on each clique is at most $\pi_\sigma$, as the preimage of every clique contains at least one congruence class modulo $\pi_\sigma$, which implies $u\in S$.

The restriction of $\alpha$ onto $S$ still induces an acyclic orientation, of which $\tau|_{[\pi_\alpha]}$ is still a linear extension. As $u$ is a source in $\alpha|_S$, we infer there exists a linear extension $\sigma:[\pi_\alpha]\to S$ of $\alpha|_S$ which has $\sigma(1)=u$. 

Then, we extend $\sigma$ to the $\t{\sigma}:[n]\to V(X')$ by setting $\t{\sigma}(k\pi_\alpha+j)=\tau\big(k\pi_\alpha+\tau^{-1}\circ\sigma(j)\big)$, where $j$ ranges through $[\pi_\alpha]$ and $k\in \{0, \dots, n/\pi_\alpha-1\}$. It is not hard to check that $\t{\sigma}$ indeed agrees with $\sigma$ on $[\pi_\alpha]$, as for $k=0$ we have $\t{\sigma}(j)=\tau\circ\tau^{-1}\circ\sigma(j)=\sigma(j)$. Hence, we have $\t{\sigma}(1)=\sigma(1)=u$. It remains to check that $\t\sigma$ is a proper linear extension of $\alpha$ and that $\t{\sigma}$ has period $\pi_\alpha$.

Let us first check that $\t\sigma\in \L(\alpha)$. Suppose $u_1u_2$ is an edge of $\o{X'}$, directed from $u_1$ to $u_2$ under $\alpha$, and let $u_1=\tau(k_1\pi_\alpha+l_1), u_2=\tau(k_2\pi_\alpha+l_2)$. As $\tau\in \L(\alpha)$, we have
\begin{equation}\label{eqn:eqn1}
    k_1\pi_\alpha+l_1< k_2\pi+l_2.
\end{equation} The goal is now to show $\t\sigma^{-1}(u_1)<\t\sigma^{-1}(u_2)$. From the definition of $\t\sigma$, we know that $u_1=\t\sigma(k_1\pi_\alpha+j_1), u_2=\t\sigma(k_2\pi_\alpha+j_2)$, where $j_1=\sigma^{-1}\circ \tau(l_1), j_2=\sigma^{-1}\circ \tau(l_2)$. To show that $\t\sigma\in \L(\alpha)$, we need to verify that 
\begin{equation}\label{eqn:eqn2}
    k_1\pi_\alpha+j_1< k_2\pi_\alpha+j_2. 
\end{equation}
From \cref{eqn:eqn1}, we see that $k_1\leq k_2$. If $k_1<k_2$, the inequality \cref{eqn:eqn2} follows immediately, as $j_1, j_2\in \{1, \dots, \pi_\alpha\}$. On the other hand, if $k_1=k_2$, we have $l_1\leq l_2$. Hence, there is an edge in $\alpha|_S$ from $\tau(l_1)$ to $\tau(l_2)$. As $\sigma\in \L(\alpha|_S)$, we must also have $\sigma^{-1}(\tau(l_1))<\sigma^{-1}(\tau(l_2))$, implying $j_1<j_2$. Thus, $\t\sigma$ is a linear extension of $\alpha$.

Now, we show that the period of $\t\sigma$ is $\pi_\alpha$. We have \begin{align*}
\t\sigma \circ \varphi^{\pi_\alpha}(k\pi_\alpha+j)=&\t\sigma((k+1)\pi_\alpha+j\mod n)\\
=&\tau((k+1)\pi_\alpha+\tau^{-1}\circ \sigma(j)\mod n)\\
=&\tau\circ\varphi^{\pi_\alpha} (k\pi_\alpha+\tau^{-1}\circ \sigma(j)).
\end{align*}
As the period of $\tau$ is $\pi_\alpha$, we infer that $\tau\circ\varphi^{\pi_\alpha} (k\pi_\alpha+\tau^{-1}\circ \sigma(j))$ is in the same clique as $\tau(k\pi_\alpha+\tau^{-1}\circ \sigma(j))=\t\sigma(k\pi_\alpha+j)$, meaning that $\t\sigma \circ \varphi^{\pi_\alpha}(k\pi_\alpha+j)$ and $\t\sigma(k\pi_\alpha+j)$ are in the same clique. This shows that the period of $\t\sigma$ is at most $\pi_\alpha$, but since $\t\sigma$ is a linear extension of $\alpha$, the period of $\t\sigma$ is exactly $\pi_\alpha$.
\end{proof}

\noindent
Now, we are ready to show that the period of an orientation is preserved under toric equivalence. If $\alpha\sim \beta$, the orientation $\beta$ can be obtained from $\alpha$ by applying a sequence of flip moves. By induction, it suffices to address the case when $\alpha$ and $\beta$ differ by a single flip on the vertex $u$, where we may assume that $u$ is a sink in $\alpha$ and a source in $\beta$. Using Lemma~\cref{lemma:periodicsources}, we find a linear extension $\sigma\in \L(\alpha)$ for which $\pi_\sigma=\pi_\alpha$ and $\sigma(n)=u$. As $\alpha$ and $\beta$ differ by a flip of $u$, we see that $\sigma\circ \varphi$ is a linear extension of $\beta$.

Note that $\sigma\circ \varphi$ has period as most $\pi_\alpha$, because $\sigma\circ\varphi\circ\varphi^{\pi_\alpha}=\sigma\circ\varphi^{\pi_\alpha}\circ\varphi\equiv\sigma\circ\varphi$. Hence, $\pi_\beta\leq\pi_\alpha$, which suffices to complete the proof by applying symmetry.
\end{proof}

\noindent
In Theorem 4.7 of \cite{DK}, Defant and Kravitz obtain the following characterization of the connected components of $\FS(\C_n, Y)$. 

\begin{proposition}[\cite{DK}]\label{prop:simplecycles}
Let $Y$ be a simple graph on the vertex set $[n]$. Let $n_1, \dots, n_r$ denote the sizes of the connected components of $\o{Y}$, and let $\nu=\gcd(n_1, \dots, n_r)$. For each toric acyclic orientation $\alphasim\in \Acyc(\o{Y})/\sim$, choose a linear extension $\sigma_{[\alpha]_\sim}$, and let $J_{[\alpha]_\sim}$ be the connected component of $\FS(\C_n, Y)$ containing $[\sigma_{[\alpha]_\sim}]$. Then
\[J_{[\alpha]_\sim}, \dots, (\varphi^*)^{\nu-1}(J_{[\alpha]_\sim}),\]
are pairwise distinct, isomorphic connected components of $\FS(\C_n, Y)$. Moreover,
\[\FS(\C_n, Y)=\bigoplus_{\alphasim\in \Acyc(\o{Y})/\sim} \bigoplus_{k=0}^{\nu-1}(\varphi^*)^k(J_{[\alpha]_\sim}).\]
\end{proposition}

\noindent
We obtain the following generalization of this result.

\begin{theorem}\label{thm:multiplicitycycles}
Let $X$ be a multiplicity graph and suppose $\o{X'}$ has connected components $Z_1, \dots, Z_r$. An orientation $\alphasimeq\in \Acyc(\o{X'})/\!\simeq$ induces orientations $[\alpha^{(1)}]_\simeq, \dots, [\alpha^{(r)}]_\simeq$ on $Z_1, \dots, Z_r$. Let $\pi_{1}, \dots, \pi_{r}$ be the periods of the induced orientations, and let $\delta_\alpha=\gcd(\pi_{1}, \dots, \pi_{r})$. For each toric acyclic orientation $\alphasimeq\in \Acyc(\o{X'})/\!\simeq$, choose a linear extension $\sigma_{[\alpha]_\simeq}:[n]\to V(X')$, and let $H_{[\alpha]_\simeq}$ be the connected component of $\FSm(\C_n, X)$ containing $[\sigma_{[\alpha]_\simeq}]_\equiv$. Then \[H_{[\alpha]_\simeq}, \dots, (\varphi^*)^{\delta_\alpha-1}(H_{[\alpha]_\simeq}),\]
are pairwise distinct, isomorphic connected components of $\FSm(\C_n, X)$.
Moreover,
\[\FSm(\C_n, X)=\bigoplus_{\alphasimeq\in \Acyc(\o{X'})/\!\simeq} \bigoplus_{k=0}^{\delta_\alpha-1}(\varphi^*)^k(H_{[\alpha]_\simeq}).\]
\end{theorem}

\noindent
Before setting up the proof of the Theorem~\cref{thm:multiplicitycycles}, we first present Proposition~\cref{prop:divisibility}. The proof of Proposition~\cref{prop:divisibility} introduces the new element that distinguishes the proof of Theorem~\cref{thm:multiplicitycycles} from the proofs of Theorems~\cref{thm:multiplicitypaths} and \cref{thm:multiplicitycyclesdoubleflip}, and the proofs from \cite{DK}.

\begin{proposition}\label{prop:divisibility}
Let $Z_t$ be one of the connected components of $\o{X'}$ and suppose $Z_t$ has $q$ vertices. Let $\sigma:[q]\to V(Z_t)$ be an arrangement and suppose there exists a sequence $\Sigma$ of $(\C_q, \o{Z_t})$-friendly swaps transforming $\sigma$ into $\rho\circ \sigma\circ \varphit^{a_t}$, for some $\rho\in \Sym_X$ and $a_t\in \Z_{>0}$. If $\alpha=\alpha_{Z_t}(\sigma)$ is the orientation induced by $\sigma$, we have $\pi_{\alpha}|a_t$.\footnote{Proposition~\cref{prop:divisibility} implies that $\pi_\alpha|\pi_\sigma$ for all $\sigma\in \L(\alpha)$, since $\sigma$ is permutation equivalent to $\sigma\circ \varphit^{\pi_\sigma}$ by definition and $\Sigma$ can be taken to be empty. Although this fact is not directly related to the notion of friends-and-strangers graphs, we do not see an easy way to prove it without referencing Proposition~\cref{prop:divisibility}.}
\end{proposition}
\begin{proof}
Since $Z_t$ are connected components of $\o{X'}$, they form a coarser partition of $V(X')$ than the cliques $S_v$ corresponding to vertices $v\in V(X)$. Hence, we may say that $V(Z_t)$ is a union of several cliques $S_v$, for $v$ in some set $X_t\subseteq V(X)$. Hence, we will think of $Z_t$ as the complement of the lift of $X_t$.

The main idea of the proof is to find an orientation $\t{\sigma}$, which differs from $\sigma$ by a sequence of $(\C_q, \o{Z_t})$-friendly swaps, and whose period $\pi_{
\t{\sigma}}$ divides $a_t$. Once this is established, a simple argument will complete the proof. Before explaining how exactly we find such $\t{\sigma}$, let us say a few words about the combinatorial interpretation of the sequence of swaps transforming $\sigma$ into $\rho\circ \sigma\circ \varphit^{a_t}$.

Recall that the arrangement $\sigma$ is a bijection $\sigma:[q]\to V(Z_t)$. Hence, we can think of $\sigma$ as labeling the graph $\C_q$, such that every position $p\in [q]$ on the cycle receives a label $u=\sigma(p)\in Z_t$. Since the condition imposed on $\Sigma$ is invariant under permuting the labels in the same clique $S_v$, we may assume that none of the swaps includes two elements of the same clique. In other words, we may assume that the cyclic order of elements of the same clique $S_v$ is unchanged.

For an arbitrary position $p$, we define the weight of $p$, denoted by $w(p)$, as the number of times the label $\sigma(p)$ moves counterclockwise in a swap minus the number of times it moves clockwise in $\Sigma$. This notion was originally introduced by Defant and Kravitz in the proof of Proposition 4.4 of \cite{DK}. Since every swap involves one label moving clockwise and one label moving counterclockwise, we must have $\sum_{p\in [q]}w(p)=0$. The notion of weight turns out to be very useful, because of its relation to adjacent pairs of vertices in $E(Z_t)$. 

More precisely, suppose that $p_1, p_2\in [q]$ satisfy $w(p_1)\geq w(p_2)+d(p_1, p_2)$, where $d(p_1, p_2)$ denotes the counterclockwise distance from $p_1$ to $p_2$ on $\C_q$. It is not hard to see that the labels $\sigma(p_1), \sigma(p_2)$ must have participated in the same $(\C_q, \o{Z_t})$-friendly swap, implying that $\sigma(p_1)\sigma(p_2)\notin E(Z_t)$. It is natural to say that the positions $p_1$ and $p_2$ have \dfn{intersecting trajectories} if either $w(p_1)\geq w(p_2)+d(p_1, p_2)$ or $w(p_2)\geq w(p_1)+d(p_2, p_1)$ holds.

As stated earlier, the strategy of the proof is to find $\t{\sigma}$ by performing a sequence of $(\C_q, \o{Z_t})$-friendly swaps on $\sigma$. First of all, let us note that $\t\sigma$ can still be transformed into $\t \rho\circ \t\sigma\circ \varphit^{a_t}$, for some $\t\rho\in \Sym_X$, using a sequence $\t \Sigma$ of $(\C_q, \o{Z_t})$-friendly swaps. If $\t\sigma=\sigma\circ \tr{p}{p+1}$, then it is not hard to see that the sequence $\t \Sigma=\tr{p}{p+1} \Sigma \tr{p+a_t}{p+a_t+1}$ satisfies the requirements. Hence, the notion of intersecting trajectories carries over to any $\t\sigma$ obtained from $\sigma$ through a sequence of $(\C_q, \o{Z_t})$-friendly swaps.

Using the notion of the intersecting trajectories, we can finally state precisely and formally our strategy for finding $\t{\sigma}$. Namely, the first step of the proof will be showing that one can transform $\sigma$ into $\t{\sigma}$ using a sequence of $(\C_q, \o{Z_t})$-friendly swaps such that $\t{\sigma}$ satisfies the same conditions as $\sigma$ and such that no two positions have intersecting trajectories in $\t{\sigma}$. Then, we show that for every $\t{\sigma}$ is chosen in this way we have $\pi_{\t{\sigma}}|a_t$.

Before we present the details of how these two steps are accomplished, let us introduce the permutation $f_\sigma:[q]\to [q]$ given by $f_\sigma=\sigmai\circ \rhoi\circ \sigma$. Intuitively, $f_\sigma(p)+a_t$ indicates the position at which the label $\sigma(p)$ ends up at after the sequence of swaps $\Sigma$. When $\sigma$ is clear from the context, or when we are working with only one arrangement, we will write $f$ instead of $f_\sigma$. The permutation $f$ splits $[q]$ into orbits, or cycles, and all elements of an orbit are in the same clique of $\o{Z_t}$. Let us denote the orbit of $p$ by $O_p=\{p, f(p), f^2(p), \dots\}$. Next, we show that all orbits have the same size.

\begin{lemma}\label{lemma:equicardinalorbits}
For any two positions $p_1, p_2\in [q]$, the orbits $O_{p_1}$ and $O_{p_2}$ have the same size.
\end{lemma}
\begin{proof}
Since $Z_t$ is connected, it suffices to show the statement is true when $\sigma(p_1)$ and $\sigma(p_2)$ are adjacent in $Z_t$. Suppose this is indeed the case and let $O_{p_1}=\{p_1, f(p_1), \dots\}$ and $O_{p_2}=\{p_2, f(p_2), \dots\}$ be their orbits. Without loss of generality, we may choose $p_1$ and $p_2$ from their orbits such that no elements of $O_{p_1}$ or $O_{p_2}$ lie in $(p_1, p_2)$, where $(p_1, p_2)$ denotes the open counterclockwise from $p_1$ to $p_2$ on $\C_q$. 

This choice also implies $(O_{p_1}\cup O_{p_2})\cap (f(p_1), f(p_2))=\emptyset$. To see why, suppose there is another element of $O_{p_1}$, say $f^i(p_1)$, in $(f(p_1), f(p_2))$. Since $\sigma(p_1)$ and $\sigma(f^{i-1}(p_1))$ are in the same clique of $\o{Z_t}$, $\sigma(f^{i-1}(p_1))$ is adjacent to $\sigma(p_2)$ in $Z_t$, meaning that the trajectories of $p_2$ and $f^{i-1}(p_1)$ do not intersect. Similarly, by assumption that the labels of the same clique do not change the cyclic order, the trajectories of $p_1$ and $f^{i-1}(p_1)$ do not intersect. Since $f^i(p_1)$ is in the interval $(f(p_1), f(p_2))$, $f^{i-1}(p_1)$ must also be in the interval $(p_1, p_2)$. As this is assumed not the be the case, we conclude $(O_{p_1}\cup O_{p_2})\cap (f(p_1), f(p_2))=\emptyset$. 

This argument inductively extends to show that $(O_{p_1}\cup O_{p_2})\cap (f^i(p_1), f^i(p_2))=\emptyset$, for all $i\geq 0$. In particular, plugging in $i=|O_{p_1}|$ lets us conclude that $f^{|O_{p_1}|}(p_2)=p_2$, meaning $|O_{p_2}|$ divides $|O_{p_1}|$. By symmetry, we also have that $|O_{p_1}|$ divides $|O_{p_2}|$, implying $|O_{p_1}|=|O_{p_2}|$.
\end{proof}
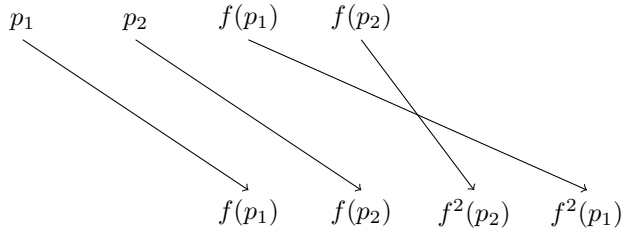
\begin{figure}[h]
    \centering
\begin{tikzpicture}
\coordinate (p1) at (-4.5, 1);
\coordinate (p2) at (-3, 1);
\coordinate (fp1g) at (-1.5, 1);
\coordinate (fp2g) at (0, 1);
\coordinate (fp1d) at (-1.5, -1);
\coordinate (fp2d) at (0, -1);
\coordinate (f2p1d) at (3, -1);
\coordinate (f2p2d) at (1.5, -1);

\draw (p1) node[above] {$p_1$};
\draw (p2) node[above] {$p_2$};
\draw (fp1g) node[above] {$f(p_1)$};
\draw (fp2g) node[above] {$f(p_2)$};
\draw (fp1d) node[below] {$f(p_1)$};
\draw (fp2d) node[below] {$f(p_2)$};
\draw (f2p1d) node[below] {$f^2(p_1)$};
\draw (f2p2d) node[below] {$f^2(p_2)$};

\draw[->] (p1) -- (fp1d);
\draw[->] (p2) -- (fp2d);
\draw[->] (fp1g) -- (f2p1d);
\draw[->] (fp2g) -- (f2p2d);

\end{tikzpicture}
\caption{Graphical representation of the orbits of $f$, along with the trajectories of $p_1, p_2, f(p_1), f(p_2)$. The positions in the top row represent the arrangement $\sigma$ before performing the sequence of $(\C_q, \o{Z_t})$-friendly swaps, while the bottom row represents the positions in the permutation after this sequence is performed. The arrows in the diagram represent schematically how to labels move throughout this process. Note that in this picture, the trajectories of $p_1$ and $p_2$ do not intersect, while the trajectories of $f(p_1)$ and $f(p_2)$ intersect.}
\end{figure}

\begin{lemma}
For any $p\in [q]$ we have $\sum_{k=0}^{|O_p|-1}w(f^k(p))=0$.
\end{lemma}
\begin{proof}
In fact, we will show that for arbitrary $p_1, p_2\in [q]$ with $\sigma(p_1)\sigma(p_2)\in E(Z_t)$ we have 
\begin{equation}\label{eqn:sumofweights}
    \sum_{k=0}^{|O_{p_1}|-1}w(f^k(p))=\sum_{k=0}^{|O_{p_1}|-1}w(f^k(q)).
\end{equation}

\noindent
Since $Z_t$ is connected, relation \cref{eqn:sumofweights} suffices to show that the sums of weights in all orbits are the same. Because the sum of all weights is $0$, the sum of weights in each orbit must also be $0$.

To show \cref{eqn:sumofweights}, we note that since the trajectories of $f^k(p_1)$ and $f^k(p_2)$ do not intersect, we can write $w(f^k(p_1))=w(f^k(p_2))+d(f^k(p_1), f^k(p_2))-d(f^{k+1}(p_1), f^{k+1}(p_2))$. Summing over all $k\in \{0, \dots, |O_{p_1}|-1\}$, the summands $d(f^k(p_1), f^k(p_2))$ telescope and we obtain \cref{eqn:sumofweights}.
\end{proof}

\begin{lemma}\label{lemma:furtherintersections}
Suppose that the trajectories of $p_1$ and $p_2$ intersect. Then, there exists some $k\in \{1, \dots, |O_{p_1}|-1\}$ for which the trajectories of $f^k(p_1)$ and $f^k(p_2)$ also intersect. 
\end{lemma}
\begin{proof}
Since the trajectories of $p_1$ and $p_2$ intersect, without loss of generality we may suppose that $w(p_1)\geq w(p_2)+d(p_1, p_2)$. Hence, equation \cref{eqn:sumofweights} gives 
\[w(f(p_1))+\dots+w(f^{|O_{p_1}|-1}(p_1))+d(p_1, p_2)\leq w(f(p_2))+\dots+w(f^{|O_{p_1}|-1}(p_2)).\]

\noindent
Assuming the trajectories of $f^k(p_1)$ and $f^k(p_2)$ do not intersect for $k\in \{1, \dots, |O_{p_1}|-1\}$, we deduce that $w(f^k(p_1))=w(f^k(p_2))+d(f^k(p_1), f^k(p_2))-d(f^{k+1}(p_1), f^{k+1}(p_2))$. Summing over $k\in \{1, \dots, |O_{p_1}|-1\}$ gives
\[w(f(p_1))+\dots+w(f^{|O_{p_1}|-1}(p_1))= w(f(p_2))+\dots+w(f^{|O_{p_1}|-1}(p_2))+d(f(p_1), f(p_2))-d(p, q).\]

\noindent
Combined with the previous inequality, this implies $d(f(p), f(q))\leq 0$, which is absurd. The conclusion is that there exists another index $k\in \{1, \dots, |O_p|-1\}$ for which the trajectories of $f^k(p)$ and $f^k(q)$ intersect.
\end{proof}

\begin{lemma}\label{lemma:gettingnointersections}
There exists an arrangement $\t\sigma$, differing from $\sigma$ by a sequence of $(\C_q, \o{Z_t})$-friendly swaps, which has no intersecting trajectories.
\end{lemma}
\begin{proof}
Suppose that $\t\sigma$ is the arrangement with the minimal number of intersecting trajectories among all arrangements obtainable from $\sigma$ through a sequence of $(\C_q, \o{Z_t})$-friendly swaps. Under the assumption that $\t\sigma$ has a pair of intersecting trajectories, we will show how to reduce the number of pairs having intersecting trajectories. 

Before presenting the algorithm which accomplishes this, let us briefly discuss the effects of a $(\C_q, \o{Z_t})$-friendly swap on the orbit structure of $f_\sigma$ and intersecting trajectories. Suppose that $\sigma_1=\sigma\circ \tr{p}{p+1}$ is a $(\C_q, \o{Z_t})$-friendly swap, and that the trajectories of $p, p+1$ were intersecting for $\sigma$. Then, these trajectories are not intersecting in $\sigma_1$. Moreover, if the trajectories of $f_{\sigma}^{-1}(p), f_\sigma^{-1}(p+1)$ are intersecting for $\sigma$, they are not intersecting for $\sigma_1$ and vice versa. This shows that the number of intersecting pairs does not increase after performing a $(\C_q, \o{Z_t})$-friendly swap on a pair with intersecting trajectories. Moreover, the number of intersecting pairs decreases if the pair $f_\sigma^{-1}(p), f_\sigma^{-1}(p+1)$ is also intersecting in $\sigma$. These changes are visually presented in Figure~\cref{fig:friendlyswap}.
\begin{figure}[h]
    \centering
\begin{tabular}{c c}
\begin{tikzpicture}
\coordinate (fip1) at (-5, 1);
\coordinate (fip2) at (-3, 1);
\coordinate (p1g) at (-1.5, 1);
\coordinate (p2g) at (-0.5, 1);
\coordinate (p1d) at (-1.5, -1);
\coordinate (p2d) at (-0.5, -1);
\coordinate (fp1d) at (2.5, -1);
\coordinate (fp2d) at (1, -1);

\draw (fip1) node[above] {\substak{u'}{f^{-1}_\sigma(p)}};
\draw (fip2) node[above] {\substak{v'}{f^{-1}_\sigma(p+1)}};
\draw (p1g) node[above] {\substak{u}{p}};
\draw (p2g) node[above] {\substak{v}{p+1}};
\draw (p1d) node[below] {\substak{u'}{p}};
\draw (p2d) node[below] {\substak{v'}{p+1}};
\draw (fp1d) node[below] {\substak{u}{f_\sigma(p)}};
\draw (fp2d) node[below] {\substak{v}{f_\sigma(p+1)}};

\draw[->] (fip1) -- (p1d);
\draw[->] (fip2) -- (p2d);
\draw[->] (p1g) -- (fp1d);
\draw[->] (p2g) -- (fp2d);

\end{tikzpicture}
&
\hspace{-1.5cm}
\begin{tikzpicture}
\coordinate (fip1) at (-5, 1);
\coordinate (fip2) at (-3, 1);
\coordinate (p1g) at (-1.5, 1);
\coordinate (p2g) at (-0.5, 1);
\coordinate (p1d) at (-1.5, -1);
\coordinate (p2d) at (-0.5, -1);
\coordinate (fp1d) at (2.5, -1);
\coordinate (fp2d) at (1, -1);

\draw (fip1) node[above] {\substak{u'}{f^{-1}_\sigma(p)}};
\draw (fip2) node[above] {\substak{v'}{f^{-1}_\sigma(p+1)}};
\draw (p1g) node[above] {\substak{v}{p}};
\draw (p2g) node[above] {\substak{u}{p+1}};
\draw (p1d) node[below] {\substak{v'}{p}};
\draw (p2d) node[below] {\substak{u'}{p+1}};
\draw (fp1d) node[below] {\substak{u}{f_\sigma(p)}};
\draw (fp2d) node[below] {\substak{v}{f_\sigma(p+1)}};

\draw[->] (fip1) -- (p2d);
\draw[->] (fip2) -- (p1d);
\draw[->] (p1g) -- (fp2d);
\draw[->] (p2g) -- (fp1d);

\end{tikzpicture}
\end{tabular}
    
\caption{For a given arrangement $\sigma$, the left diagram represents the trajectories of $\sigma$, while the right one represents the trajectories of $\sigma_1=\sigma \circ \tr{p}{p+1}$. The letters $u, v, u', v'$ represent labels assigned to relevant positions before and after the sequence of swaps is performed, under the convention that $u=\sigma(p), v=\sigma(p+1), u'=\sigma(f^{-1}_\sigma(p)), v'=\sigma(f^{-1}_\sigma(p+1))$.}
\label{fig:friendlyswap}
\end{figure}
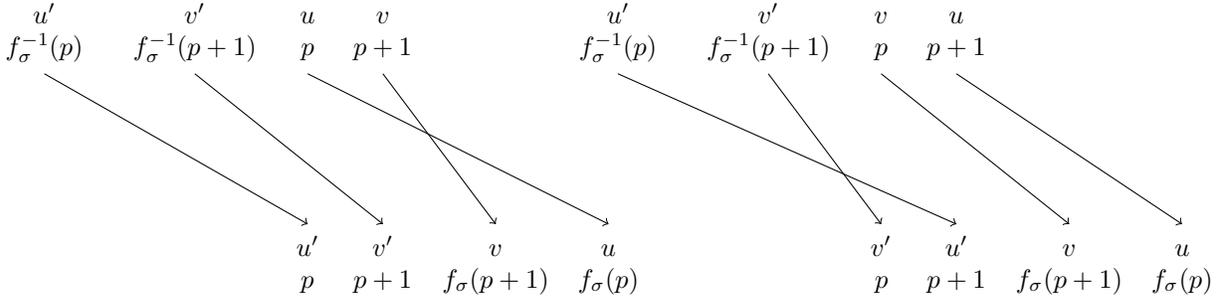

\noindent
Suppose now that $\t \sigma$ contain an intersecting pair of positions, say $p_1, p_2$. By Lemma~\cref{lemma:furtherintersections} the trajectories of $f^k(p)$ and $f^k(q)$ intersect for some $k\in \{1, \dots, |O_p|-1\}$. Let us consider the smallest such $k$. 

The main goal of the sequence of swaps we will construct is to obtain an arrangement $\sigma'$ in which $f^k(p_1)$ and $f^k(p_2)$ are adjacent, and then perform a $(\C_q, \o{Z_t})$-friendly swap on these two positions. In case $k=1$, by the above discussion, this will reduce the intersection number, which is exactly what we want. For $k>1$, this will reduce to the case when the trajectories of $f^{k-1}(p_1)$ and $f^{k-1}(p_2)$ intersect, which can be dealt with by repeating the above procedure.

Let us now explain how we make $f^k(p_1), f^k(p_2)$ adjacent. It is simple to see that if the trajectories of $a$ and $b$ intersect, then for every $c\in (a, b)$ the trajectory of $c$ intersects that of either $a$ or $b$. Hence, for every $x\in (f^k(p_2), f^k(p_1))$, we either have that the trajectory of $w$ intersects the trajectory of $f^k(p_1)$ or the trajectory of $f^k(p_2)$. For $i\in \{1, 2\}$, let $S_i$ be the set of those $x$ whose trajectories intersect the trajectory of $f^k(p_i)$. By applying the above argument again, we see that the trajectory of $s_1\in S_1$ intersects the trajectory of every $s_2\in S_2\cap (s_1, f^k(p_1))$. Hence, every such pair $(s_1, s_2)$ has labels which are adjacent in $\o{Z_t}$. Therefore one may perform a sequence of $(\C_q, \o{Z_t})$-friendly swaps to obtain an arrangement in which all positions which have intersecting trajectories with $f^k(p_2)$ come before those positions which have intersecting trajectories with $f^k(p_1)$, in the interval $(f^k(p_2), f^k(p_1))$. Then, we may perform a sequence of $(\C_q, \o{Z_t})$-friendly swaps which makes $f^k(p_2), f^k(p_1)$ adjacent. As explained above, this suffices to complete the proof.
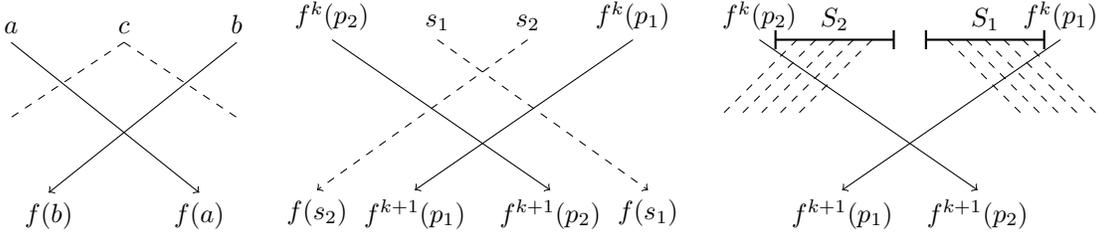
\begin{figure}[h]
    \centering
\begin{tabular}{c c c c}
\begin{tikzpicture}
\coordinate (a) at (-1.5, 1);
\coordinate (b) at (1.5, 1);
\coordinate (fa) at (1, -1);
\coordinate (fb) at (-1, -1);
\coordinate (c) at (0, 1);

\draw (a) node[above] {$a$};
\draw (b) node[above] {$b$};
\draw (fa) node[below] {$f(a)$};
\draw (fb) node[below] {$f(b)$};
\draw (c) node[above] {$c$};

\draw[->] (a) -- (fa);
\draw[->] (b) -- (fb);
\draw[dashed] (c) -- (-1.5, 0);
\draw[dashed] (c) -- (1.5, 0);
\end{tikzpicture}
&
\begin{tikzpicture}
\coordinate (fkp1) at (2, 1);
\coordinate (fkp2) at (-2, 1);
\coordinate (fk+1p1) at (-0.9, -1);
\coordinate (fk+1p2) at (0.9, -1);
\coordinate (s1) at (-0.6, 1);
\coordinate (s2) at (0.6, 1);
\coordinate (fs1) at (2.2, -1);
\coordinate (fs2) at (-2.2, -1);

\draw (fkp1) node[above] {$f^k(p_1)$};
\draw (fkp2) node[above] {$f^k(p_2)$};
\draw (fk+1p1) node[below] {$f^{k+1}(p_1)$};
\draw (fk+1p2) node[below] {$f^{k+1}(p_2)$};
\draw (s1) node[above] {$s_1$};
\draw (s2) node[above] {$s_2$};
\draw (fs1) node[below] {$f(s_1)$};
\draw (fs2) node[below] {$f(s_2)$};

\draw[->] (fkp1) -- (fk+1p1);
\draw[->] (fkp2) -- (fk+1p2);
\draw[dashed, ->] (s1) -- (fs1);
\draw[dashed, ->] (s2) -- (fs2);
\end{tikzpicture}
&
\begin{tikzpicture}
\coordinate (fkp1) at (2, 1);
\coordinate (fkp2) at (-2, 1);
\coordinate (fk+1p1) at (-0.9, -1);
\coordinate (fk+1p2) at (0.9, -1);
\coordinate (s1) at (1, 1);
\coordinate (s2) at (-1, 1);

\draw (fkp1) node[above] {$f^k(p_1)$};
\draw (fkp2) node[above] {$f^k(p_2)$};
\draw (fk+1p1) node[below] {$f^{k+1}(p_1)$};
\draw (fk+1p2) node[below] {$f^{k+1}(p_2)$};
\draw (s1) node[above] {$S_1$};
\draw (s2) node[above] {$S_2$};

\draw[->] (fkp1) -- (fk+1p1);
\draw[->] (fkp2) -- (fk+1p2);
\draw[thick, |-|] (-1.8, 1) -- (-0.2,1);
\draw[thick, |-|] (0.2,1) -- (1.8, 1);
\draw[dashed] (1.5, 1) -- (2.5, 0);
\draw[dashed] (1.25, 1) -- (2.25, 0);
\draw[dashed] (1, 1) -- (2, 0);
\draw[dashed] (0.75, 1) -- (1.75, 0);
\draw[dashed] (0.5, 1) -- (1.5, 0);
\draw[dashed] (-1.5, 1) -- (-2.5, 0);
\draw[dashed] (-1.25, 1) -- (-2.25, 0);
\draw[dashed] (-1, 1) -- (-2, 0);
\draw[dashed] (-0.75, 1) -- (-1.75, 0);
\draw[dashed] (-0.5, 1) -- (-1.5, 0);
\end{tikzpicture}
\end{tabular}
\caption{Schematic representation of the proof of Lemma~\cref{lemma:gettingnointersections}. The left diagram illustrates the fact that if the trajectories of $a$ and $b$ intersect and $c\in (a, b)$, then the trajectory of $c$ intersects the trajectory of either $a$ or $b$. The middle and the right diagram represent steps in transforming $\t\sigma$ into an arrangement in which $f^k(p_1)$ and $f^k(p_2)$ are adjacent.}
\end{figure}
\end{proof}

\begin{lemma}\label{lemma:nointersectionsimpliesdivisibility}
If $\t\sigma$ is an arrangement in which no pair of positions has intersecting trajectories, then $\pi_{\t\sigma}|a_t$.
\end{lemma}
\begin{proof}
We begin by showing that $w(p)$ is the same for all $p\in[q]$. Let $p$ and $p+1$ be two adjacent vertices of $\C_q$. Since their trajectories do not intersect, we have $w(p)\leq w(p+1)$ and $w(p+1)\leq w(p)+(q-1)$. In case $w(p)\neq w(p+1)$, we must have $w(p+1)\geq w(p)+1$. It is not hard to see that, in this case, the trajectory of $f^{-1}(f(p)+1)$ must intersect either the trajectory of $p$ or the trajectory of $p+1$. As $\t\sigma$ has no pairs with intersecting trajectories, this is impossible and therefore $w(p)=w(p+1)$. As this holds for all positions $p\in [q]$, we conclude that $w(p)$ is the same for all $p\in [q]$. Since $\sum_{p\in [q]}w(p)=0$, we have $w(p)=0$ for all $p\in [q]$. 

\begin{figure}[h]
    \centering
\begin{tabular}{c c c c}
\begin{tikzpicture}
\coordinate (p) at (-1, 1);
\coordinate (p+1) at (0, 1);
\coordinate (fp) at (-1, -1);
\coordinate (fp+1) at (2.5, -1);
\coordinate (fp+1') at (0.5, -1);

\draw (p) node[above] {$p$};
\draw (p+1) node[above] {$p+1$};
\draw (fp) node[below] {$f(p)$};
\draw (fp+1) node[below] {$f(p+1)$};
\draw (fp+1') node[below] {$f(p)+1$};

\draw[->] (p) -- (fp);
\draw[->] (p+1) -- (fp+1);
\draw[dashed] (fp+1') -- (-1.5, 0);
\draw[dashed] (fp+1') -- (2.5, 0);
\end{tikzpicture}

\end{tabular}
\caption{Graphical representation of the fact that the trajectory of $f^{-1}(f(p)+1)$ intersects either the trajectory of $p$ or the trajectory of $p+1$ in the proof of Lemma~\cref{lemma:nointersectionsimpliesdivisibility}.}
\end{figure}
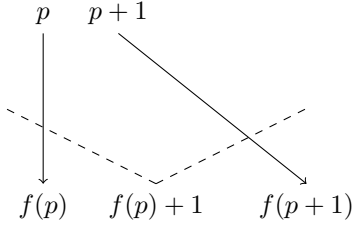

\noindent
This means that $f(p)=p$ for all $p\in [q]$, and hence the empty sequence of $(\C_q, \o{Z_t})$-friendly swap transforms $\t\sigma$ into $\rho\circ\t \sigma\circ \varphit^{a_t}$. In other words, $\t\sigma\equiv \t\sigma\circ \varphi^{a_t}$, implying $\pi_{\t\sigma}|a_t$, which completes the proof of Lemma~\cref{lemma:nointersectionsimpliesdivisibility}.
\end{proof}

\noindent
Combining Lemmas~\cref{lemma:gettingnointersections} and \cref{lemma:nointersectionsimpliesdivisibility} shows that there exists an arrangement $\t \sigma$, differing from $\sigma$ by a sequence of $(C_q, \o{Z_t})$-friendly swaps, and having $\pi_{\t \sigma}|a_t$. Let us now argue why this suffices to conclude $\pi_{\alpha}|a_t$, and assume, towards contradiction, that this is not the case.

We have argued previously that replacing $\sigma$ with any arrangement $\tau$ differing from it by a sequence of $(C_q, \o{Z_t})$-friendly swaps does not change the property that $\tau$ can be transformed into $\rho'\circ \tau\circ \varphit^{a_t}$ through a sequence of $(C_q, \o{Z_t})$-friendly swaps, for some $\rho'\in \Sym_X$. Let $\tau\in \L(\alpha)$ be the arrangement with the minimal period $\pi_\tau=\pi_\alpha$. Since $\tau\circ \varphit^{a_t}=\tau\circ \varphi_q^{a_t-m\pi_\alpha}$ for any $m\geq 0$, we may assume that $1\leq a_t< \pi_\alpha$ and that $\tau$ can be transformed into $\rho'\circ \tau\circ \varphit^{a_t}$ through a sequence of $(C_q, \o{Z_t})$-friendly swaps. However, since $\t \sigma$ differs from $\tau$ by a sequence of $(C_q, \o{Z_t})$-friendly swaps, we conclude that $\t\sigma$ also satisfies the same property as $\tau$ with $a_t<\pi_\alpha$. This presents a contradiction, since $\pi_{\t\sigma}|a_t$ and $\pi_{\t\sigma}\geq \pi_\alpha$ by definition of $\pi_\alpha$. This contradiction completes the proof of Proposition~\cref{prop:divisibility}.
\end{proof}

\noindent
Equipped with Proposition~\cref{prop:divisibility}, we are ready to prove Theorem~\cref{thm:multiplicitycycles}.

\begin{proof}[Proof of Theorem~\cref{thm:multiplicitycycles}.]

We structure the proof in a similar way to the proofs of Theorems~\cref{thm:multiplicitypaths} and \cref{thm:multiplicitycyclesdoubleflip}. However, the analysis of the arising graph $G$ is more involved than earlier and uses Proposition~\cref{prop:divisibility} in an important way.

Let $G$ be the graph arising from $\FS(\C_n, X')$ after adding the set of edges $E_0=\{\sigma\tau:\sigma\equiv\tau\}$. Corollary~\cref{cor:quotientcomponents} states that the connected components of the graph $\FSm(\C_n, X)$ are the images of the connected components of $G$ under the projection $\sigma\mapsto [\sigma]_\equiv$.  Hence, the main focus of the proof is to describe the connected components of $G$. More precisely, if $L_{\alphasimeq}=G|_{\{\sigma:[\sigma]_\equiv\in V(H_{\alphasimeq})\}}$ is the preimage of $H_{\alphasimeq}$ under the discussed projection, we aim to show that
\begin{equation}\label{eqn:Gdecomposition}
    G=\bigoplus_{\alpha\in \Acyc(\o{X'})/\!\simeq} \bigoplus_{i=0}^{\delta_\alpha-1} (\varphi^*)^i L_{\alphasimeq}.
\end{equation}

\noindent
Since the projection operation commutes with the automorphism $\varphi^*$, which is to say that the projection of $(\varphi^*)^iL_{\alphasimeq}$ is $(\varphi^*)^i H_\alphasimeq$, showing \cref{eqn:Gdecomposition} would suffice to prove the second part of the theorem. The fact that $H_{\alphasimeq}, \dots, (\varphi^*)^{\delta_\alpha-1} H_{\alphasimeq}$ are distinct isomorphic connected components follows directly, since $\varphi^*$ is an automorphism of the graph $\FSm(\C_n, X)$. 

Let us begin the proof of \cref{eqn:Gdecomposition} by noting that there are no edges of $G$ out of the set $\L({\alphasimeq})$. To see why, we pick an arbitrary edge $\sigma\tau\in E(G)$ and show $\alpha_{\o{X'}}(\sigma)\simeq \alpha_{\o{X'}}(\tau)$. If $\sigma\tau\in E_0$, we must have $\sigma\equiv \tau$, and hence $\alpha_{\o{X'}}(\sigma)\equiv \alpha_{\o{X'}}(\tau)$. On the other hand, if $\sigma\tau\in E(\FS(\C_n, X'))$, we have $\alpha_{\o{X'}}(\sigma)\approx \alpha_{\o{X'}}(\tau)$, and therefore also $\alpha_{\o{X'}}(\sigma)\sim \alpha_{\o{X'}}(\tau)$. Hence, all edges of $G$ must be fully contained within one of the sets $\L(\alphasimeq)$, which gives the following decomposition of $G$:

\begin{equation}\label{eqn:Gdecomposition2}
    G=\bigoplus_{\alpha\in \Acyc(\o{X'})/\!\simeq} G|_{\L(\alphasimeq)}.
\end{equation}

\noindent
In order to infer \cref{eqn:Gdecomposition} from \cref{eqn:Gdecomposition2}, we need to show that 

\begin{equation}\label{eqn:Gdecomposition3}
G|_{\L(\alphasimeq)}=\bigoplus_{i=0}^{\delta_\alpha-1}(\varphi^*)^i L_{\alphasimeq}.
\end{equation}

\noindent
Since the equivalence relation $\simeq$ is coarser than $\sim$, the equivalence class $\alphasimeq$ can be written as a disjoint union of equivalence classes $[\alpha_1]_\sim, \dots, [\alpha_k]_\sim$. Proposition~\cref{prop:simplecycles} implies that 
\begin{equation}\label{eqn:standarddecomp}
    \FS(\C_n, X')|_{\L(\alphasimeq)}=\bigoplus_{j=1}^k \bigoplus_{i=0}^{\nu-1} (\varphi^*)^i J_{[\alpha_k]_\sim},
\end{equation}
\noindent
where $J_{[\alpha_k]_\sim}$ is the connected component of $\FS(\C_n, X')$ containing an arbitrarily chosen vertex $\sigma_{[\alpha_k]_\sim}\in \L([\alpha_k]_\sim)$. Since $\alpha_1, \dots, \alpha_k$ are all equivalent under $\simeq$, we may choose $\sigma_{[\alpha_1]_\sim}, \dots, \sigma_{[\alpha_k]_\sim}$ to be equivalent under $\equiv$. We now present two lemmas, which show that there exists an edge of $E_0$ between components $(\varphi^*)^i J_{[\alpha_l]_\sim}$ and $(\varphi^*)^j J_{[\alpha_m]_\sim}$ if and only if $\delta_\alpha|i-j$.

\begin{lemma}\label{lemma:divisibility1}
If $\delta_\alpha$ divides $i-j$, then there exists an edge of $E_0$ between components $(\varphi^*)^i J_{[\alpha_l]_\sim}$ and $(\varphi^*)^j J_{[\alpha_m]_\sim}$. 
\end{lemma}
\begin{proof}
Let us first consider the case $l=m$, as the case $l\neq m$ can be easily deduced from it.  The main idea is to show that, for every arrangement $\sigma\in (\varphi^*)^i J_{[\alpha_l]_\sim}$, there exists an edge of $E_0$ between $\sigma$ and a vertex of $(\varphi^*)^{i+\pi_t}J_{[\alpha_l]_\sim}$, where $\pi_t$ is the period of an induced orientation $\alpha^{(t)}$ on $Z_t$.

Let $\beta$ be the orientation induced by $\sigma$ on $\o{X'}$, and let $\beta^{(t)}$ be its restriction to $Z_t$. Because $\sigma\in  \L(\alphasimeq)$, and consequently $\beta^{(t)}\simeq \alpha^{(t)}$, the period of $\beta^{(t)}$ must be $\pi_t$. 

Since the period of $\beta^{(t)}$ is $\pi_t$, there exists a sequence of $\pi_t$ positive single flips that transforms $\beta^{(t)}$ into ${\beta^{(t)}_1}$ which is permutation equivalent to $\beta^{(t)}$. Since there are no edges of $\o{X'}$ leaving $Z_t$, the same sequence of $\pi_t$ flips can be performed on $\beta$ to get a orientation $\beta_1$ with $\beta_1\equiv \beta$. In particular, this means that there exists a linear extension $\sigma_1\in \L(\beta_1)$ for which $\sigma\equiv\sigma_1$. The goal is now to show $\sigma_1\in (\varphi^*)^{i+\pi_t} J_{[\alpha_l]_\sim}$.

Let $\beta_2$ be the orientation induced on $\o{X'}$ by $ \sigma\circ\varphi^{\pi_t}$. Note that $\beta_2$ can also be obtained from $\beta$ through a sequence of $\pi_t$ positive single flips. Using Corollary~\cref{cor:doubleflips}, this implies $\beta_1\approx \beta_2$.

Since $\sigma\circ\varphi^{\pi_t}\in (\varphi^*)^{\pi_t+i} J_{[\alpha_l]_\sim}$, Theorem 4.1 of \cite{DK} shows that the vertex $\sigma_1$ also lies in $(\varphi^*)^{\pi_t+i} J_{[\alpha_l]_\sim}$. Since $\sigma_1\equiv\sigma$, we see there exists an edge of $E_0$ between $\sigma$ and a vertex of $(\varphi^*)^{\pi_t+i} J_{[\alpha_l]_\sim}$.

By repeating the above argument and using the transitivity of the relation $\equiv$, we conclude that there exists an edge of $E_0$ between $\sigma$ and a vertex of $(\varphi^*)^{i+\sum_{t=1}^r \lambda_t \pi_t} J_{[\alpha_l]_\sim}$, where $\lambda_1, \dots, \lambda_r$ are arbitrary integer coefficients. Using Bezout's theorem, we conclude there exist coefficients $\lambda_1, \dots, \lambda_t$ for which $\sum_{t=1}^r \lambda_t\pi_t=u\gcd(\pi_1, \dots, \pi_r)=u\delta_\alpha$, where $u$ is an arbitrary integer. Hence, if $\delta_\alpha$ divides $i-j$, we have an edge between components every vertex of $(\varphi^*)^{i} J_{[\alpha_l]_\sim}$ and $(\varphi^*)^{j} J_{[\alpha_l]_\sim}$. 
Since $\sigma_{[\alpha_1]_\sim}, \dots, \sigma_{[\alpha_k]_\sim}$ are all permutation equivalent, we have an edge of $E_0$ between $(\varphi^*)^{i} J_{[\alpha_l]_\sim}$ and $(\varphi^*)^{i} J_{[\alpha_m]_\sim}$, for all $l, m\in [k]$. As before, since every vertex of $(\varphi^*)^{i} J_{[\alpha_m]_\sim}$ has an edge of $E_0$ to $(\varphi^*)^{j} J_{[\alpha_m]_\sim}$, when $\delta_\alpha|i-j$, by transitivity of $\equiv$ we conclude that there exists an edge of $E_0$ between $(\varphi^*)^{i} J_{[\alpha_l]_\sim}$ and $(\varphi^*)^{j} J_{[\alpha_m]_\sim}$ whenever $\delta_\alpha$ divides $i-j$. This completes the proof of Lemma~\cref{lemma:divisibility1}. 
\end{proof}

\begin{lemma}\label{lemma:divisibility2}
If there exists an edge of $E_0$ between components $(\varphi^*)^i J_{[\alpha_l]_\sim}$ and $(\varphi^*)^j J_{[\alpha_m]_\sim}$, then $\delta_\alpha$ divides $i-j$. 
\end{lemma}
\begin{proof}[Proof of Lemma~\cref{lemma:divisibility2}.] 
As in the proof of Lemma~\cref{lemma:divisibility1}, we begin by focusing on the case $l=m$, and then reduce the general case to this one. Our argument has two main steps -- in the first step, we separate the problem into smaller subproblems, by restricting it to one of the connected components $Z_t$. Then, we apply Proposition~\cref{prop:divisibility} to solve each of the subproblems.

Suppose there exists an edge of $E_0$ between $\sigma\in \varphi^{i}J_{[\alpha_l]_\sim}$ and $\rho\circ\sigma\in \varphi^{j}J_{[\alpha_l]_\sim}$, for some $\rho\in \Sym_X$. In other words, there exists a sequence of $a=j-i$ positive single flips and $b$ double flips that transforms $\beta=\alpha_{\o{X'}}(\sigma)$ into a permutation equivalent orientation $\rho(\beta)=\alpha_{\o{X'}}(\rho\circ \sigma)$.

The first step is to associate every flip to the component $Z_t$ in which it is performed. More precisely, if the sequence of flips $\Sigma$ transforming $\sigma$ into $\rho\circ \sigma$ consists of single flips at vertices $u_1, \dots, u_a$, and double flips at pairs $(v_1, w_1), \dots, (v_b, w_b)$, we consider three different types of flips: single flips, double flips with both $v_i$ and $w_i$ in the same connected component $Z_t$, and double flips with $v_i$ and $w_i$ in different connected components. We split the flips of $\Sigma$ into $r$ subsequences of flips, denoted  $\Sigma_1, \dots, \Sigma_r$, so that $\Sigma_t$ contains only flips at vertices of $Z_t$. Hence, the first two types of flips are simply associated to the component they are performed in. For each flip of the third type $(v_i, w_i)$, we split it into a positive single flip at $v_i$ and a negative single flip at $w_i$, which are then placed into subsequences corresponding to their respective components. The result of this process is splitting $\Sigma$ into $\Sigma_1, \dots, \Sigma_r$, where $\Sigma_t$ is a legal sequence of single and double flips performed on the restriction $\beta|_{Z_t}$. Moreover, the result of applying $\Sigma_t$ to $\betat=\beta|_{Z_t}$ is precisely the orientation $\rho(\beta)|_{Z_t}$. 

Let $a_{t, +}$ ($a_{t, -}$, resp.) denote the number of positive (negative, resp.) single flips in $\Sigma_t$, and let $a_t=a_{t, +}-a_{t, -}$ denote their difference. Because the number of positive flips in $\Sigma$ is $a$, and since splitting every double flip amounts to adding one positive and one negative single flip, we conclude that $a_1+\dots+a_r=a$. The main goal now is to show that $\pi_t|a_t$, which will immediately imply $\delta_\alpha|a$. 

Let us now rephrase the problem in the setting of a single connected component $Z_t$. As discussed earlier, by Corollary 4.2 of \cite{DK}, $\Sigma_t$ is equivalent to a sequence of $a_t$ positive flips (or $-a_t$ negative flips, if $a_t<0$). The problem now simplifies to showing that, if $\betat, \rho(\betat)\in \Acyc(Z_t)$ are permutation equivalent orientations and, and $\rho(\betat)$ is obtained from $\betat$ after a sequence of $a_t$ positive single flips and some double flips, we must have $\pi_t|a_t$. This is precisely what Proposition~\cref{prop:divisibility} states, albeit in slightly different terminology.

If we denote $|V(Z_t)|$ by $q$, Corollary 4.2 of \cite{DK} shows that $\rho(\betat)$ must be double flip equivalent to the acyclic orientation induced by $\sigmat\circ {\varphit}^{a_t}$, where we think of the arrangement $\sigmat$ as mapping $\C_q$ to $Z_t$. In other words, the orientations induced by $\sigmat$ and $\rhoi\circ \sigmat\circ \varphit^{a_t}$ are double flip equivalent. Proposition~\cref{prop:simplecycles} implies that these two arrangement differ by a sequence of $(\C_q, \o{Z_t})$-friendly swaps. In this context, Proposition~\cref{prop:divisibility} shows that $\pi_{t}|a_t$, which is exactly what we need to complete the proof in the case $l=m$.

Suppose now that there exists an edge of $E_0$ between vertices $\sigma\in (\varphi^*)^i J_{[\alpha_l]_\sim}$ and $\tau\in (\varphi^*)^j J_{[\alpha_m]_\sim}$. Since $\sigma_{[\alpha_l]_\sim}\equiv \sigma_{[\alpha_m]_\sim}$, we know that there exists an edge from every vertex of $\rho\circ \tau\in (\varphi^*)^j J_{[\alpha_l]_\sim}$ to a vertex of $\tau \in (\varphi^*)^j J_{[\alpha_m]_\sim}$. Hence, by transitivity of $\equiv$, we conclude there exists an edge of $E_0$ between $\sigma\in (\varphi^*)^i J_{[\alpha_l]_\sim}$ and $\rho\circ \in (\varphi^*)^j J_{[\alpha_l]_\sim}$, which is possible only if $\delta_\alpha|j-i$, as desired.
\end{proof}

\noindent
Combining Lemmas~\cref{lemma:divisibility1} and \cref{lemma:divisibility2} with \cref{eqn:standarddecomp}, we conclude that $G|_{\L(\alphasimeq)}$ contains exactly $\delta_\alpha$ connected components, with the vertex sets corresponding to $\bigcup_{j=1}^k\bigcup_{i\equiv i_0\mod \delta_\alpha} V\left((\varphi^*)^iJ_{[\alpha_k]_\sim}\right)$. Moreover, these $\delta_\alpha$ connected components differ only by an application of $\varphi^*$. 

Since $L_{\alphasimeq}$ is one of these $\delta_\alpha$ connected components, decomposition \cref{eqn:Gdecomposition3} now becomes obvious. As discussed above, this suffices to infer \cref{eqn:Gdecomposition}, and to complete the proof of Theorem~\cref{thm:multiplicitycycles}.
\end{proof}

\begin{corollary}
The graph $\FSm(\C_n, X)$ is connected if and only if the complement of the lift of $X$ is a forest of trees of coprime sizes.
\end{corollary}
\begin{proof}
Note that if $X'$ is the complement of a forest of trees of coprime sizes, by the Corollary 4.14 of \cite{DK} we have that $\FS(\C_n, X')$ is connected. Hence, Corollary \cref{cor:lifting} implies that $\FSm(\C_n, X)$ must be connected. Establishing the other direction is somewhat trickier, and requires the machinery developed in this section. 

By Theorem~\cref{thm:multiplicitycycles}, if the graph $\FSm(\C_n, X)$ has only one connected component, we must have $|\Acyc(\o{X'})|=1$ and $\delta_\alpha=1$ for the unique element $\alphasimeq\in \Acyc(\o{X'})/\! \simeq$. Unlike the proof of Corollary 4.14 in \cite{DK}, which uses a systematic way to count the number of toric acyclic orientations of a graph using its Tutte polynomial, we present a direct and straightforward argument characterizing all graphs with a single orientation up to $\simeq$ equivalence.

Let us begin by showing the complement of $X$ has no cycles. If there was a cycle in $\o{X}$, consisting of vertices $v_1, \dots, v_k$ in this order, consider the orientation $\alpha\in \Acyc(\o{X'})$ which directs all edges from $u\in S_{v_i}$ to $w\in S_{v_j}$ whenever $i<j$. If we imagine $S_{v_1}, \dots, S_{v_k}$ are placed in the plane in the counterclockwise order, it is not hard to see that every cycle of $\o{X'}$, having one vertex in each of the cliques $S_{v_i}$, has exactly one edge oriented clockwise. For each fixed cycle of this type, it is not hard to see that even after several flips are performed, there is still exactly one edge oriented clockwise. We have thus shown that this property is maintained for every orientation $\alpha'$ obtained from $\alpha$ through a sequence of flips. However, there exist orientations of $\o{X'}$ which do not have this property, meaning that $\Acyc(\o{X'})/\!\simeq$ has more than one element, which is impossible. An example of such orientations is an orientation in which all edges of $\alpha$ are reversed.

Since $\o{X}$ has no cycles, it is a forest. Suppose $v\in X$ is a vertex of capacity $c\geq 2$. We will show that it can be adjacent to a single vertex of capacity $1$. For every vertex $u$ adjacent to $S_v$, we let $p(u, \alpha)$ be the number of edges directed from $u$ to $S_v$ under $\alpha$. We note that throughout a sequence of flips the quantity $p(u_1, \alpha)-p(u_2, \alpha)$ remains constant modulo $c$, for any vertices $u_1, u_2$ in the neighborhood of $S_v$. If $S_v$ is adjacent to at least two vertices, we consider the orientation $\alpha$ in which all edges incident to $S_v$ are oriented away from $S_v$ and an orientation $\alpha'$ in which exactly one edge incident to $S_v$ is reversed. By the above remark, $\alpha$ and $\alpha'$ do not differ by a sequence of flips as long as there are at least $2$ neighbors of $S_v$ in $\o{X'}$. Again, this is impossible since $\Acyc(\o{X'})/\!\simeq$ is assumed to have only one equivalence class.  

We conclude $\o{X'}$ is a forest, consisting of trees $T_1, \dots, T_m$. Since $\FSm(\C_n, X)$ is connected, every $\alpha\in \Acyc(\o{X'})/\!\simeq$ has $\delta_\alpha=1$. Recall that $\delta_\alpha$ is defined as the greatest common divisor of periods of orientation $\alpha^{(i)}$, where $\alpha^{(i)}$ is the restriction of $\alpha$ to $T_i$. We will show now that the period of $\alpha^{(i)}$ is $|T_i|$. Since, up to toric equivalence, there is only one acyclic orientation of $\o{X'}$, we may choose a root $u_i$ in every tree $T_i$ and assume that $\alpha^{(i)}$ directs the edges of $T_i$ away from the root. For this specific orientation, it is clear that at least $|T_i|$ positive flips are needed in order to transform it back into $\alpha$. Hence, the period of $\alpha^{(i)}$ is $|T_i|$. Since $\delta_\alpha=1$, this implies $\gcd(|T_1|, \dots, |T_m|)=1$, which completes the proof.
\end{proof}

\section{Further research directions}\label{sec:furtherresearch}

\noindent
As mentioned in the introduction, Theorem~\cref{thm:randombipartiteconnectivity} determines the correct dependence, up to a factor of $n^{o(1)}$, of the threshold probability $p_{\rm bip}=p_{\rm bip}(n)$ for which $\FS(X, Y)$ has exactly two connected components, when $X, Y\sim \G(K_{n, n}, p)$. It would be interesting to close the gap, in both bipartite and non-bipartite settings, and determine the exact threshold.

If one believes that the main obstruction to the connectivity of randomly generated friends-and-strangers graphs are isolated vertices (as in the case of the usual random graphs), then one would conjecture that the threshold probabilities $p_{\rm gen}$ and $p_{\rm bip}$ for the connectivity of friends-and-strangers graphs would be the same as for containing isolated vertices, i.e. $p_{\rm gen}, p_{\rm bip}=\left(\frac{\log n}{n}\right)^{1/2}$ (cf. Remark~\cref{rmk:thresholdisolatedvertices}). Hence, we make the following conjecture, which would determine the threshold probabilities precisely.

\begin{conjecture}
For $p=\omega\left(\frac{\log^{1/2} n}{n^{1/2}}\right)$ and large $n$ we have the following two statements:
\begin{itemize}
    \item if $X, Y\sim \G(n, p)$, then $\FS(X, Y)$ is connected with high probability.
    \item if $X, Y\sim \G(K_{n, n}, p)$, then $\FS(X, Y)$ has exactly two connected components with high probability.
\end{itemize}
\end{conjecture}

\noindent
In the context of multiplicity graphs, there exists a further generalization of this concept which allows for both graphs $X$ and $Y$ to have multiplicities. More precisely, if $X$, $Y$ are multiplicity graphs with lists of multiplicities $c^{(X)}, c^{(Y)}$, having the same total capacity, then we can define the double multiplicity friends-and-strangers graph $\FSmm(X, Y)$ whose vertices are arrangements of labels represented by the vertices of $X$ which are being placed onto the vertices of $Y$, such that the label $u\in V(X)$ appears exactly $c^{(X)}_{u}$ times and such that exactly $c^{(Y)}_{v}$ labels are placed onto every vertex of $v\in V(Y)$. Many of the questions addressed in this paper can be posed for double-multiplicity friends-and-strangers graphs. For example, since this paper exactly characterizes all multiplicity lists for which $\FSm(X, \S_n)$ and $\FSm(\S_n, X)$ are connected, it is natural to ask for an characterizations of all multiplicity lists $c^{(X)}, c^{(\S_n)}$ which make $\FSmm(X, \S_n)$ connected. In particular, we believe that the following characterization can be derived.

\begin{conjecture}
Let $X$ be a connected simple graph, and let $c^{(X)}, c^{(\S_n)}$ be two multiplicity lists, in which the center of $\S_n$ receives multiplicity $k=c^{(\S_n)}_{\es}$. Then $\FSmm(X, \S_n)$ is connected if and only if $X$ contains no $k$-bridge in which all vertices have multiplicity $1$.
\end{conjecture}

\noindent
Finally, this paper mainly investigates connectivity of various models of friends-and-strangers graphs. However, one can still ask questions about higher connectivity of $\FS(X, Y)$ and $\FSm(X, Y)$. For example, one can pose the following question in case $X, Y$ are random graphs. 

\begin{question}
Let $p(n)=n^{-1/2+o(1)}$ and let $X, Y$ be random graphs from $\G(n, p)$. For which values of $k$ is the graph $\FS(X, Y)$ is $k$-connected with high probability?
\end{question}

\section*{Acknowledgements}
\noindent
This research was conducted at the University of Minnesota Duluth REU, with generous support from Jane Street Capital, the NSA (grant number H98230-22-1-0015), and Princeton University. I would like to thank Joe Gallian for organizing this REU program, Noah Kravitz and Colin Defant for introducing me to this problem, proposing many new research directions, and providing useful feedback on early drafts of this paper. Further, I would like to thank Kiril Bangachev for carefully reading the paper and pointing out several mistakes in its early drafts, as well as Maya Sankar and Mitchell Lee for helpful discussions about the problem, reading drafts of this paper, and generously sharing their ideas with me. Finally, I would like to thank an anonymous reviewer for his helpful comments and pointing out that reference \cite{BJS} may lead to the proof of Theorem~\cref{thm:lowerboundstrengthening}.

\end{document}